\documentclass[11pt]{report}

\usepackage{algorithmic}
\usepackage{arydshln}
\usepackage{natbib}
\usepackage{times}
\usepackage{bm}
\usepackage{verbatim}
\usepackage{stmaryrd}
\usepackage{lineno}
\usepackage{amsmath}
\usepackage{bbm}
\usepackage{enumerate}
\usepackage{amsthm}
\usepackage{comment}
\usepackage{graphicx, graphics}
\usepackage{multirow}
\usepackage{subfigure}
\usepackage{amssymb,uwthesis}

\usepackage{indentfirst}
\usepackage{extarrows}
\usepackage[colorlinks]{hyperref}  
\usepackage{bbm}
 \usepackage{color}
\usepackage{mathrsfs}

\makeatletter

\newcommand{\Rmnum}[1]{\expandafter\@slowromancap\romannumeral #1@}
\makeatother
\newtheorem{proposition}[theorem]{Proposition}
\numberwithin{equation}{section}

\newcommand{\be}{\begin{equation}}
\newcommand{\ee}{\end{equation}}

\providecommand{\abs}[1]{\vert#1\vert}
\providecommand{\norm}[1]{\Vert#1\Vert}

\newcommand{\fl}[1]{\lfloor{#1}\rfloor} 

\newcommand{\supp}{\mathop{\mathrm{supp}}}
\newcommand{\RNum}[1]{\uppercase\expandafter{\romannumeral #1\relax}}

\def\wm{\mu_1}

\def\nd{\bar{\varphi}} 
\def\fluc{\overline{h}} 
\def\cov0{V_0} 
\def\covnu{V_\nu} 
\def\covg{V} 
\def\LH{H} 
\def\STail{M}
\def\i{\imath} 
\def\consta{c_0}

\def\cF{\mathcal{F}}

\def\cM{\mathcal{M}}
\def\cN{\mathcal{N}}

\def\cI{\mathcal{I}}

\def\bE{\mathbb{E}}
\def\bN{\mathbb{N}}
\def\bP{\mathbb{P}}

\def\bR{\mathbb{R}}
\def\bZ{\mathbb{Z}}
\def\bC{\mathbb{C}}

\def\ind{\mathbf{1}}
 
\def\h{h}

\def\mE{\mathbf{E}}   

\def\mP{\mathbf{P}}
\def\m1{\mathbf{1}}

 \def\Vvv{{\rm\mathbb{V}ar}}  \def\Cvv{{\rm\mathbb{C}ov}} 
 
 \def\cJ{\mathcal{J}}  
 
\allowdisplaybreaks[1]

\begin{document}
\title{Discrete Time Harness Processes}
\author{Yun Zhai}
\oraldate{June 4, 2015}
\profA{Timo Sepp\"al\"ainen, Co-Advisor, Professor, Mathematics}
\profB{Brian S. Yandell, Co-Advisor, Professor, Statistics}
\profC{Benedek Valk\'{o}, Associate Professor, Mathematics}
\profD{David F. Anderson, Associate Professor, Mathematics}
\profE{Philip Matchett Wood, Assistant Professor, Mathematics}

\degree{Doctor of Philosophy}
\dept{Statistics}
\thesistype{dissertation}
\beforepreface
\prefacesection{Abstract}
We study the invariant measures and fluctuation limits of discrete-time harness processes in one spatial dimension. We construct one essential ergodic (under spatial shifts) invariant measure of the increment process derived from harness process, and all other ergodic invariant measures can be obtained by adding constants. We also show that the weak limit of the one dimensional height fluctuations starting from the increments under several translation-invariant ergodic measures will obey Edwards-Wilkinson equation, and the finite-dimensional marginal convergence can be extended to a process level convergence. 

\prefacesection{Acknowledgements}
The finalization of my Ph.D.\ dissertation indicates a five-year odyssey full of unforgettable memories has nearly come to an end. I remember when I landed in the United States the very first time, five years was like an infinite time to me. Now putting the first letter in my thesis is just like a yesterday's thing. When I look back upon this wildest adventure in my life, I see obstacles of all kinds, heavy stress from study and research. I am glad that those challenges never destroyed me but rather helped me to mentally and techically prepare for the future.\par         

Here I would like to express the deepest appreciation to all those who have helped me during my Ph.D. study. Without their support, I could not have achieved where I am today. First, I want to give a special and heartfelt gratitude to my co-advisors Professor Timo Sepp{\"a}l{\"a}inen and Professor Brian Yandell, who are the most knowledgable, considerate and patient Ph.D.\ advisors I have ever seen.\par

As an advisor, Timo has been constantly happy to discuss whenever I get stuck in my research. He always provides me with good and accurate insights, which not just help me to solve the obstacles but also point out a clear direction of my research. Timo has also put a great deal of effort in assisting me to polish my dissertation. He read every word in each version and gave tons of useful suggestions on mathematics, organization, latex tips and even grammar. Aside from research, when I was teaching assistant the very first time, Timo took time to observe my class and taught me how to teach. He himself is a great teacher. I once took his large deviation course. I must say his lessons are not just lessons but artworks. With Timo, I have learned not just the beauty of probability theory or how to do research but also the conscientious attitude toward every task in front of us. For this, I can not thank him enough.\par

Brian, on the other hand, has supported me on many administrative matters. He is the one who helped me to build a solid bridge to the probability group in mathematics department and find a way in doing probability research ``legally'' as a statistics Ph.D.\ student. In addition, he has provided me with excellent guidance on Ph.D.\ regulations like choosing minor, credits requirement etc. In order to help me pass the preliminary exam smoothly, Brian also managed to oraganize a dry run for me to practice. I would like to thank Professor Brian Yandell for his consistent effort in making my Ph.D. life comfortable and well-organized.\par

I am very grateful to Professor Benedek Valk\'{o}, Professor David Anderson and Professor Philip Wood for being my thesis committee members. They brought up many interesting questions and insightful ideas which offered different aspects in looking at the problems and inspired my further thinking.\par      

Last but not least, I would like to thank my parents for their unconditional love, understanding and support. I would like to thank Jing Kong for listening all my complaints, cheering me up, cooking delicious meals and delighting my Ph.D.\ life. A special thanks goes to my friend Jin Qian who has shared his wisdom in any linear algebra problems I have.\par  
\prefacesection{Notation and Symbols}
Here is a collection of some notations we use throughout the thesis.
\begin{itemize}
\item{\makebox[4cm][l]{$\bZ$}} is the set of all integer numbers.
\item{\makebox[4cm][l]{$\bZ^+$}} is the set of all nonnegative integer numbers.
\item{\makebox[4cm][l]{$\bN$}} is the set of all positive integer numbers.
\item{\makebox[4cm][l]{$\bZ^d$}} is the d-dimensional integer lattice.
\item{\makebox[4cm][l]{$\bR$}} is the set of all real numbers.
\item{\makebox[4cm][l]{$\bR^+$}} is the set of all nonnegative real numbers.
\item{\makebox[4cm][l]{$\abs{x}$}} is either the absolute value if $x$ is a number or the Euclidean norm if 
\item[]{\makebox[4cm][l]{}} $x$ is a vector.
\item{\makebox[4cm][l]{$\imath$}} is the imaginary unit which equals to $\sqrt{-1}$.
\item{\makebox[4cm][l]{$C$}} is a positive finite constant, the value of which may vary from line to \item[]{\makebox[4cm][l]{}} line.
\item{\makebox[4cm][l]{$X\sim \mu$}} means the random variable $X$ has distribution $\mu$.
\item{\makebox[4cm][l]{$\sigma(X_1,X_2,\ldots,X_n)$}} is the $\sigma$-algebra generated by random variables $X_1, X_2,\ldots, X_n$.
\item{\makebox[4cm][l]{$\bP$}} is the generic probability function.
\item{\makebox[4cm][l]{$\bE$}} is the expectation under probability measure $\bP$.
\item{\makebox[4cm][l]{$\mP,\mE$}} are the probability measure and expectation of the random walk $X_t^i$ 
\item[]{\makebox[4cm][l]{}} coming from the dual representation of the harness processes.
\item{\makebox[4cm][l]{$\bE^\nu$}} is the expectation under initial distribution $\nu$. 
\end{itemize}



\afterpreface

\chapter{Introduction}\label{intro}
In statistical mechanics, for $1+1$ dimensional surface growth models, it is believed that even though the microscopic evolutions of different models may be different in general, macroscopic behaviors are often similar and usually can be categorized into two classes based on the macroscopic flux: the Edwards-Wilkinson (EW) and the Kardar-Parisi-Zhang (KPZ) universality classes (more details can be found in \cite{20}). A central model in the KPZ class is the well-known KPZ equation: 
\be h_t=vh_{xx}+\frac{1}{2}\lambda(h_x)^2+\sqrt{D}\dot{W},\ee
where $\dot{W}$ is the space-time white noise. It is predicted but not completely proved that the order of universal height fluctuation is $t^{1/3}$ where $t$ is the time parameter (see \cite{21} for a survey).\par

In the EW universality class, on the other hand, the limit of the height fluctuation can be described as the solution of the stochastic heat equation with additive noise (often called Edwards-Wilkinson (EW) equation (see \cite{31})):
\be Z_t=vZ_{xx}+\sqrt{D}\dot{W}.\label{i0}\ee 
And the order of macroscopic height fluctuation (scaling time and space by some functions of $n$) is expected to be $n^{1/4}$. This conjecture is supported by past work in independent random walks (\cite{36} and \cite{11}), independent random walks in static and dynamical random environment (RWRE) (\cite{37} and \cite{38}), random average process (RAP) (\cite{33}) and a recent model under continuous space and time setting from the Howitt-Warren flows (\cite{39}). In this paper, we consider a specific surface growth model called harness process, the one space dimensional version of which obeys EW universality.\par 

Harness processes were first named and studied by J. M. Hammersley around 1956 when he was looking at a problem on long-range misorientation in the crystalline structure of metals (see \cite{22}). Later on, \cite{23} has investigated the continuous-time harness processes where the weight vector is symmetric, unnormalized and infinite while the random noises are normally distributed. He has proved the convergence to the equilibrium state from the initial configuration under various conditions. A few years later, \cite{17} generalized his results for asymmetric but finite weight vector and non-gaussian random noises, and proved both the existence and uniqueness of the translation-invariant equilibrium state. Interestingly, he has also pointed out that if the weight is normalized, in one-dimensional case, the order of the height fluctuations is exactly $t^{1/4}$, but the order decreases to $(\log t)^{1/2}$ in two-dimensional case, and the height fluctuations become bounded for higher dimensions.\par

Under the same continuous time setting, \cite{3} have introduced the random walk representation of the harness processes, through which they constructed an invariant measure as the limit of the process starting from the flat configuration. They have shown that with Gaussian noises and finite support assumption on the transition kernel, for $d\ge 3$, the invariant measures of harness processes are Gaussian Gibbs fields (also called harmonic crystals), which has been studied in \cite{19} and \cite{18}. For lower dimensions ($d=1,2$), the invariant measure for the process itself on the entire $\bZ^d$ lattice may not exist in general, but still they have found the stationary measure for the process ``pinned at the origin'' ($h_t(0)\equiv 0$) or ``viewed from the height at the origin'' ($h_t(\cdot)-h_t(0)$). \cite{16} studied the influence of the tail distribution of the noises on the convergence of the harness process under the discrete time setting. He also gave the connection between the decay rates of the noise distribution and the limit distribution.\par 

In this thesis, we study the discrete-time version of harness process in one spatial dimension, which has some connections with independent random walks model in \cite{11} and one dimensional RAP model in \cite{33}. Chapter \ref{main-results} will first give a detailed description of the model (Section \ref{model-description}) and then discuss the main results ( Section \ref{invariant-measure-results} and Section \ref{fluctuation-limits-results}). Section \ref{invariant-measure-results} focuses on finding the invariant measure for the increment process derived from the harness model. We appeal to \cite{3}, provide an invariant measure as the distribution of the limit of an $L^2$-martingale and show that it is indeed the unique ergodic (spatially speaking) measure with finite first moment. Unlike the product form invariant distributions stated in \cite{11} and \cite{33}, the invariant distribution for the increment process in our case does have non-zero correlations (except in some special cases).  Section \ref{invariant-measure-results} will provide asymptotic results for the scaled height fluctuations. We will show that the fluctuation is subdiffusive ($O(n^{1/4})$), and the scaled hight fluctuations starting from i.i.d., invariantly distributed and strongly mixing initial increments will converge to two-parameter Gaussian processes in the sense of convergence of finite-dimensional distributions. More interestingly, the time marginal of the limit process in the second case is a fractional Brownian motion with Hurst parameter $1/4$. In addition, the process-level tightness of the convergence will be achieved. Chapter \ref{proofs} will cover all the proofs. Appendix \ref{potential-kernel} will discuss some useful properties of the potential kernel of one dimensional recurrent random walks. Appendix \ref{lclt} provides a proof of Local Central Limit Theorem (LCLT) and several applications.\par

\chapter{The harness process and the main results}\label{main-results}
\section{The model}\label{model-description}
For fixed dimension $d\in\bN$, the harness processes is a collection $\{\h_t: t\in\bZ^+\}$ where each $\h_t$ is a real-valued random height function on $\bZ^d$, the evolution of which obeys the following rule, 
\be
\h_{t+1}(i)=\sum_{k\in\bZ^d}w(k)\h_t(i+k)+\xi_{t+1}(i),\quad i\in\bZ^d,t\in \bZ^+,\label{i1}\ee
where $\{w(k)\}_{k\in\bZ^d}$ is a fixed weight vector with the following properties
\begin{align}
&0\le w(k) <1\mbox{, for all }k\in\bZ^d\mbox{, }\sum_{k\in\bZ^d}w(k)=1\mbox{ and the support }\supp(w)=\{k\in\bZ^d: w(k)>0\}\notag\\
&\mbox{is finite.}\label{weight-assump-1}
\end{align}
The mean (vector) of $w$ is denoted by $\wm=\sum_{k\in \bZ^d}kw (k)$. In dimension $d=1$, we write the variance as
\be \sigma_1^2=\sum_{k\in\bZ}(k-\wm)^2w (k).\ee
Assumption \eqref{weight-assump-1} implies that $0<\sigma_1^2<\infty$. 
\begin{align}
&\{\xi_t(k)\}_ {k\in\bZ^d, t\in\bZ}\mbox{ are assumed to be i.i.d.\ random noise variables with mean zero and variance }\notag\\
&\Vvv\bigl(\xi_0(0)\bigr)=\sigma_\xi^2<\infty.\label{xi-assump-1}
\end{align}

Roughly speaking, \eqref{i1} can be viewed as a discrete version of the EW equation in \eqref{i0}. Note that the evolution of random average process (RAP) is quite similar to \eqref{i1} except two differences: it does not have the noise term $\xi$ and the weight vector $\{w(k)\}_{k\in\bZ^d}$ is a random vector called random environment (see \cite{1}).\par 

We will think of the weight vector $\{w(k)\}_{k\in\bZ^d}$ as the transition probability of a discrete-time random walk on $\bZ^d$. We will denote this transition kernel by  
\be p(i,j)=w(j-i),\quad i,j\in \bZ^d,\label{21-d1}\ee 
and multistep transition probabilities by
\be p^k(i,j)=\sum_{i_1,i_2,\ldots,i_{k-1}\in\bZ^d}p(i,i_1)p(i_1,i_2)\cdots p(i_{k-1},j),  \quad i,j\in \bZ^d,k\in\bZ^+,\ee
where $p^0(i,j)=\ind\{i=j\}$, $p^1(i,j)=p(i,j)$. The random walk on $\bZ^d$ with transition probability $p(i,j)$ and initial position $i$ is denoted by $\{X_t^i\}_{t\in\bZ^+}$.\par 

For future use, we denote another transition kernel
\be q(i,j)=\sum_{z\in\bZ^d}w(z)w(j-i+z),\quad i,j\in\bZ^d.\label{21-d2}\ee
We use $\{Y_t^i\}_{t\in\bZ^+}$ to represent the random walk with transition probability $q(i,j)$ and starting point $i$. Notice that $Y_t^i$ is a symmetric random walk, the distribution of which is the same as $\tilde{X}_t^{i}-X_t^{0}$ (see the proof of Lemma \ref{lmm2-2}) where $\tilde{X}_t^{i}$ and $X_t^{0}$ are independently distributed random walks with transition probability $p(\cdot, \cdot)$. The multistep transitions
\be q^k(i,0)=\sum_{z\in\bZ^d}p^k(i,z)p^k(0,z),\quad i\in\bZ^d, k\in\bZ^+.\label{multistep-q}\ee
Under assumption \eqref{weight-assump-1}, $q$ will also be finitely supported and nondegenerate.\par

As an analogue of the Harris graphical construction in \cite{3}, the harness process $\{\h_t\}_{t\in\bZ^+}$ has the following random walk representation.
\begin{lemma}\label{lmm2-1}
For all $t\in\bZ^+$, $i\in\bZ^d$,
\be \h_t(i)=\mE\left[\h_0(X_t^{i})\right]+\sum_{k=1}^t\mE\left[\xi_k(X_{t-k}^{i})\right].\label{21-l1}\ee 
\end{lemma}
Notice that the initial state $\{\h_0(i): i\in\bZ^d\}$ and the noise variables $\{\xi_k(i): i\in\bZ^d, k\in\bN\}$ remain random in the expectation above. We assume that $\{\h_0(i)\}_{i\in\bZ^d}$ is independent of $\{\xi_k(i)\}_{i\in\bZ^d, k\in\bN}$.\par

For surface growth models, on the macroscopic and deterministic scale, the height equation should be a Hamilton-Jacobi equation: $\frac{\partial v}{\partial t}+H(\nabla v)=0$. The slope satisfies conservation law, and the function $H$ is called the flux. If the flux is linear, then the model falls into EW class. If the flux is strictly convex or concave, then the system is in KPZ class. For harness process, by applying \eqref{21-l1}, we can show that the flux is linear. To be specific, the macroscopic height function $h_t$ is simply translated by speed $b=-\wm$.\par
\begin{theorem}\label{hydrodynamic-limit}
Assume \eqref{weight-assump-1} and \eqref{xi-assump-1}. Suppose $\{\h^n_t(i): t\in\bZ^+, i\in\bZ^d\}_{n\in\bN}$ is a sequence of independent harness processes and $\frac{1}{n}\h^n_0(\fl{nx})$ converges in probability to a continuous function $u(x)$ with $u(0)=0$ uniformly on any bounded set as $n$ goes to infinity, i.e.
\be \lim_{n\to\infty}\bP\left(\sup_{\abs{x}\le R}\left|\frac{1}{n}\h^n_0(\fl{nx})-u(x)\right|>\epsilon\right)=0,\quad\forall\epsilon,R>0.\label{hydro-limit-assum}\ee
Then, for all $x\in\bR^d$,
\be \frac{1}{n}\h^n_{\fl{nt}}(\fl{nx})\overset{p}{\to}u(x-bt),\mbox{ as }n\to\infty.\label{hydro-limit}\ee
\end{theorem}

The limit in \eqref{hydro-limit} is called ``hydrodynamic limit" of the process. And $v(x,t)=u(x-bt)$ is the unique solution of the linear transport equation 
\be v_t+bv_x=0,\mbox{ with initial condition }v(x,0)=u(x). \label{transport-equation}\ee
This is the dynamics of the macroscopic harness process. The lines $x(t)=x+bt$ are called the characteristics of \eqref{transport-equation}. This hydrodynamic limit suggests that the harness process should obey EW universality.\par 

From now on, we restrict to dimension $d=1$ and further assume the probability vector $\{w(i)\}_{i\in\bZ}$ to have span 1, i.e.  
\be\max\{k\in\bZ^+:\exists\ell\in\bZ,\quad s.t.\quad\supp(w)\subset \ell+k\bZ\}=1.\label{weight-assump-2}\ee

The new assumption \eqref{weight-assump-2} guarantees that the transition kernel $q(i,j)$ will also have span 1. We summarize the properties of $q$ below.
\begin{lemma}\label{lmm2-2}
Assume $d=1$, \eqref{weight-assump-1} and \eqref{weight-assump-2}. $q$-walk is symmetric (and hence recurrent), irreducible and has span 1. The mean and variance of the one-step transition are
 \be \sum_{x\in\bZ}xq(0,x)=0, \quad\sum_{x\in\bZ}x^2q(0,x)=2\sigma_1^2.\ee
\end{lemma}

\section{Invariant measures}\label{invariant-measure-results}

The proofs for the results in this section can be found in Section \ref{invariant-measure-proofs}.\par

Because of the nonexistence of invariant distributions of the harness process $\h_t$ in one space dimension (see \cite{34}), in this section, we will mainly forcus on the construction and the uniqueness of the ergodic (spatially speaking) invariant measures of the increment process $\{\eta_t(x):x\in\bZ\}_{t\in\bZ^+}$ which is defined below.
\be
\eta_t(i)=\h_t(i)-\h_t(i-1), \quad i\in\bZ,t\in\bZ^+.\label{22-0}\ee  
From the dynamics of harness processes \eqref{i1}, we can derive the evolution of the increments $\eta_t$.
\be\eta_{t+1}(i)=\sum_{k\in\bZ}w(k)\eta_t(i+k)+\xi_{t+1}(i)-\xi_{t+1}(i-1),\quad i\in\bZ,t\in \bZ^+.\label{increment-evolution}\ee
For the invariant distributions of the general increment processes in higher dimensions ($d\ge 2$), please see \cite{34}.\par 

We would like to set up some basic terminology before we move on to any specific result. Let $\cM$ be the space of probability measures on $\bR^\bZ$. A measure $\nu\in\cM$ is said to be 
invariant for the process $\eta_t$ defined in \eqref{22-0} if $\eta_0\sim \nu$ implies $\eta_1\sim\nu$. The convex set of all invariant measures of $\eta_t$ is denoted by $\cI$. Let $\{\theta_x\}_{x\in\bZ}$ be the set of shift operators in space. As an example, for $\eta\in \bR^\bZ$, $(\theta_x \eta)(i)=\eta(i+x)$, $\forall i\in\bZ$. A measure $\nu\in\cM$ is said to be shift invariant in space if $\nu(\theta_xA)=\nu(A)$ for all Borel sets $A\subseteq \bR^\bZ$ and $x\in\bZ$. The collection of all shift invariant measures in $\cM$ is denoted by $\cJ$. A Borel set $B\subseteq \bR^\bZ$ is invariant if $\theta_xB=B$ for all $x\in\bZ$. A shift-invariant measure $\nu\in\cJ$ is ergodic if $\nu(B)=0$ or $1$ for every invariant Borel set $B\subseteq\bR^\bZ$.\par  

For the construction of the invariant measures of $\eta_t$, we first define a harness process $\{\h_{[s,t]}(i):i\in\bZ\}_{t\ge s}$ starting at time $s$ with a flat configuration, i.e.\ $\h_{[s,s]}(i)= 0$, $\forall i\in\bZ$. Then, from the dual representation \eqref{21-l1}, at time $t>s$, the heights $\h_{[s,t]}(\cdot)$ can be represented as
\be \h_{[s,t]}(i)=\sum_{j\in\bZ}\sum_{k=s+1}^{t}\xi_k(j)p^{t-k}(i,j), \quad i\in\bZ, t>s. \label{22-1} \ee

In addition, we can also define a harness process $\h^\varsigma_{[s,t]}$ starting at time $s$ with configuration $\varsigma$, i.e.\ $\h^\varsigma_{[s,s]}(i)= \varsigma(i)$, $\forall i\in\bZ$. Then, we can also write $\h^\varsigma_{[s,t]}$ as
\be \h^\varsigma_{[s,t]}(i)=\sum_{j\in\bZ}\sum_{k=s+1}^{t}\xi_k(j)p^{t-k}(i,j)+\sum_{j\in\bZ}\varsigma(j)p^{t-s}(i,j), \quad i\in\bZ,t>s.\label{22-2}\ee

We can show that
\begin{theorem}\label{thm2-1}
Assume $d=1$, \eqref{weight-assump-1}, \eqref{xi-assump-1} and \eqref{weight-assump-2}.
\begin{enumerate}
\item For each fixed $i\in \bZ$, and $t\in \bZ^+$, the process $\{\h_{[t-s,t]}(i)-\h_{[t-s,t]}(i-1): s\in\bZ^+\}$ is an $L^2$-martingale with respect to the filtration $(\cF_s)_{s\ge0}$ where $\cF_s=\sigma\bigl(\xi_{t-s+1}(\cdot),\ldots,\xi_t(\cdot)\bigr)$. Furthermore, the martingale $\h_{[t-s,t]}(i)-\h_{[t-s,t]}(i-1)$ converges both almost surely and in $L^2$-norm as $s\to\infty$. We denote the limit by $\Delta_t(i)$, $i\in\bZ$.
\item $\{\Delta_t\}_{t\in\bZ^+}$ is a stationary Markov process. The distribution of $\Delta_0$ is an ergodic (space-wise) invariant measure for the increment process $\eta_t(\cdot)$.
\end{enumerate}
\end{theorem}
The representation \eqref{22-1} suggests that the process $\Delta_t$ can be written as
\be \Delta_t(i)=\sum_{j\in\bZ}\sum_{k=0}^{\infty}\xi_{t-k}(j)\left[p^k(i,j)-p^k(i-1,j)\right],\quad i\in\bZ,t\in\bZ^+.\label{initial-increments-rep}\ee
By \eqref{initial-increments-rep}, one can easily check that the process $\{\Delta_t\}_{t\in\bZ^+}$ obeys evolution \eqref{increment-evolution} and hence itself is an increment process of a harness process.\par
 
Let us denote the distribution of $\Delta_0$ by $\pi_0\in \cI\cap\cJ$. Since $\pi_0$ is ergodic, it is one extreme point of $\cI$. The mean and covariances of $\pi_0$ are described below.
\begin{proposition}\label{crlly2-1}
Assume $d=1$, \eqref{weight-assump-1}, \eqref{xi-assump-1} and \eqref{weight-assump-2}. $\pi_0$ has mean zero and covariance $\cov0(\cdot,\cdot)$ given by 
\be
\cov0(i,j)=\sigma_\xi^2[a(i-j-1)+a(i-j+1)-2a(i-j)],\quad i,j\in\bZ,\label{t37-1}\ee
where $a(x)$ is the potential kernel,
\be
a(x)=\sum_{k=0}^\infty [q^k(0,0)-q^k(x,0)],\quad x\in\bZ,\label{t37-2}\ee
and the associated transition kernel $q$ is defined in \eqref{21-d2}.
\end{proposition}
Notice that the convergence of the infinite series in \eqref{t37-2} is guaranteed by assumption \eqref{weight-assump-1} and \eqref{weight-assump-2} (see either Lemma \ref{lmm2-6} in Section \ref{invariant-measure-proofs} below or P28.8 in \cite{4}).\par

Now we can see that the invariant measure $\pi_0$ is not degenerate since $a(0)=0$ and $a(x)>0$ for all $x\neq 0$ due to Lemma \ref{lmm3-1} in the proof of Theorem \ref{thm2-1}. The potential kernel $a(x)$ has been well studied in \cite{4}. And some of the useful results are listed in Appendix \ref{potential-kernel}. From the properties of the potential kernel $a(x)$, we can make a few comments on the covariance function $\cov0(0,x)$.
\begin{corollary}\label{rmk2-4}
Assume $d=1$, \eqref{weight-assump-1}, \eqref{xi-assump-1} and \eqref{weight-assump-2}. 
 \begin{enumerate}
\item The spectral density function (see definition below, the details can be found in Chapter 4 of \cite{32}) of $\cov0(0,x)$ can be written as 
\be f(\theta)=\frac{\sigma_\xi^2}{\pi}\cdot\frac{1-\cos(\theta)}{1-\sum_{k\in\bZ}q(0,k)e^{\i k\theta}};\label{characteristic-function}\ee
\item There exist constants $A, c>0$ such that 
\be \abs{\cov0(0,x)}\le Ae^{-c\abs{x}}, \quad \forall x\in\bZ;\label{cov-bound}\ee
\item  \be \sum_{k\in\bZ} \cov0(0,k)=\frac{\sigma_\xi^2}{\sigma_1^2},\label{series-of-covariances}\ee  
The series in \eqref{series-of-covariances} is called the series of covariances, and it converges absolutely.
\end{enumerate}
\end{corollary}
\begin{definition}
A function $f$ defined on $(-\pi,\pi]$ is the unique spectral density of a stationary process $\{X_t\}_{t\in\bZ}$ with covariance $\covg(\cdot,\cdot)$ if
\begin{itemize}
\item $f(\theta)\ge 0$ for all $\theta\in (-\pi,\pi]$,
\item $\covg(0,k)=\int_{-\pi}^\pi e^{\i k\theta}f(\theta) d\theta$, for all $k\in\bZ$.
\end{itemize}
\end{definition}

Now let us further investigate the invariant measure $\pi_0$. First, let us give the following definitions.
\begin{definition}\label{def2-3}
A mean-zero real-valued stochastic process $\{\eta(x)\}_{x\in\bZ}$ that is stationary in the wide sense (covariance-stationary)  is called linearly regular if the space
$$\LH(-\infty,-\infty)=\bigcap_{x}\LH(-\infty, x)$$
is trivial, where $\LH(a,b)$ is the mean square closed linear hull of $\{\eta(y): a\le y \le b\}$, i.e.\ $\LH(a,b)$ is the minimal closed set in $L^2(\bP)$ that contains the linear span of  $\{\eta(y): a\le y \le b\}$.
\end{definition} 
More details about the space $\LH(a,b)$ can be found in \cite{15} (see Chapter I.5).
\begin{definition}\label{def2-4}
A mean-zero real-valued stochastic process $\{\eta(x)\}_{x\in\bZ}$ that is stationary in the wide sense (covariance-stationary) is called completely linearly regular if
$$\rho(x)=\sup\limits_{\substack{\phi_1\in \LH(x,\infty), \phi_2\in \LH(-\infty, 0)\\ \|\phi_1\|_2=\|\phi_2\|_2=1}}|\bE \left[\phi_1\phi_2\right]|\to 0, \quad\mbox{as }x\to \infty.$$
$\{\rho(x)\}_{x\in\bZ^+}$ are called the coefficients of complete linear regularity.
\end{definition}

The linear regularity condition has been introduced and well-studied in \cite{15}, and it plays an important role in the prediction theory of stationary random processes (see \cite{29} for detail). Here we will show that $\pi_0$ is indeed completely linearly regular.
\begin{theorem}\label{thm2-2}
 Assume $d=1$, \eqref{weight-assump-1}, \eqref{xi-assump-1} and \eqref{weight-assump-2}. The $\pi_0$-distributed process $\{\eta(x)\}_{x\in\bZ}$ is completely linearly regular with linear regularity coefficient
$$\rho(x)=o(x^{-n}),\quad\mbox{for all }n\in \bZ^+.$$
\end{theorem}

More interestingly, if we set the noise $\xi$ to be Gaussian in \eqref{initial-increments-rep}, the result can be stronger.
\begin{definition}\label{def2-5}
A stationary stochastic process $\{\eta(x)\}_{x\in\bZ}$ is called completely regular if 
$$\varrho(x)=\sup\limits_{\substack{\phi_1\in L^2(\cF_{x}^\infty), \phi_2\in L^2(\cF_{-\infty}^0)\\ \|\phi_1\|_2=\|\phi_2\|_2=1}}|\Cvv(\phi_1,\phi_2)|\to 0, \quad\mbox{as }x\to \infty.$$
where $\cF_{m}^n=\sigma\{\eta(x): m\le x\le n\}$, and $\{\varrho(x)\}_{x\in\bZ^+}$ are called the coefficients of complete regularity.
\end{definition}

\begin{corollary}\label{crlly2-2}
Assume $d=1$, \eqref{weight-assump-1}, \eqref{xi-assump-1}, \eqref{weight-assump-2} and $\{\xi_t(i)\}_{i\in\bZ,t\in\bZ}$ have i.i.d.\ Gaussian distribution. Then $\pi_0$ is a centered Gaussian field (also called Gauss measure) with the covariance function $\cov0(\cdot,\cdot)$. $\pi_0$-distributed process $\{\eta(x)\}_{x\in\bZ}$ is stationary and completely regular with regularity coefficient 
$$\varrho(x)=o(x^{-n}),\quad\mbox{for all }n\in \bZ^+.$$
\end{corollary}

For the uniqueness of the ergodic (spatially speaking) invariant measure of the increment process $\{\eta_t\}_{t\in\bZ^+}$, we have the following theorem. 
\begin{theorem}{(Uniqueness)}\label{thm2-3}
Assume $d=1$, \eqref{weight-assump-1}, \eqref{xi-assump-1} and \eqref{weight-assump-2}. Let $\nu\in\cI\cap\cJ$ satisfy the following properties. $\nu$ is an ergodic measure, $\bE^\nu\left|\eta(0)\right|<\infty$ and $\bE^\nu\left[\eta(0)\right]=c$. Denote the distribution of $\{c+\Delta_0(x)\}_{x\in\bZ}$ by $\pi_c$. Then, 
$$\nu=\pi_c.$$  
\end{theorem}
More results about the structure of $\cI$ can be found in \cite{34}.\par

\section{Limits for height fluctuations}\label{fluctuation-limits-results}
The proofs for the results in this section can be found in Section \ref{fluctuation-limit}.\par

In this section, we assume that the initial height function $\h_0: \bZ\to\bR$ is normalized by $h_0(0)=0$. The distribution of the initial increment process $\{\eta_0(x)\}_{x\in\bZ}$ is assumed to be shift invariant and ergodic. We denote the mean, the variance and the series of covariances of the initial increments by 
\be\mu_0=\bE\left[\eta_0(0)\right],\quad\sigma^2_0=\Vvv\left[\eta_0(0)\right],\quad \varsigma^2=\sum_{x\in\bZ}\Cvv\left[\eta_0(0),\eta_0(x)\right].\label{initial-increments}\ee 
The convergence of the series above will be guaranteed by condition (a), (b) and (c) in Theorem \ref{thm2-4} below. One can show that $\varsigma^2$ is the limit of $n^{-1}\Vvv\left[\eta_0(1)+\eta_0(2)+\cdots+\eta_0(n)\right]$ and hence nonnegative (see Lemma 1.1 in \cite{27}).\par 

We are interested in the fluctuation on the marcroscopic characteristic line $x(t)=bt$ with spatial scaling $\sqrt{n}$ (note that $b=-\wm$). We find that the magnitude of this fluctuation is $n^{1/4}$. To be more specific, we are studying the weak limit of the following subdiffusive-scaled fluctuation:
\be
 \fluc_n(t,r)=n^{-1/4}\left\{\h_{\lfloor nt\rfloor}\bigl(\lfloor r\sqrt{n}\rfloor +\lfloor ntb\rfloor\bigr)-\mu_0r\sqrt{n}\right\}.\label{2-1}\ee
From Lemma \ref{lmm2-1}, $\fluc_n(t,r)$ has the following dual representation 
\be
\fluc_n(t,r)=n^{-1/4}\left\{\mE\left[\h_0(X_{\lfloor nt\rfloor}^{y(n)})\right]+\sum_{k=1}^{\lfloor nt\rfloor}\mE\left[\xi_k(X_{\lfloor nt\rfloor -k}^{y(n)})\right]-\mu_0r\sqrt{n}\right\},\label{2-l1}\ee
where $y(n)=\lfloor ntb\rfloor +\lfloor r\sqrt{n}\rfloor$, and $\{X_k^i\}_{k\in\bZ^+}$ is a random walk on $\bZ$ starting from site $i\in\bZ$ with transition kernel $p(x,y)$ defined in \eqref{21-d1}. The expectation $\mE$ only acts on the random walk $X_\centerdot^{y(n)}$.\par

Our main work is to show that the process $\{\fluc_n(t,r)\}_{t\in\bR^+,r\in\bR}$ will converge weakly to the weak solution of an Edwards-Wilkinson equation \eqref{i0}. We will study the fluctuation limits under three circumstances: the initial increments $\{\eta_0(x):x\in\bZ\}$ are (a) i.i.d. (b) $\pi_0$-distributed, or (c) a strongly mixing stationary sequence. The strong mixing condition is defined below.
\begin{definition}\label{def2-6}
Let $\{\eta(i): i\in\bZ\}$ be a stochastic sequence and $\mathcal{F}_n^m=\sigma(\eta(i), n\le i\le m)$. We say that the sequence $\eta$ is strong mixing if $\alpha(n)\to 0$ as $n\to \infty$ where the strong mixing coefficient is
\be \alpha(n)=\sup_k \alpha (\mathcal{F}_{-\infty}^k, \mathcal{F}_{k+n}^\infty),\ee
where
\be \alpha(\mathcal{A},\mathcal{B})=\sup_{A\in\mathcal{A}, B\in\mathcal{B}}|\bP(A\cap B)-\bP(A)\bP(B)|\label{strong-mixing-def}\ee
for two sub-$\sigma$-algebras $\mathcal{A}$ and $\mathcal{B}$ on a probability space $(\Omega, \cF, P)$. 
\end{definition}
For the properties of strong mixing conditions(e.g.\ the differences and relations between strong mixing and completely regular), we refer to \cite{14}.\par

Now let us depict the limit process. Let us denote the centered Gaussian p.d.f and c.d.f with variance $\nu^2$ by
\be \varphi_{\nu^2}(x)=\frac{1}{\sqrt{2\pi\nu^2}}\exp\left(-\frac{x^2}{2\nu^2}\right)\mbox{ and }\Phi_{\nu^2}(x)=\int_{-\infty}^x\varphi_{\nu^2}(y)dy,\label{gamma1}\ee
and define the Gaussian process $\{Z(t,r): t\in\bR^+,r\in\bR\}$ to be the sum of two stochastic integrals
\be Z(t,r)= \frac{\sigma_{\xi}}{\sigma_1}\int\int_{[0,t]\times \bR}\varphi_{\sigma_1^2(t-s)}(r-x)dW(s,x)+\varsigma\int_\bR\varphi_{\sigma_1^2t}(r-x)B(x)dx,\ee
where $\{W(t,r):t\in\bR^+,r\in\bR\}$ is a two-parameter Brownian motion and $\{B(r): r\in\bR\}$ is a two-sided Brownian motion. $W$ and $B$ are independent. In fact, $Z(t,r)$ is also the unique mild solution (\cite{35}) of the following EW equation on $\bR^+\times \bR$:
\be \frac{\partial Z}{\partial t}=\frac{\sigma_1^2}{2}\frac{\partial^2Z}{\partial r^2}+\frac{\sigma_{\xi}}{\sigma_1}\dot{W}, \quad Z(0,r)=\varsigma B(r).\ee
The process $\{Z(t,r)\}_{t\in\bR^+,r\in\bR}$ has zero mean and covariance 
\be
\mathbb{E}\bigl[Z(s,q)Z(t,r)\bigr]=\frac{\sigma_\xi^2}{\sigma_1^2}\Gamma_1\bigl((t,r),(s,q)\bigr)+\varsigma^2\Gamma_2\bigl((t,r),(s,q)\bigr),\label{2-t13b}\ee 
where $\Gamma_1$, $\Gamma_2$ are given as follows. First define the function
\be \Psi_{\nu^2}(x)=\nu^2\varphi_{\nu^2}(x)-x\left(1-\Phi_{\nu^2}(x)\right).\label{gamma2}\ee
Then, the two functions $\Gamma_1$, $\Gamma_2$ are expressed as
\be \Gamma_1\bigl((s,q),(t,r)\bigr)=\Psi_{\sigma_1^2(t+s)}(r-q)-\Psi_{\sigma_1^2|t-s|}(r-q),\label{gamma3}\ee
and
\be \Gamma_2\bigl((s,q),(t,r)\bigr)=\Psi_{\sigma_1^2s}(-q)+\Psi_{\sigma_1^2t}(r)-\Psi_{\sigma_1^2(t+s)}(r-q).\label{gamma4}\ee 

\begin{theorem}\label{thm2-4}
Assume $d=1$, \eqref{weight-assump-1}, \eqref{xi-assump-1}, \eqref{weight-assump-2}, $\bE\left[\xi_0(0)^4\right]<\infty$, and that one of the following conditions is true.
\begin{enumerate}[(a)]
\item $\{\eta_0(x):x\in\bZ\}$ is an i.i.d.\ sequence with finite second moment;
\item $\{\eta_0(x):x\in\bZ\}$ obeys the invariant measure $\pi_0$ of the sequence $\{\Delta_0(x)\}_{x\in\bZ}$ defined in \eqref{initial-increments-rep};
\item $\{\eta_0(x):x\in\bZ\}$ is a strongly mixing stationary sequence, and there exists a $\delta>0$ such that $\bE|\eta_0(0)|^{2+\delta}<\infty$, and  the strong mixing coefficients of $\{\eta_0(x)\}_{x\in\bZ}$ satisfy 
\be\sum_{j=0}^\infty (j+1)^{2/\delta}\alpha(j)<\infty.\label{strong-mixing-coef-assump1}\ee
\end{enumerate}
Then, the series of covariances $\sum_{x\in\bZ}\Cvv(\eta_0(0),\eta_0(x))$ converges absolutely. The fluctuation process $\{\fluc_n(t,r)\}_{t\in\bR^+,r\in\bR}$ will converge weakly to the Gaussian process $\{Z(t,r)\}_{t\in\bR^+,r\in\bR}$ in the sense of finite dimensional distributions, i.e.\ for any fixed integer $N>0$, any pairs $(t_1,r_1), (t_2,r_2),\ldots,(t_N,r_N)\in \bR^+\times \bR$,   
\be\left(\fluc_n(t_1,r_1), \fluc_n(t_2,r_2),\ldots,\fluc_n(t_N,r_N)\right)\Rightarrow \left(Z(t_1,r_1), Z(t_2,r_2),\ldots,Z(t_N,r_N)\right),\quad\mbox{as }n\to\infty.\label{2-t13a}\ee
\end{theorem}
\begin{remark}
\begin{enumerate}
\item In \eqref{2-t13b}, we can see that the covariance of the limit process $Z$ has two parts, the $\Gamma_1$ part comes from the dynamical fluctuations (i.e.\ the randomness caused by the noise variables $\{\xi_k(x)\}_{k\in\bZ^+,x\in\bZ}$), while the $\Gamma_2$ part is contributed by the initial fluctuations (the randomness of the initial increments $\{\eta_0(x)\}_{x\in\bZ}$).
\item In case (a), $\varsigma^2=\sigma_0^2$. In case (b), from \eqref{series-of-covariances}, $\varsigma^2=\frac{\sigma_\xi^2}{\sigma_1^2}$.
\item In case (c), it is possible that $\varsigma^2=0$. If that happens, the randomness of the initial increments will not have any impact on the limit process.
\item If the noise terms $\{\xi_k(x)\}_{k\in\bZ,x\in\bZ}$ are normally distributed, then case (b) is covered by case(c) due to Corollary \ref{crlly2-2} and the fact that complete regularity is stronger than strong mixing, i.e.\ $\alpha(x)\le\varrho(x)$, $\forall x\in\bZ^+$ (see \cite{14}). 
\item In case (b), at $r=0$, the limit process $\{Z(t, 0)\}_{t\in\bR^+}$ is a fractional Brownian motion with Hurst parameter $1/4$. The covariance has the form
\be \bE \left[Z(s, 0)Z(t, 0)\right]=\frac{\sigma_\xi^2}{\sqrt{2\pi \sigma_1^2}}(\sqrt{s}+\sqrt{t}-\sqrt{|t-s|}).\ee
\end{enumerate}   
\end{remark}


Notice that the fluctuation process $\{\fluc_n(t,r)\}_{t\in\bR^+,r\in\bR}$ lives in a 2-parameter cadlag path space (continuous from right above and have limits from other directions). Let us denote this 2-parameter cadlag function space with Skorohod's topology by (see Definition \ref{skorohod-metric}, more details can be found in \cite{12}) 
\begin{align*}
D_2=&D_2(Q,\bR)\\
:=&\{f:Q\rightarrow \bR\mbox{ s.t. for }\forall (t_0,r_0)\in Q,\lim_{\begin{subarray}{l} (t,r)\in Q_{(t_0,r_0)}^i \\ (t,r)\rightarrow (t_0,r_0)\end{subarray}}f(t,r)\mbox{ exists for }i=1,2,3,4,\\
\mbox{and}&\lim_{\begin{subarray}{l} (t,r)\in Q_{(t_0,r_0)}^1 \\ (t,r)\rightarrow (t_0,r_0)\end{subarray}}f(t,r)=f(t_0,r_0)\}.
\end{align*}
where $Q=[0,T]\times [-R,R]$, and $Q_{(t_0,r_0)}^i$, $i=1,2,3,4$ are four quadrants of $Q$:
\begin{align*}
&Q_{(t_0,r_0)}^1:=\{(t,r)\in Q: t\ge t_0, r\ge r_0\}, Q_{(t_0,r_0)}^2:=\{(t,r)\in Q: t\ge t_0, r< r_0\},\\
&Q_{(t_0,r_0)}^3:=\{(t,r)\in Q: t< t_0, r< r_0\}, Q_{(t_0,r_0)}^4:=\{(t,r)\in Q: t< t_0, r\ge r_0\}.
\end{align*}

\begin{definition}\label{skorohod-metric}
Let $\Lambda$ be the set of all transformations $\lambda: Q \to Q$ of the form $\lambda(t,r)=(\lambda_1(t), \lambda_2(r))$ where both $\lambda_1$ and $\lambda_2$ are strictly increasing, continuous bijections. We define the Skorohod distance between $x,y\in D_2$ to be
$$d_S(x,y)=\inf_{\lambda\in\Lambda}\max(\norm{x-y\lambda},\norm{\lambda}),$$
where $\norm{x-y\lambda}=\sup_{u\in Q}\abs{x(u)-y(\lambda(u))}$ and $\norm{\lambda}=\sup_{u\in Q}\abs{\lambda(u)-u}$.
\end{definition}

We will show that under stronger assumptions on the moments of $\{\eta_0(x)\}_{x\in\bZ}$ and $\{\xi_k(x)\}_{k,x\in\bZ}$ and the strong mixing coefficients $\{\alpha(k)\}_{k\in\bZ^+}$, the weak convergence in finite dimensional distributions of $\fluc_n(\cdot,\cdot)$ in Theorem \ref{thm2-4} can be strengthened into a process level convergence.\par

\begin{theorem}\label{thm2-6}
Assume $d=1$, \eqref{weight-assump-1}, \eqref{xi-assump-1}, \eqref{weight-assump-2}, $\bE\left[\xi_0(0)^{12}\right]<\infty$, and that one of the following conditions is true.
\begin{enumerate}[(a)]
\item $\{\eta_0(x): x\in\bZ\}$ is an i.i.d.\ sequence with finite 12th moment;
\item $\{\eta_0(x):x\in\bZ\}$ has the distribution $\pi_0$ of the sequence $\{\Delta_0(x)\}_{x\in\bZ}$ defined in \eqref{initial-increments-rep};
\item $\{\eta_0(x):x\in\bZ\}$ is a strongly mixing stationary sequence, and there exists $\varepsilon_0>0$ such that $\bE\left[|\eta_0(0)|^{12+\varepsilon_0}\right]<\infty$ and the strong mixing coefficients $\{\alpha(x)\}_{x\in\bZ^+}$ satisfy 
\be\sum_{i=0}^{\infty}(i+1)^{10+132/\varepsilon_0}\alpha(i)<\infty.\label{strong-mixing-coef-assump2}\ee
\end{enumerate}
Then the fluctuation process $\{\fluc_n(t,r)\}_{t\in\bR^+,r\in\bR}$ converges weakly to $\{Z(t,r)\}_{t\in\bR^+,r\in\bR}$ on $D_2$ in the Skorokhod topology, i.e.\ $$\lim_{n\to\infty}\bE f(\fluc_n)= \bE f(Z)$$ 
for all Skorokhod-continuous bounded functions $f:D_2\rightarrow \bR$.
\end{theorem}

\chapter{Proofs}\label{proofs}
\section{Proofs of the initial preparations}
\begin{proof}[Proof of Lemma \ref{lmm2-1}]
According to the evolution \eqref{i1} and \eqref{21-d1},
\begin{align}
\h_t(i)=&\sum_{k\in\bZ^d} w(k)\h_{t-1}(i+k)+\xi_t(i)=\mE\left[\h_{t-1}(X_1^{i})\right]+\xi_t(i)\notag\\
=&\mE\left[\mE\left(\h_{t-2}(X_2^{i})\mid X_1^{i}\right)+\xi_{t-1}(X_1^{i})\right]+\xi_t(i)\notag\\
=&\mE\left[\h_{t-2}(X_2^{i})\right]+\mE\xi_{t-1}(X_1^{i})+\xi_t(i)\notag\\
=&\ldots\notag\\
=&\mE\left[\h_0(X_t^{i})\right]+\sum_{k=1}^t\mE\xi_k(X_{t-k}^{i}).\label{2-new}
\end{align}
\end{proof} 

\begin{proof}[Proof of Theorem \ref{hydrodynamic-limit}]
From the dual representation \eqref{21-l1},
$$\frac{1}{n}h^n_{\fl{nt}}(\fl{nx})=\frac{1}{n}\mE\left[\h^n_0(X_{\fl{nt}}^{\fl{nx}})\right]+\frac{1}{n}\sum_{k=1}^{\fl{nt}}\mE\left[\xi^n_k(X_{\fl{nt}-k}^{\fl{nx}})\right].$$
Thus, for all $\epsilon>0$, $x\in\bR^d$,
\begin{align}
&\bP\left(\left|\frac{1}{n}h^n_{\fl{nt}}(\fl{nx})-u(x-bt)\right|>\epsilon\right)\notag\\
&\le  \bP\left(\left|\frac{1}{n}\mE\left[\h^n_0(X_{\fl{nt}}^{\fl{nx}})\right]-u(x-bt)\right|+\left|\frac{1}{n}\sum_{k=1}^{\fl{nt}}\mE\left[\xi^n_k(X_{\fl{nt}-k}^{\fl{nx}})\right]\right|>\epsilon\right)\notag\\
&\le  \bP\left(\left|\frac{1}{n}\mE\left[\h^n_0(X_{\fl{nt}}^{\fl{nx}})\right]-u(x-bt)\right|>\frac{\epsilon}{2}\right)+\bP\left(\left|\frac{1}{n}\sum_{k=1}^{\fl{nt}}\mE\left[\xi^n_k(X_{\fl{nt}-k}^{\fl{nx}})\right]\right|>\frac{\epsilon}{2}\right).\label{hydro-limit-1}
\end{align}
For the first part in \eqref{hydro-limit-1}, 
\begin{align*}
&\bP\left(\left|\frac{1}{n}\mE\left[\h^n_0(X_{\fl{nt}}^{\fl{nx}})\right]-u(x-bt)\right|>\frac{\epsilon}{2}\right)=\bP\left(\left|\frac{1}{n}\sum_{i\in\bZ^d}p^{\fl{nt}}(\fl{nx},i)\h^n_0(i)-u(x-bt)\right|>\frac{\epsilon}{2}\right)\\
&\le\bP\Biggl(\biggl|\sum_{i\in\bZ^d}p^{\fl{nt}}(\fl{nx},i)\left[\h^n_0(i)/n-u(i/n)\right]\biggr|+\biggl|\sum_{i\in\bZ^d}p^{\fl{nt}}(\fl{nx},i)u(i/n)-u(x-bt)\biggr|>\frac{\epsilon}{2}\Biggr)\\
&\le\bP\left(\sum_{i\in\bZ^d}p^{\fl{nt}}(\fl{nx},i)\left|\h^n_0(i)/n-u(i/n)\right|>\frac{\epsilon}{4}\right)\\
&\quad\quad\quad\quad\quad\quad\quad\quad\quad+\ind\left\{\biggl|\sum_{i\in\bZ^d}p^{\fl{nt}}(\fl{nx},i)u(i/n)-u(x-bt)\biggr|>\frac{\epsilon}{4}\right\}\\
&\le \bP\left(\sup_{\abs{y}\le Mt+\abs{x}+1}\left|\h^n_0(\fl{ny})/n-u(y)\right|>\frac{\epsilon}{4}\right)+\ind\left\{\biggl|\mE \left[u\left(X_{\fl{nt}}^{\fl{nx}}/n\right)\right]-u(x-bt)\biggr|>\frac{\epsilon}{4}\right\},
\end{align*}
where the last inequality is because from the assumption \eqref{weight-assump-1}, we can find large enough constant $M>0$ such that $w(x)=0$ for all $\abs{x}>M$.\par
 
The condition \eqref{hydro-limit-assum} directly implies that 
$$\lim_{n\to\infty}\bP\left(\sup_{\abs{y}\le Mt+\abs{x}+1}\left|\h^n_0(\fl{ny})/n-u(y)\right|>\frac{\epsilon}{4}\right)=0.$$
And by LLN and the continuity of $u(x)$, one can easily show that
$$\lim_{n\to\infty}\ind\left\{\biggl|\mE \left[u\left(X_{\fl{nt}}^{\fl{nx}}/n\right)\right]-u(x-bt)\biggr|>\frac{\epsilon}{4}\right\}=0.$$
Therefore, we have proved that
\be\lim_{n\to\infty}\bP\left(\left|\frac{1}{n}\mE\left[\h^n_0(X_{\fl{nt}}^{\fl{nx}})\right]-u(x-bt)\right|>\frac{\epsilon}{2}\right)=0.\label{hydro-limit-2}\ee

For the second part in \eqref{hydro-limit-1}, by Markov Inequality, we have
\begin{align}
&\bP\left(\left|\frac{1}{n}\sum_{k=1}^{\fl{nt}}\mE\xi^n_k(X_{\fl{nt}-k}^{\fl{nx}})\right|>\frac{\epsilon}{2}\right) = \bP\left(\left|\sum_{k=1}^{\fl{nt}}\sum_{i\in\bZ^d}p^{\fl{nt}-k}(\fl{nx},i)\xi^n_k(i)\right|>\frac{\epsilon n}{2}\right)\notag\\ 
\le & \frac{4}{\epsilon^2n^2}\bE\left\{\left[\sum_{k=1}^{\fl{nt}}\sum_{i\in\bZ^d}p^{\fl{nt}-k}(\fl{nx},i)\xi^n_k(i)\right]^2\right\}= \frac{4\sigma_\xi^2}{\epsilon^2n^2}\sum_{k=1}^{\fl{nt}}\sum_{i\in\bZ^d}\left[p^{\fl{nt}-k}(\fl{nx},i)\right]^2\notag\\
= & \frac{4\sigma_\xi^2}{\epsilon^2n^2}\sum_{k=0}^{\fl{nt}-1}q^k(0,0)\le \frac{4\sigma_\xi^2\fl{nt}}{\epsilon^2n^2}\to 0,\quad\mbox{as }n\to\infty.\label{hydro-limit-3}
\end{align}
The transition kernels $p$ and $q$ are defined in \eqref{21-d1} and \eqref{21-d2} respectively.\par

Combine \eqref{hydro-limit-2} and \eqref{hydro-limit-3} together, we get
$$\lim_{n\to\infty}\bP\left(\left|\frac{1}{n}h^n_{\fl{nt}}(\fl{nx})-u(x-bt)\right|>\epsilon\right)=0,$$
which completes the proof of Theorem \ref{hydrodynamic-limit}.
\end{proof}

\begin{proof}[Proof of Lemma \ref{lmm2-2}]
Let $\tilde{X}_t^{0}$ and $X_t^{0}$ be two independently distributed random walks with transition probability $p$ in \eqref{21-d1}. We first check that the random walk $Y_t^0=\tilde{X}_t^{0}-X_t^{0}$ has the transition probability $q$.\par

For $\forall k\in \bZ^+, x,y \in \bZ$, 
\begin{align*}
&\mP(Y_{k+1}^0=y|Y_k^0=x)=\frac{\mP(Y_{k+1}^0=y,Y_k^0=x)}{\mP(Y_k^0=x)}=\frac{\mP\bigl(X_{k+1}^0-\tilde{X}_{k+1}^0=y,X_{k}^0-\tilde{X}_{k}^0=x\bigr)}{\mP(X_{k}^0-\tilde{X}_{k}^0=x)}\\
&=\frac{\sum_{u\in\bZ}\sum_{v\in\bZ}\mP\bigl(X_{k+1}^0=y+u,\tilde{X}_{k+1}^0=u,X_{k}^0=x+v,\tilde{X}_{k}^0=v\bigr)}{\sum_{v\in\bZ}\mP(X_{k}^0=x+v)\mP(\tilde{X}_{k}^0=v)}\\
&=\frac{\sum_{u\in\bZ}\sum_{v\in\bZ}w(y+u-x-v)w(u-v)\mP(X_{k}^0=x+v)\mP(\tilde{X}_{k}^0=v)}{\sum_{v\in\bZ}\mP(X_{k}^0=x+v)\mP(\tilde{X}_{k}^0=v)}\\
&\overset{z=u-v}{=}\frac{\sum_{v\in\bZ}\sum_{z\in\bZ}w(y-x+z)w(z)\mP(X_{k}^0=x+v)\mP(\tilde{X}_{k}^0=v)}{\sum_{v\in\bZ}\mP(X_{k}^0=x+v)\mP(\tilde{X}_{k}^0=v)}\\
&=\sum_{z\in\bZ}w(y-x+z)w(z)=q(x,y).
\end{align*}

The symmetry is because
$$q(x,y)=\sum_{z\in\bZ}w(y-x+z)w(z)\overset{u=z-x+y}{=}\sum_{u\in\bZ}w(u)w(u+x-y)=q(y,x).$$

The equivalence of mean zero and recurrence for one dimensional random walks can be found in \cite{4} (T3.1, page 33).\par

For span 1 and irreducibility,  we use B\'{e}zout's Identity and its corollary.
\begin{lemma}{B\'{e}zout's Identity}\label{bezout1}
Let $a_1,a_2,\ldots, a_n$ be integers, not all zero, let $d$ be their greatest common divisor, i.e.\ $d=\gcd(a_1,a_2,\ldots,a_n)$. Then there are integers $x_1,x_2,\ldots,x_n$ such that 
$$d=\sum_{i=1}^na_ix_i.$$
\end{lemma}
\begin{corollary}\label{bezout2}
Let $a_1,a_2,\ldots, a_n$ be integers, not all zero, let $d=\gcd(a_1,a_2,\ldots,a_n)$. Then
$$\{kd: k\in\bZ\}=\{\sum_{i=1}^na_ix_i: x_1,x_2,\ldots,x_n\in\bZ\}.$$
\end{corollary}
The proof of the case $n=2$ can be found in \cite{10} (see Theorem 2-3 and its corollary on page 25), the multi-dimensional case ($n>2$) can be proved by using the result of the case $n=2$.\par

Note that $\supp(q)=\supp(p)-\supp(p)$. Let us denote all the elements in $\supp(p)-\supp(p)$ by
$$\supp(p)-\supp(p)=\{i-j\hspace{1mm}:\hspace{1mm}i,j\in\supp(p)\}\overset{def}{=}\{a_1,a_2,\ldots,a_m\}.$$
The irreducibility of $q$ is equivalent to $\{t_1a_1+t_2a_2+\ldots+t_ma_m\hspace{1mm}:\hspace{1mm}t_i\in\bZ_+\}=\bZ$. Moreover, we can easily see that $\{t_1a_1+t_2a_2+\ldots+t_ma_m\hspace{1mm}:\hspace{1mm}t_i\in\bZ_+\}=\{t_1a_1+t_2a_2+\ldots+t_ma_m\hspace{1mm}:\hspace{1mm}t_i\in\bZ\}$ due to the symmetry of $\supp(q)$.\par

From Corollary \ref{bezout2}, 
$$\{t_1a_1+t_2a_2+\ldots+t_ma_m\hspace{1mm}:\hspace{1mm}t_i\in\bZ\}=d\bZ,$$
where $d=\gcd(a_1,a_2,\ldots,a_m)$.\par

Since $p$ has span 1, $d=1$. Therefore, $q$ is irreducible. Also, since $0\in\supp(q)$, $d=1$ implies that $q$ also has span 1.\par 

The variance of the jump can be calculated by simply noticing $\Vvv(Y_1^0)=\Vvv(X_1^0)+\Vvv(\tilde{X}_1^0)=2\sigma_1^2$.\par

Thus, the proof of Lemma \ref{lmm2-2} is complete. 
\end{proof}

\section{Proofs of invariant distributions}\label{invariant-measure-proofs}
\subsection{The construction of the invariant distributions}
\begin{proof}[Proof of Theorem \ref{thm2-1}]
First, we would like to show that for all fixed $i\in \bZ$, and $t\in \bZ^+$, the process $\{\h_{[t-s,t]}(i)-\h_{[t-s,t]}(i-1): s\in\bZ^+\}$ is an $L^2$-martingale.\par

For all $0\le s\le r$, from \eqref{22-1},
\begin {align*}
\bE\left[\h_{[t-r,t]}(i)-\h_{[t-s,t]}(i)|\cF_s\right]=&\bE\left[\sum_{j\in \bZ}\sum_{k=t-r+1}^{t}\xi_k(j)p^{t-k}(i,j)-\sum_{j\in\bZ}\sum_{k=t-s+1}^{t}\xi_k(j)p^{t-k}(i,j)\biggl|\cF_s\right]\\
=&\bE\left[\sum_{j\in\bZ}\sum_{k=t-r+1}^{t-s}\xi_k(j)p^{t-k}(i,j)\biggl|\cF_s\right]=0.
\end{align*}
The last equation is because $\xi_{t-r+1}(\cdot), \xi_{t-r+2}(\cdot),\ldots,\xi_{t-s}(\cdot)$ are independent of $\cF_s$. After some simple manipulations on the equation above, we can show the martingale property of the process $\{\h_{[t-s,t]}(i)-\h_{[t-s,t]}(i-1)\}_{s\in\bZ^+}$. In order to check the $L^2$ boundedness, we first give an explicit formula for the $2$nd moment of $\h_{[t-s,t]}(i)-\h_{[t-s,t]}(i-1)$. Recall that the transition probability $q$ is defined in \eqref{21-d2}. 
\begin{lemma}\label{lmm2-4}
Assume \eqref{weight-assump-1} and \eqref{xi-assump-1}. For $-\infty < s \le t$, $i,j\in \bZ$, 
\be \bE\bigl[\h_{[s,t]}(i)\bigr]=0,\quad\bE\bigl[\h_{[s,t]}(i)\bigr]^2=\sigma_\xi^2\sum_{k=0}^{t-s-1}q^k(0,0);\label{l33a}\ee
\be \bE\left[\left(\h_{[s,t]}(j)-\h_{[s,t]}(i)\right)^2\right]=2\sigma_\xi^2\sum_{k=0}^{t-s-1}\bigl[q^k(0,0)-q^k(j-i,0)\bigr] , \quad j\not=i,\label{l33b}\ee
where $q^0(i,0)=\ind\{i=0\}$. 
\end{lemma}
\begin{proof}

$$\bE\left[\h_{[s,t]}(i)\right]=\bE\left[\sum_{j\in\bZ}\sum_{k=s+1}^{t}\xi_k(j)p^{t-k}(i,j)\right]=0.$$
\begin{align*}
\bE\left[\h_{[s,t]}(i)\right]^2=&\bE\left[\sum_{j\in\bZ}\sum_{k=s+1}^{t}\xi_k(j)p^{t-k}(i,j)\right]^2=\sigma_\xi^2\sum_{j\in\bZ}\sum_{k=s+1}^{t}p^{t-k}(i,j)^2\\
=&\sigma_\xi^2\sum_{k=s+1}^{t}q^{t-k}(0,0)=\sigma_\xi^2\sum_{k=0}^{t-s-1}q^k(0,0).
\end{align*}
Notice that
\be\bE\bigl[\h_{[s,t]}(j)-\h_{[s,t]}(i)\bigr]^2=\bE\bigl[\h_{[s,t]}(j)\bigr]^2+\bE\bigl[\h_{[s,t]}(i)\bigr]^2-2\bE\bigl[\h_{[s,t]}(i)\h_{[s,t]}(j)\bigr].\label{l33d}\ee
From \eqref{l33a}, 
\be\bE\bigl[\h_{[s,t]}(j)\bigr]^2=\bE\bigl[\h_{[s,t]}(i)\bigr]^2=\sigma_\xi^2\sum_{k=0}^{t-s-1}q^k(0,0).\label{l33e}\ee
For the last term in \eqref{l33d},
\begin{align}
&\bE\bigl[\h_{[s,t]}(i)\h_{[s,t]}(j)\bigr]=\bE\left\{\left[\sum_{k\in\bZ}\sum_{n=s+1}^{t}\xi_n(k)p^{t-n}(j,k)\right]\left[\sum_{k\in\bZ}\sum_{n=s+1}^{t}\xi_n(k)p^{t-n}(i,k)\right]\right\}\notag\\
=&\sigma_\xi^2\sum_{k\in\bZ}\sum_{n=s+1}^{t}p^{t-n}(j,k)p^{t-n}(i,k)=\sigma_\xi^2\sum_{n=s+1}^{t}q^{t-n}(j-i,0)=\sigma_\xi^2\sum_{k=0}^{t-s-1}q^k(j-i, 0).\label{l33c}
\end{align}
Plug \eqref{l33e} and \eqref{l33c} into \eqref{l33d}, we show that 
$$\bE\bigl[\h_{[s,t]}(j)-\h_{[s,t]}(i)\bigr]^2=2\sigma_\xi^2\sum_{k=0}^{t-s-1}\bigl[q^k(0,0)-q^k(j-i,0)\bigr].\qedhere$$
\end{proof}

One can show that the sum on the right hand side of \eqref{l33b} converges as $s$ goes to $-\infty$ under the assumption \eqref{weight-assump-2}. In fact, 
\begin{lemma}\label{lmm2-6}
Assume \eqref{weight-assump-1} and \eqref{weight-assump-2}. For $\forall i \in \bZ$, there exists a constant $\consta(i)<\infty$, s.t. 
\be
\sum_{k=s}^\infty\bigl[q^k(0,0)-q^k(i,0)\bigr]\le\consta(i)s^{-1/2}, \quad \forall s\in\bZ^+.\label{l34}\ee 
\end{lemma}
\begin{proof}
First, we give some useful properties of the transition probability $q$.
\begin{lemma}\label{lmm3-1}
Assume \eqref{weight-assump-1}. Then
\be q^k(i,0)< q^k(0,0),\quad q^{k+1}(0,0)\le q^k(0,0),\quad \forall k>0, i\neq 0.\label{qprop1}\ee
\end{lemma} 
\begin{proof}
From \eqref{multistep-q}, 
$$q^k(i,0)=\sum_{j\in\bZ}p^k(i,j)p^k(0,j)\le \sum_{j\in\bZ}\frac{1}{2}\left[\left(p^k(i,j)\right)^2+\left(p^k(0,j)\right)^2\right]=\sum_{j\in\bZ}\left(p^k(0,j)\right)^2=q^k(0,0),$$
where $p$ is the transition probability defined in \eqref{21-d1}. We can see that $q^k(i,0)=q^k(0,0)$ if and only if 
\be p^k(0,j-i)=p^k(0,j),\mbox{ for all } j\in\bZ.\label{qpprop2}\ee

Suppose that there exists $i\neq 0$ such that $q^k(i,0)=q^k(0,0)$. Notice that $q^k(0,0)>0$ due to $q(0,0)>0$. Hence, $q^k(i,0)>0$. Then, there exists $\ell\in\bZ$ such that $p^k(0,\ell-i)>0$, $p^k(0,\ell)>0$. According to \eqref{qpprop2}, we have  
$$p^k(0,\ell-mi)=p^k(0,\ell)>0,\mbox{ for all }m\in\bZ.$$
This contradicts the assumption that $p$ has finite range.\par

The second inequality can be proved by using the first one,
$$q^{k+1}(0,0)=\sum_{j\in\bZ}q^k(0,j)q(j,0)\le \sum_{j\in\bZ}q^k(0,0)q(j,0)=q^k(0,0).\qedhere$$
\end{proof}
As an analogue to Kolmogorov Backwards Equation, we can rewrite the probability of the random walk $Y_\centerdot^0$ returning to site 0 at time s as
\be
q^s(0,0)=\sum_{k=s}^\infty\sum_{i\in\bZ}q(0,i)\bigl[q^k(0,0)-q^k(i,0)\bigr].\label{b1}\ee
In fact, 
\begin{align*}
&\sum_{k=s}^\infty\sum_{i\in\bZ}q(0,i)\bigl[q^k(0,0)-q^k(i,0)\bigr]=\sum_{k=s}^\infty\left[\sum_{i\in\bZ}q(0,i)q^k(0,0)-\sum_{j\in\bZ}q(0,j)q^k(j,0)\right]\\
=&\sum_{k=s}^\infty\bigl[q^k(0,0)-q^{k+1}(0,0)\bigr]=q^s(0,0).
\end{align*}
Note that every term in the summation \eqref{b1} is nonnegative because of Lemma \ref{lmm3-1}.\par

For $q^s(0,0)$ in \eqref{b1}, we have the following bound.
\begin{lemma}\label{q-uniform-bound}
Assume \eqref{weight-assump-1}.
\be
\exists C>0, \quad s.t.\quad q^s(0,x)\le Cs^{-1/2},\quad\forall s\in\bZ^+,x\in\bZ.\label{b2}\ee
\end{lemma}
The proof of Lemma \ref{q-uniform-bound} can be found in \cite{4} (P7.6, page 72).\par

For $q(0,j)>0$,  from \eqref{b1} and \eqref{b2}, we have 
$$q(0,j)\sum_{k=s}^\infty\bigl[q^k(0,0)-q^k(j,0)\bigr]\le q^s(0,0)\le Cs^{-1/2}.$$
Let $\consta(j)=\frac{C}{q(0,j)}$. Then
$$\sum_{k=s}^\infty\bigl[q^k(0,0)-q^k(j,0)\bigr]\le\consta(j)s^{-1/2}.$$ 
For $q(0,j)=0$, since the random walk $Y_\centerdot^0$ is irreducible under assumption \eqref{weight-assump-2} (Lemma \ref{lmm2-2}), thus, $\exists d>1$, s.t. $q^d(0,j)\not=0$. Then, we can use the method above with $q^d(0,j)$ instead of $q(0,j)$. Again, by an analogue to Kolmogorov Backwards Equation, we have the following equation:
\be
\sum_{k=s}^\infty\sum_{i\in\bZ}q^d(0,i)\bigl[q^k(0,0)-q^k(i,0)\bigr]=\sum_{k=s}^{s+d-1}q^k(0,0).\label{b3}\ee
Combining \eqref{b2} and \eqref{b3}, we have 
$$q^d(0,j)\sum_{k=s}^\infty\bigl[q^k(0,0)-q^k(j,0)\bigr]\le \sum_{k=s}^{s+d-1}q^k(0,0) \le dCs^{-1/2}.$$
Let $\consta(j)=\frac{dC}{q^d(0,j)}$. Then
$$\sum_{k=s}^\infty\bigl[q^k(0,0)-q^k(j,0)\bigr]\le \consta(j)s^{-1/2}.$$ 
The proof of Lemma \ref{lmm2-6} is complete.
\end{proof}

Combine \eqref{l33b} and \eqref{l34} together, we can find a constant $C>0$ such that for all $i\in\bZ$ and $s,t\in\bZ^+$,
\be\bE\left[\left(\h_{[t-s,t]}(i)-\h_{[t-s,t]}(i-1)\right)^2\right]\le C.\label{mctc}\ee
Hence, we have shown that $\{\h_{[t-s,t]}(i)-\h_{[t-s,t]}(i-1): s\in\bZ^+\}$ is an $L^2$-martingale. By the Martingale convergence Theorem (see, e.g.\ Theorem 5.4.5 in \cite{2}), \eqref{mctc} implies the almost sure and $L^2$ convergence of $\h_{[t-s,t]}(i)-\h_{[t-s,t]}(i-1)$ as $s$ goes to $\infty$. Lemma \ref{lmm2-6} gives an $L^2$ speed of convergence.\par

Notice that $\h_{[t-s,t]}(i)-\h_{[t-s,t]}(i-1)=\sum_{j\in\bZ}\sum_{k=t-s+1}^{t}\xi_k(j)\left[p^{t-k}(i,j)-p^{t-k}(i-1,j)\right]$. Taking $s$ to infinity, we can represent the limit $\Delta_t(i)$ as
$$\Delta_t(i)=\sum_{j\in\bZ}\sum_{k=0}^{\infty}\xi_{t-k}(j)\left[p^k(i,j)-p^k(i-1,j)\right],\quad i\in\bZ,t\in\bZ^+.$$

The stationarity of $\Delta_t(\cdot)$ can be seen directly from the construction. And the Markov property can be derived from the Markov property of the harness processes. In fact, according to the setting \eqref{i1},
$$\h_{[t-s,t+1]}(i)-\h_{[t-s,t+1]}(i-1)=\sum_{j\in\bZ}w(j)\left[\h_{[t-s,t]}(i+j)-\h_{[t-s,t]}(i+j-1)\right]+\xi_{t+1}(i)-\xi_{t+1}(i-1).$$
Let $s\to\infty$, we have
\be\Delta_{t+1}(i)=\sum_{j\in\bZ}w(j)\Delta_{t}(i+j)+\xi_{t+1}(i)-\xi_{t+1}(i-1).\label{delta-1}\ee

Also from \eqref{delta-1}, we see that the evolution of the process $\{\Delta_t\}_{t\in\bZ^+}$ is the same as the increment dynamic \eqref{increment-evolution}. Therefore, $\Delta_{\centerdot}$ is an increment process and surely its distribution is the invariant measure of the increment process $\{\eta_t\}_{t\in\bZ^+}$ defined in \eqref{22-0}.\par 


Next, let us prove the ergodicity. Notice that from \eqref{22-1},
$$\h_{[t-s,t]}(x)-\h_{[t-s,t]}(x-1)=\sum_{j\in\bZ}\sum_{k=t-s+1}^{t}\xi_k(j+x)\left[p^{t-k}(0,j)-p^{t-k}(0,j+1)\right]=f_s(\theta_x\xi),\quad x\in\bZ,$$
where $f_s(\xi)=\sum_{j\in\bZ}\sum_{k=t-s+1}^{t}\xi_k(j)\left[p^{t-k}(0,j)-p^{t-k}(0,j+1)\right]$, and $\theta$ is the space-shift operator. \par

Let us denote $\bar{f}(\xi)=\limsup_{s\to\infty}f_s(\xi)$. Since $\lim_{s\to\infty}\h_{[t-s,t]}(i)-\h_{[t-s,t]}(i-1)=\Delta_t(i)$ a.s., $\bar{f}(\theta^i\xi)=\Delta_t(i)$ a.s. Also, according to the settings, $\{\xi_k(j): k\in\bZ, j\in\bZ\}$ are i.i.d. Therefore, by Theorem 7.1.3 in \cite{2}, we may conclude that the sequence $\Delta_t(\cdot)$ is ergodic under spatial translations.\par 

Thus, Theorem \ref{thm2-1} has been proved. 
\end{proof}

\begin{proof}[Proof of Proposition \ref{crlly2-1}]
$\bE[\Delta_t(i)]=0$ is due to the fact that $\Delta_t(i)$ is the $L^2$-limit of $\h_{[t-s,t]}(i)-\h_{[t-s,t]}(i-1)$ as $s\to \infty$.\par 

For the covariance, 
\begin{align*}
\bE \left[\Delta_t(i)\Delta_t(j)\right]=&\lim_{s\to\infty} \bE\left[\left(\h_{[t-s,t]}(i)-\h_{[t-s,t]}(i-1)\right)\left(\h_{[t-s,t]}(j)-\h_{[t-s,t]}(j-1)\right)\right]\\
=&\sigma_\xi^2\sum_{k=0}^{\infty}\left[2q^k(i-j,0)-q^k(i-j-1,0)-q^k(i-j+1,0)\right]\\
=&\sigma_\xi^2[a(i-j-1)+a(i-j+1)-2a(i-j)],\quad i,j\in\bZ,
\end{align*}
where the second equality comes from \eqref{l33c}.
\end{proof}

\begin{proof}[Proof of Corollary \ref{rmk2-4}]
\eqref{characteristic-function}, \eqref{cov-bound} are from Lemma \ref{lmmA-3} in the Appendix. And Lemma \ref{lmmA-4} implies \eqref{series-of-covariances}.
\end{proof}

\subsection{Properties of the invariant distributions}
\begin{proof}[Proof of Theorem \ref{thm2-2}]
This result is a direct application of Theorem 8 from \cite{15} (page 181, section V.6). The theorem is stated as a lemma below.
\begin{lemma}\label{lmm3-3}
A necessary and sufficient condition for 
$$\rho(x)=O(x^{-r-\beta}),\mbox{ where }0<\beta<1,$$
is that the spectral density $f(\lambda)$ permits a representation of the form 
$$f(\lambda)=|P(e^{i\lambda})|^2w(\lambda),$$
where $P(z)$ is a polynomial with zeros on $|z|=1$ and the function $w(\lambda)$ is strictly positive, i.e.\ $\inf_{\lambda\in(-\pi,\pi]}w(\lambda)>0$, and $r$ times differentiable with the $r$th derivative satisfying a H\"{o}lder condition of order $\beta$.
\end{lemma}

In our case, according to \eqref{characteristic-function}, the spectral density function $f(\lambda)=\frac{\sigma_\xi^2}{\pi}\cdot\frac{1-\cos(\lambda)}{1-\sum_{k\in\bZ}q(0,k)e^{ik\lambda}}$. From the proof of Lemma \ref{lmmA-3}, we can see that $f(\lambda)$ is infinitely differentiable and $f(0)=\frac{\sigma_\xi^2}{2\pi\sigma_1^2}>0$ (hence strictly positive). Let $P(z)=1$ and $w(\lambda)=f(\lambda)$ in Lemma \ref{lmm3-3}, we finish the proof of Theorem \ref{thm2-2}.
\end{proof}

\begin{proof}[Proof of Corollary \ref{crlly2-2}]
Proving $\pi_0$ to be a Gaussian field is trivial due to the fact that $\pi_0$ is non-degenerate, $\Delta_t(\cdot)$ is the limit of $\h_{[t-s,t]}(\cdot)-\h_{[t-s,t]}(\cdot-1)$ and $\h_{[t-s,t]}(\cdot)$ are jointly Gaussian distributed.\par

For Gaussian processes, the coefficients of complete linear regularity are equal to the coefficients of complete regularity (see page 249 in \cite{29}), i.e.
$$\rho(x)=\varrho(x), \quad\forall x\ge 0.$$
By Theorem \ref{thm2-2}, the proof is complete.
\end{proof}

\begin{proof}[Proof of Theorem \ref{thm2-3}]
%
%
%

Suppose there exist two invariant (by time) and ergodic (under spatial translations) distributions with same finite mean for the increment process $\{\eta_t\}_{t\in\bZ^+}$. Let us denote them by $\pi^1,\pi^2\in\cI \cap \cJ$. Then we can define two initial increments: $\pi^1$-distributed $\{\eta^1_0(x):x\in\bZ\}$ and $\pi^2$-distributed $\{\eta^2_0(x):x\in\bZ\}$. Let us assume that $\eta^1_0$ and $\eta^2_0$ are coupled in the way that the difference process $\{\eta^1_0(x)-\eta^2_0(x)\}_{x\in\bZ}$ is also ergodic, and the increment process $\{\eta^1_t(x): x\in\bZ\}_{t\in\bZ^+}$ and $\{\eta^2_t(x):x\in\bZ\}_{t\in \bZ^+}$ evolve from initial increments $\eta^1_0$ and $\eta^2_0$ respectively with the same noise $\{\xi_t(i): t\in\bN, i\in\bZ\}$. The existence of such coupling method can be proved by the following lemma.
\begin{lemma}\label{lmm3-4}
For $i=1,2$, let $\Omega_i$ be a complete separable metric space with Borel $\sigma$-algebra $\mathcal{F}_i$, and $T_i$ be a measurable transformation on $(\Omega_i, \mathcal{F}_i)$. Let us suppose that for $i=1,2$, $\nu_i$ is an ergodic invariant measure on $(\Omega_i,\mathcal{F}_i)$ w.r.t.\ $T_i$. Then, there exists an ergodic invariant measure $\mu$ on the product space $(\Omega_1\times \Omega_2, \mathcal{F}_1\otimes\mathcal{F}_2)$ $w.r.t.$ $T_1\times T_2$, such that for all $A\in \mathcal{F}_1$, $B\in\mathcal{F}_2$, 
$$\mu(A\times \Omega_2)=\nu_1(A), \mu(\Omega_1\times B)=\nu_2(B).$$
\end{lemma}
\begin{proof}
Note that the product measure $\nu=\nu_1\otimes\nu_2$ is an invariant measure on $(\Omega_1\times \Omega_2, \mathcal{F}_1\otimes\mathcal{F}_2)$ $w.r.t.$ $T_1\times T_2$. By the Ergodic Decomposition Theorem, there exists a probability measure $\rho_{\nu}$ on the set of ergodic measures $\mathcal{M}_e$ on  $(\Omega_1\times \Omega_2, \mathcal{F}_1\otimes\mathcal{F}_2)$ $w.r.t.$ $T_1\times T_2$, such that
$$\nu=\int_{\mathcal{M}_e}\pi\rho_{\nu}(d\pi),$$

Note that
\be\nu_1(\cdot)=\nu(\cdot\times \Omega_2)=\int_{\mathcal{M}_e}\pi(\cdot\times \Omega_2)\rho_{\nu}(d\pi),\label{ergodic-decomp-1}\ee
\be\nu_2(\cdot)=\nu(\Omega_1\times \cdot)=\int_{\mathcal{M}_e}\pi(\Omega_1\times \cdot)\rho_{\nu}(d\pi).\label{ergodic-decomp-2}\ee
And one can easily show that for all $\pi\in\mathcal{M}_e$, the marginals $\pi(\cdot\times \Omega_2)$ and $\pi(\Omega_1\times \cdot)$ are also ergodic. Thus, \eqref{ergodic-decomp-1} and \eqref{ergodic-decomp-2} are in fact ergodic decomposition of $\nu_1$ and $\nu_2$ respectively. Since $\nu_1$ and $\nu_2$ are ergodic, we have for $\rho_\nu$-almost every $\pi\in \supp(\rho_\nu)$,
\be \pi(\cdot\times \Omega_2)\equiv \nu_1(\cdot), \pi(\Omega_1\times \cdot)\equiv \nu_2(\cdot),\label{er1}\ee
where $\supp(\rho_\nu)=\{\pi\in \mathcal{M}_e: \mbox{ for }\forall\mbox{ open neighbourhood }N_\pi\subseteq \mathcal{M}_e\mbox{ of }\pi, \rho_\nu(N_\pi)>0\}$.\par

The proof for Lemma \ref{lmm3-4} is complete by picking up $\mu$ from $\supp(\rho_\nu)$ such that \eqref{er1} holds.
\end{proof}

Using \eqref{21-l1} and \eqref{22-0}, the increment processes $\eta^1_t$ and $\eta^2_t$ can have the following expressions
\be\eta^i_t(j)=\sum_{x\in\bZ}p^{t}(j,x)\eta^i_0(x)+\sum_{k=1}^t\sum_{x\in \bZ}p^{t-k}(j,x)\left[\xi_k(x)-\xi_k(x-1)\right],\quad j\in\bZ,t\in\bZ^+,i=1,2.\label{u-4}\ee

Let us denote the difference of the two increment processes by $\zeta_t(\cdot)$, i.e. 
\be \zeta_t(i)=\eta^1_t(i)-\eta^2_t(i), \quad i\in \bZ, t\in\bZ^+,\ee 
and the underlying ergodic distribution of $\{\zeta_0(i): i\in\bZ\}$ as $\nu$. Notice that $\bE^\nu\left[\zeta_0(x)\right]=0$ due to the assumption $\bE^{\pi^1}\left[\eta^1_0(x)\right]=\bE^{\pi^2}\left[\eta^2_0(x)\right]$.\par

From \eqref{u-4}, 
\be \zeta_t(i)=\sum_{x\in\bZ}p^{t}(i,x)\zeta_0(x), \quad i\in \bZ, t\in\bZ^+.\ee


For $x\in\bZ$, $t\in\bZ^+$, and $\zeta\in \bR^\bZ$, let us set 
$$\zeta^r=\{\zeta^r(i)=(-r)\vee(\zeta(i)\wedge r)\}_{i\in\bZ}, \quad r>0,$$
$$g_t(x,\zeta)=\sum_{y\in\bZ}p^{t}(x,y)\zeta(y), \quad g_t(x,\zeta, r)=\sum_{y\in\bZ}p^{t}(x,y)\zeta^r(y).$$
and the characteristic function of the transition $p$
$$\phi_X(\theta)=\sum_{y\in\bZ}p(0,y)e^{\i\theta y}, \quad \theta\in\bR.$$


First, we will show that for every fixed $r>0$, $g_t(x,\zeta,r)$ converges to a constant in $L^2(\nu)$ as $t\to\infty$. Then, we will prove that such convergence implies the convergence of $g_t(x,\zeta)$ in $L^1(\nu)$.\par

Note that the covariance $\covnu^r(x)=\bE^\nu[\zeta^r(0)\zeta^r(x)]$ is a positive definite sequence (i.e.\ $\sum_{x,y}\covnu^r(x-y)z_x\overline{z_y}\ge 0$, for any choice of finitely many complex numbers $\{z_n\}$). By Herglotz' Theorem (see Chapter \RNum{19}.6 in \cite{30}), there exists a bounded measure $\gamma^r$ on $[-\pi, \pi)$ such that
$$\covnu^r(x)=\int e^{-\i x\theta}\gamma^r(d\theta), \quad x\in\bZ.$$ 

Let $\{X_t\}_{t\in\bZ^+}$ and $\{\tilde{X}_t\}_{t\in\bZ^+}$ be two i.i.d.\ copies of the random walk with transition probability $p$. We use them to compute the covariance of $g_t(x,\zeta,r)$ and $g_s(x,\zeta,r)$ under measure $\nu$.
\begin{align}
\int g_t(x,\zeta,r)g_s(x,\zeta,r)\nu(d\zeta)=& \int \mE^{x}[\zeta^r(X_t)]\mE^{x}[\zeta^r(\tilde{X}_s)]\nu(d\zeta)\notag\\
=& \mE^{(x,x)}\int \zeta^r(X_t) \zeta^r(\tilde{X}_s) \nu(d\zeta)=\mE^{(x,x)} \left[\covnu^r(\tilde{X}_s-X_t)\right]\notag\\
=&\mE^{(x,x)}\int e^{-\i\theta(\tilde{X}_s-X_t)}\gamma^r(d\theta)=\int \overline{\mE^x e^{\i\theta\tilde{X}_s}}\cdot\mE^x e^{\i\theta X_t}\gamma^r(d\theta)\notag\\
=&\int [\overline{\phi_X(\theta)}]^s[\phi_X(\theta)]^t\gamma^r(d\theta).\label{cov-g-result1}
\end{align}
If we switch the position of $s$ and $t$ above, we can further get 
\be\int [\overline{\phi_X(\theta)}]^s[\phi_X(\theta)]^t\gamma^r(d\theta)=\int [\phi_X(\theta)]^s[\overline{\phi_X(\theta)}]^t\gamma^r(d\theta).\label{cov-g-result2}\ee

Apply \eqref{cov-g-result1} and \eqref{cov-g-result2}, we can get 
\begin{align}
&\int \left[g_t(x,\zeta,r)-g_s(x,\zeta,r)\right]^2 \nu(d\zeta) = \int \left[g_t(x,\zeta,r)^2-2g_t(x,\zeta,r)g_s(x,\zeta,r)+g_s(x,\zeta,r)^2\right] \nu(d\zeta) \notag\\
&=\int \left[\left|\phi_X(\theta)\right|^{2t}-2\overline{\phi_X(\theta)^s}\phi_X(\theta)^t+\left|\phi_X(\theta)\right|^{2s}\right]\gamma^r(d\theta)=\int \left |\phi_X(\theta)^t-\phi_X(\theta)^s\right|^2 \gamma^r(d\theta)\notag\\
&=\int_{\theta\neq 0} \left |\phi_X(\theta)^t-\phi_X(\theta)^s\right|^2 \gamma^r(d\theta).\label{u-6}
\end{align}

Notice that \eqref{weight-assump-2} makes sure that $|\phi_X(\theta)|<1$, $\forall \theta\in [-\pi,\pi)\setminus \{0\}$ (Lemma \ref{lmmB-1} in the Appendix). Thus, the integrand $\left |\phi_X(\theta)^t-\phi_X(\theta)^s\right|^2$ in \eqref{u-6} will converge to zero as $s, t\to\infty$ for $\theta\in [-\pi,\pi)\setminus \{0\}$. From Bounded Convergence Theorem, we may conclude that for any fixed $x\in\bZ$, $r>0$, $\{g_t(x,\zeta,r)\}_{t\in\bZ^+}$ is a Cauchy sequence in $L^2(\nu)$. Hence, there exists a $L^2(\nu)$ limit
\be g(x,\zeta,r)=\lim_{t\to\infty}g_t(x, \zeta,r),\quad x\in\bZ.\ee 

Next, we will prove that $g(x,\zeta,r)$ is nothing but a constant function of $x$ for $\nu$-almost every fixed $\zeta$.
\begin{lemma}\label{lmm3-5}
Under the conditions in Theorem \ref{thm2-3}, for all fixed $r>0$ and $\nu$-almost every fixed $\zeta$, there exists a constant $C(\zeta,r)$ such that
$$g(x,\zeta,r)\equiv C(\zeta,r),\mbox{ for all }x\in\bZ.$$
\end{lemma}
\begin{proof}
Notice that 
$$\abs{g_t(x,\zeta,r)}\le \sum_{y\in\bZ}p^{t}(x,y)\abs{\zeta^r(y)}\le r, \quad t\in\bZ^+, x\in\bZ, \zeta\in \bR^\bZ.$$ 
Thus, for $\nu$-almost every fixed $\zeta$,
$$\abs{g(x,\zeta,r)}\le r,\quad x\in\bZ.$$
Also, letting $s\to\infty$ in $g_{s+t}(x,\zeta, r)=\sum_{y\in\bZ}p^{t}(x,y)g_s(y,\zeta, r)$ shows that 
\be g(x,\zeta, r)=\sum_{y\in\bZ}p^{t}(x,y)g(y,\zeta, r),\quad t\in\bZ^+.\label{harmonic-def}\ee
functions with property \eqref{harmonic-def} are called $p$-harmonic. So far we have shown that $g(x,\zeta,r)$ is a bounded $p$-harmonic function w.r.t.\ $x$. The proof is complete by the following lemma.
\end{proof}
\begin{lemma}\label{bounded-harmonic-function}
Assume \eqref{weight-assump-1} and \eqref{weight-assump-2}. Bounded $p$-harmonic functions are constants. 
\end{lemma}
\begin{proof}
Suppose $h(x)$ is a $p$-harmonic function, i.e.\ $h(x)=\sum_{z\in\bZ}p(x,z)h(z)$, $x\in\bZ$. If $p(x,y)>0$, one can use the coupling described on page 69 of \cite{40} to show that $h(x)=h(y)$. If $p(x,y)=0$, from assumption \eqref{weight-assump-2}, we can find a path $x=x_0,x_1,x_2,\ldots,x_{m-1},x_m=y$ on $\bZ$ such that $p(x_i,x_{i+1})+p(x_{i+1},x_i)>0$, $i=0,1,\ldots,m-1$, and hence $h(x)=h(x_1)=\cdots=h(x_{m-1})=h(y)$.   
\end{proof}

Now we have shown that for $\nu$-almost every fixed $\zeta$, the limit $g(x,\zeta,r)$ is independent of $x$. Then we look at $g_t(x,\zeta)$.
\begin{align*}
\bE^\nu\left|g_t(x,\zeta)-g_s(x,\zeta)\right|\le &\bE^\nu\left|g_t(x,\zeta)-g_t(x,\zeta,r)\right|+\bE^\nu\left|g_s(x,\zeta)-g_s(x,\zeta,r)\right|\\
&\quad\quad+\bE^\nu\left|g_t(x,\zeta,r)-g_s(x,\zeta,r)\right|\\
\le & 2\bE^\nu\left|\zeta(0)-\zeta^r(0)\right|+\left\{\bE^\nu\left[g_t(x,\zeta,r)-g_s(x,\zeta,r)\right]^2\right\}^{1/2}.
\end{align*}
From the finite first moment assumption on $\pi^1$ and $\pi^2$, $\lim_{r\to\infty}\bE^\nu\left|\zeta(0)-\zeta^r(0)\right|=0$. Thus, $\{g_t(x,\cdot)\}$ is a Cauchy sequence in $L^1(\nu)$. Let us denote the limit
\be g(x,\zeta)=\lim_{t\to\infty}g_t(x,\zeta), \quad x\in\bZ.\ee

Since
 $$\bE^{\nu}\left|g(x,\zeta)-g(x,\zeta,r)\right|\le \bE^{\nu}\left|g(x,\zeta)-g_t(x,\zeta)\right|+\bE^\nu\left|\zeta(0)-\zeta^r(0)\right|+\bE^{\nu}\left|g_t(x,\zeta,r)-g(x,\zeta,r)\right|.$$
Letting $t\to\infty$ above shows that
\be\bE^{\nu}\left|g(x,\zeta)-g(x,\zeta,r)\right|\le\bE^\nu\left|\zeta(0)-\zeta^r(0)\right|\to 0,\quad\mbox{as }r\to \infty.\ee
This implies that for $\nu$-almost every fixed $\zeta$, 
\be g(0,\zeta)=g(x,\zeta),\quad \forall x\in\bZ.\label{constant-harmonic-function-g}\ee

On the other hand, by the translation-invariant property of $p^t(\cdot, \cdot)$,
$$g_t(x,\zeta)=\sum_{y\in\bZ}p^{t}(x,y)\zeta(y)=\sum_{y\in\bZ}p^{t}(0,y-x)\zeta(y)=\sum_{z\in\bZ}p^{t}(0,z)\zeta(z+x)=g_t(0,\theta_x \zeta).$$
Letting $t\to\infty$ leads to
\be g(x,\zeta)=g(0,\theta_x \zeta), \quad x\in\bZ,\quad\nu-a.s.\ee 
Combine this with \eqref{constant-harmonic-function-g}, we see that for $\nu$-almost all $\zeta$,
\be g(0,\zeta)=g(0,\theta_x \zeta),\quad\forall x\in\bZ.\label{u-10}\ee

By the ergodicity of $\nu$, we have
\be g(0,\zeta)=\int g(0,\zeta) \nu(d\zeta)=\lim_{t\to\infty}\int g_t(0,\zeta) \nu(d\zeta)=0,\quad \nu-a.s.\ee
Note that the second equation is due to the convergence of $g_t(0,\zeta)$ in $L^1(\nu)$.\par

Recall that $g_t(x,\zeta_0)=\zeta_t(x)=\eta^1_t(x)-\eta^2_t(x)$. Thus, we have proved that \be\eta^1_t(x)-\eta^2_t(x)\overset{L^1(\nu)}{\to}0.\ee

For any finite set $\Lambda=\{x_1,x_2,\ldots,x_m\}\subset \bZ$, let function $f: \bR^m\to\bR$ be any bounded Lipschitz function. We have
\begin{align*}
&\left|\bE^{\pi^1}f(\eta(\Lambda))-\bE^{\pi^2}f(\eta(\Lambda))\right|\le \bE^\nu\abs{f(\eta^1_t(\Lambda))-f(\eta^2_t(\Lambda))}\le C \bE^\nu \left\{\Bigl[\sum_{i=1}^m\bigl(\eta^1_t(x_i)-\eta^2_t(x_i)\bigr)^2\Bigr]^{1/2}\right\}\\
\le& C\sum_{i=1}^m\bE^\nu\bigl|\eta^1_t(x_i)-\eta^2_t(x_i)\bigr|\to 0,\quad\mbox{as }t\to\infty.
\end{align*}
This implies that the marginal distributions of $\pi^1$ and $\pi^2$ on $\Lambda$ are the same. And, thus, $\pi^1=\pi^2$ which contradicts the assumption that $\pi^1$ and $\pi^2$ are two different probability measures.\par

One can easily check that $\pi_c\in\cI\cap\cJ$ and it is ergodic with mean $c$. Thus, the proof of Theorem \ref{thm2-3} is complete.
\end{proof}

\section{Proofs of the distributional limits}\label{fluctuation-limit}
\subsection{Convergence of finite-dimensional distributions}
\begin{proof}[Proof of Theorem \ref{thm2-4}]
Let us define 
\begin{align}
&\overline{H}_n(t,r)=n^{-1/4}\left(\mE (X_{\lfloor nt\rfloor}^{y(n)})-r\sqrt{n}\right);\label{2-4}\\
&\overline{S}_n(t,r)=n^{-1/4}\sum_{i\in\mathbb{Z}}\bigl(\eta_0(i)-\mu_0\bigr)\left\{\ind_{\{i>0\}}\mP\bigl(i\le X_{\lfloor nt\rfloor}^{y(n)}\bigr)-\ind_{\{i\le 0\}}\mP\bigl(i>X_{\lfloor nt\rfloor}^{y(n)}\bigr)\right\};\label{2-5}\\
&\overline{F}_n(t,r)=n^{-1/4}\sum_{k=1}^{\lfloor nt\rfloor}\mE\left[\xi_k(X_{\lfloor nt\rfloor -k}^{y(n)})\right].\label{2-6}
\end{align}
Then we can rewrite $\fluc_n(t,r)$ as
\begin{lemma}\label{lmm3-7}
\be
\fluc_n(t,r)=\mu_0\overline{H}_n(t,r)+\overline{S}_n(t,r)+\overline{F}_n(t,r).\label{2-7}
\ee
\end{lemma}
\begin{proof}
From \eqref{2-l1}, we just need to show that $n^{-1/4}\left\{\mE\left[\h_0(X_{\lfloor nt\rfloor}^{y(n)})\right]-\mu_0r\sqrt{n}\right\}=\mu_0\overline{H}_n(t,r)+\overline{S}_n(t,r)$.
\begin{align*}
&n^{-1/4}\left\{\mE\left[\h_0(X_{\lfloor nt\rfloor}^{y(n)})\right]-\mu_0r\sqrt{n}\right\}\\
=&n^{-1/4}\left\{\mE\left[\ind_{\{X_{\lfloor nt\rfloor}^{y(n)}>0\}}\sum_{i=1}^{X_{\lfloor nt\rfloor}^{y(n)}}\eta_0(i)-\ind_{\{X_{\lfloor nt\rfloor}^{y(n)}<0\}}\sum_{i=X_{\lfloor nt\rfloor}^{y(n)}+1}^0\eta_0(i)\right]-\mu_0r\sqrt{n}\right\}\\
=&n^{-1/4}\left\{\sum_{i>0}\eta_0(i)\mP\left(i\le X_{\lfloor nt\rfloor}^{y(n)}\right)-\sum_{i\le 0}\eta_0(i)\mP\left(i>X_{\lfloor nt\rfloor}^{y(n)}\right)-\mu_0r\sqrt{n}\right\}\\
=&\mu_0\overline{H}_n(t,r)+\overline{S}_n(t,r).
\end{align*}
The last equality can be reached by adding and subtracting $\mu_0$ from each term and doing some rearrangements.
\end{proof}
Note that
$$\mE(X_{\lfloor nt\rfloor}^{y(n)})=\wm\lfloor nt\rfloor +y(n)=\lfloor r\sqrt{n}\rfloor +O(1).$$
Thus, 
\be\mu_0\overline{H}_n(t,r)=O(n^{-1/4}),\label{H-uniform-bound}\ee
and $\lim_{n\to\infty}\mu_0\overline{H}_n(t,r)=0$ uniformly over $(t,r)$.\par

For $\overline{S}_n(t,r)$ and $\overline{F}_n(t,r)$, they are independent and we will treat them separately in Lemma \ref{lmm3-8} and Lemma \ref{lmm3-10}. We start with $\overline{S}_n$.\par

Let $\{S(t,r):t\in\bR^+,r\in\bR\}$ be a mean-zero Gaussian process with the following covariance:
\be\bE\bigl[S(t,r)S(s,q)\bigr]=\varsigma^2\Gamma_2\bigl((t,r),(s,q)\bigr),\quad t,s\in\bR^+, r,q\in\bR.\label{2-t6a}\ee
Then, for $\overline{S}_n(t,r)$, we have
\begin{lemma}\label{lmm3-8}
Under the conditions in Theorem \ref{thm2-4}, $\{\overline{S}_n(t,r)\}_{t\in\bR^+,r\in\bR}$ will converge weakly (in the sense of finite dimensional distributions) to the Gaussian process $\{S(t,r)\}_{t\in\bR^+,r\in\bR}$ as $n\to\infty$. 
\end{lemma}

\begin{remark}\label{rmk3-1}
In the proof of the above lemma, we use the following alternative definition of $\Gamma_2$:
\begin{align}
\Gamma_2\bigl((s,q),(t,r)\bigr)=&\int_{-\infty}^0\bP(B_{\sigma_1^2s}>q-x)\bP(B_{\sigma_1^2t}>r-x)dx\notag\\
&\quad\quad+\int_0^\infty \bP(B_{\sigma_1^2s}\le q-x)\bP(B_{\sigma_1^2t}\le r-x)dx,\label{2-r7}
\end{align}
where $\{B_t\}_{t\in\bR^+}$ is a standard 1-dimensional Brownian motion.
\end{remark}

\begin{proof}
Notice that in order to show the convergence of finite-dimensional distributions of $\overline{S}_n(\cdot,\cdot)$, we only need to show that for each fixed $N\in\bN$, $\{(t_j,r_j)\in \bR^+\times \bR:j=1,\ldots,N\}$ and $\{\theta_j\in \bR:j=1,\ldots,N\}$, we have 
$$\sum_{j=1}^N\theta_j \overline{S}_n(t_j,r_j)\Rightarrow \sum_{j=1}^N\theta_j S(t_j,r_j),\quad \mbox{as }n\to\infty.$$ 

Note that 
\begin{align*}
\sum_{j=1}^N\theta_j \overline{S}_n(t_j,r_j)=&n^{-1/4}\sum_{j=1}^N\theta_j\sum_{i\in\mathbb{Z}}\bigl(\eta_0(i)-\mu_0\bigr)\Bigl\{\ind_{\{i>0\}}\mP\bigl(i\le X_{\lfloor nt_j\rfloor}^{\lfloor nt_jb\rfloor +\lfloor r_j\sqrt{n}\rfloor}\bigr)\\
&\quad\quad\quad-\ind_{\{i\le 0\}}\mP\bigl(i>X_{\lfloor nt_j\rfloor}^{\lfloor nt_jb\rfloor +\lfloor r_j\sqrt{n}\rfloor}\bigr)\Bigr\}\\
=&n^{-1/4}\sum_{i\in\mathbb{Z}}\bigl(\eta_0(i)-\mu_0\bigr)\sum_{j=1}^N\theta_j\Bigl\{\ind_{\{i>0\}}\mP\bigl(i\le X_{\lfloor nt_j\rfloor}^{\lfloor nt_jb\rfloor +\lfloor r_j\sqrt{n}\rfloor}\bigr)\\
&\quad\quad\quad-\ind_{\{i\le 0\}}\mP\bigl(i>X_{\lfloor nt_j\rfloor}^{\lfloor nt_jb\rfloor +\lfloor r_j\sqrt{n}\rfloor}\bigr)\Bigr\}.
\end{align*}

Let us denote 
\be a_{n,i}=n^{-1/4}\left\{\ind_{\{i>0\}}\sum_{j=1}^N\theta_j\mP\bigl(i\le X_{\lfloor nt_j\rfloor}^{\lfloor nt_jb\rfloor +\lfloor r_j\sqrt{n}\rfloor}\bigr)-\ind_{\{i\le 0\}}\sum_{j=1}^N\theta_j\mP\bigl(i>X_{\lfloor nt_j\rfloor}^{\lfloor nt_jb\rfloor +\lfloor r_j\sqrt{n}\rfloor}\bigr)\right\}.\label{mid-2}\ee

Then, 
\be\sum_{j=1}^N\theta_j \overline{S}_n(t_j,r_j)=\sum_{i\in\mathbb{Z}}a_{n,i}\bigl(\eta_0(i)-\mu_0\bigr). \label{mid-2b}\ee

Let us consider the three cases in Theorem \ref{thm2-4} separately.\par

Case (a):   If $\eta_0(x)$'s are i.i.d., let $\ell(n)$ be any increasing function of $n$ such that $\lim_{n\to\infty}\ell(n)=\infty$, we will show that only $\ell(n)\sqrt{n}$ number of terms matters in the above summation \eqref{mid-2b}. To be specific,
\begin{lemma}\label{lmm3-9}
\be\lim_{n\to \infty}\bE \Bigl|\sum_{|i|> \ell(n)\sqrt{n}}a_{n,i}\bigl(\eta_0(i)-\mu_0\bigr)\Bigr|^2=0.\label{mid-1}\ee
\end{lemma}
\begin{proof}
Notice that 
\begin{align*}
&\bE \Bigl|\sum_{|i|> \ell(n)\sqrt{n}}a_{n,i}\bigl(\eta_0(i)-\mu_0\bigr)\Bigr|^2= \sum_{|i|> \ell(n)\sqrt{n}} \bE \left[a_{n,i}^2\bigl(\eta_0(i)-\mu_0\bigr)^2\right]\\
&=n^{-1/2}\sigma_0^2\Biggl\{\sum_{i<-\ell(n)\sqrt{n}}\biggl[\sum_{j=1}^N\theta_j\mP\bigl(i>X_{\lfloor nt_j\rfloor}^{\lfloor nt_jb\rfloor +\lfloor r_j\sqrt{n}\rfloor}\bigr)\biggr]^2\\
&\quad\quad\quad\quad\quad\quad\quad\quad\quad\quad\quad\quad\quad\quad+\sum_{i>\ell(n)\sqrt{n}}\biggl[\sum_{j=1}^N\theta_j\mP\bigl(i\le X_{\lfloor nt_j\rfloor}^{\lfloor nt_jb\rfloor +\lfloor r_j\sqrt{n}\rfloor}\bigr)\biggr]^2\Biggr\}\\
&\le C n^{-1/2}\sum_{j=1}^N \theta_j^2\left[\sum_{i<-\ell(n)\sqrt{n}}\mP\bigl(i>X_{\lfloor nt_j\rfloor}^{\lfloor nt_jb\rfloor +\lfloor r_j\sqrt{n}\rfloor}\bigr)+\sum_{i>\ell(n)\sqrt{n}}\mP\bigl(i\le X_{\lfloor nt_j\rfloor}^{\lfloor nt_jb\rfloor +\lfloor r_j\sqrt{n}\rfloor}\bigr)\right].
\end{align*}

By standard large deviation theory, for any $\epsilon>0$, there exist constants $K_j>0$, $j=1,\ldots,N$ such that when $i<\lfloor r_j\sqrt{n}\rfloor$,
\be\mP\bigl(i>X_{\lfloor nt_j\rfloor}^{\lfloor nt_jb\rfloor +\lfloor r_j\sqrt{n}\rfloor}\bigr)\le\begin{cases} \exp\{-K_j(i-\lfloor r_j\sqrt{n}\rfloor)^2/nt_j\} & \text{if } \abs{i-\lfloor r_j\sqrt{n}\rfloor}\le nt_j\epsilon,\\  \exp\{-K_j\abs{i-\lfloor r_j\sqrt{n}\rfloor}\} & \text{if } \abs{i-\lfloor r_j\sqrt{n}\rfloor}> nt_j\epsilon,\end{cases}\label{ldp1}\ee
and when $i> \lfloor r_j\sqrt{n}\rfloor$,
\be\mP\bigl(i\le X_{\lfloor nt_j\rfloor}^{\lfloor nt_jb\rfloor +\lfloor r_j\sqrt{n}\rfloor}\bigr)\le\begin{cases} \exp\{-K_j(i-\lfloor r_j\sqrt{n}\rfloor)^2/nt_j\} & \text{if } \abs{i-\lfloor r_j\sqrt{n}\rfloor}\le nt_j\epsilon,\\  \exp\{-K_j\abs{i-\lfloor r_j\sqrt{n}\rfloor}\} & \text{if } \abs{i-\lfloor r_j\sqrt{n}\rfloor}> nt_j\epsilon.\end{cases}\label{ldp2}\ee

Hence, we can further bound the second moment of $\sum_{|i|> \ell(n)\sqrt{n}}a_{n,i}\bigl(\eta_0(i)-\mu_0\bigr)$ by
\begin{align*}
&\bE \Bigl|\sum_{|i|> \ell(n)\sqrt{n}}a_{n,i}\bigl(\eta_0(i)-\mu_0\bigr)\Bigr|^2 \le  Cn^{-1/2}\sum_{j=1}^N \theta_j^2\biggl[\sum_{m\in I_1(j)}e^{-K_jm^2/nt_j}+\sum_{m\in I_2(j)}e^{-K_j\abs{m}}\biggr]\\
&\le C \sum_{j=1}^N\theta_j^2\biggl[\int_{-\infty}^{-\ell(n)-r_j+1}e^{-K_jx^2/t_j}dx+\int_{\ell(n)-r_j-1}^{\infty}e^{-K_jx^2/t_j}dx+\frac{1}{\sqrt{n}}e^{-K_jnt_j\epsilon}\biggr]\to 0,\mbox{ as }n\to\infty,
\end{align*}
where $I_1(j)=[-nt_j\epsilon, -\ell(n)\sqrt{n}-\lfloor r_j\sqrt{n}\rfloor)\cup (\ell(n)\sqrt{n}-\lfloor r_j\sqrt{n}\rfloor,nt_j\epsilon]$, and $I_2(j)=(-\infty,-nt_j\epsilon) \cup (nt_j\epsilon,\infty)$.\par

Thus, the proof for Lemma \ref{lmm3-9} is complete.
\end{proof}

And for the main part $\sum_{|i|\le \ell(n)\sqrt{n}}a_{n,i}\bigl(\eta_0(i)-\mu_0\bigr)$, we will use the Lindeberg-Feller Central Limit Theorem to show the convergence.
\begin{theorem}\label{LF} (Lindeberg-Feller) For each $n>0$, assume that $\{X_{n,j}, j=1,2,\ldots,J(n)\}$ are independent, mean-zero, square-integrable random variables, and let $T_n=\sum_{j=1}^{J(n)}X_{n,j}$. Let us suppose the following two conditions hold:
\begin{enumerate}
\item $\lim_{n\to\infty} \sum_{j=1}^{J(n)} \bE(X_{n,j}^2)=\sigma^2$;
\item for all $\epsilon>0$, $\lim_{n\to\infty}\sum_{j=1}^{J(n)}\bE\left(X_{n,j}^2 \ind\{|X_{n,j}|\ge\epsilon\}\right)=0.$
 \end{enumerate}
Then, $T_n$ will converge weakly to a Gaussian random variable with mean $0$ and variance $\sigma^2$. 
\end{theorem}

Now let us first check the limit of $\bE\left[\sum_{|i|\le \ell(n)\sqrt{n}}a_{n,i}\bigl(\eta_0(i)-\mu_0\bigr)\right]^2$. Notice that
\begin{align}
&\bE\left[\sum_{|i|\le \ell(n)\sqrt{n}}a_{n,i}\bigl(\eta_0(i)-\mu_0\bigr)\right]^2 \notag\\
=& n^{-1/2}\sigma_0^2\sum_{-\ell(n)\sqrt{n}\le i \le 0}\left[\sum_{j_1,j_2=1}^N\theta_{j_1}\theta_{j_2}\mP\bigl(i>X_{\lfloor nt_{j_1}\rfloor}^{\lfloor nt_{j_1}b\rfloor +\lfloor r_{j_1}\sqrt{n}\rfloor}\bigr)\mP\bigl(i>X_{\lfloor nt_{j_2}\rfloor}^{\lfloor nt_{j_2}b\rfloor +\lfloor r_{j_2}\sqrt{n}\rfloor}\bigr)\right]\label{slim-1}\\
+& n^{-1/2}\sigma_0^2\sum_{0 < i \le \ell(n)\sqrt{n}}\left[\sum_{j_1,j_2=1}^N\theta_{j_1}\theta_{j_2}\mP\bigl(i\le X_{\lfloor nt_{j_1}\rfloor}^{\lfloor nt_{j_1}b\rfloor +\lfloor r_{j_1}\sqrt{n}\rfloor}\bigr)\mP\bigl(i\le X_{\lfloor nt_{j_2}\rfloor}^{\lfloor nt_{j_2}b\rfloor +\lfloor r_{j_2}\sqrt{n}\rfloor}\bigr)\right].\label{slim-2}
\end{align}
Let us consider the first part \eqref{slim-1}. Let $\STail>0$ be any fixed positive number such that $\STail<\ell(n)$. We can further break \eqref{slim-1} into two parts.
\begin{align}
 &n^{-1/2}\sigma_0^2\sum_{-\ell(n)\sqrt{n}\le i \le 0}\left[\sum_{j_1,j_2=1}^N\theta_{j_1}\theta_{j_2}\mP\bigl(i>X_{\lfloor nt_{j_1}\rfloor}^{\lfloor nt_{j_1}b\rfloor +\lfloor r_{j_1}\sqrt{n}\rfloor}\bigr)\mP\bigl(i>X_{\lfloor nt_{j_2}\rfloor}^{\lfloor nt_{j_2}b\rfloor +\lfloor r_{j_2}\sqrt{n}\rfloor}\bigr)\right]\notag\\
&= n^{-1/2}\sigma_0^2\sum_{-\STail\sqrt{n}\le i \le 0}\left[\sum_{j_1,j_2=1}^N\theta_{j_1}\theta_{j_2}\mP\bigl(i>X_{\lfloor nt_{j_1}\rfloor}^{\lfloor nt_{j_1}b\rfloor +\lfloor r_{j_1}\sqrt{n}\rfloor}\bigr)\mP\bigl(i>X_{\lfloor nt_{j_2}\rfloor}^{\lfloor nt_{j_2}b\rfloor +\lfloor r_{j_2}\sqrt{n}\rfloor}\bigr)\right]\label{slim-3}\\
&+ n^{-1/2}\sigma_0^2\sum_{-\ell(n)\sqrt{n}\le i < -\STail\sqrt{n}}\left[\sum_{j_1,j_2=1}^N\theta_{j_1}\theta_{j_2}\mP\bigl(i>X_{\lfloor nt_{j_1}\rfloor}^{\lfloor nt_{j_1}b\rfloor +\lfloor r_{j_1}\sqrt{n}\rfloor}\bigr)\mP\bigl(i>X_{\lfloor nt_{j_2}\rfloor}^{\lfloor nt_{j_2}b\rfloor +\lfloor r_{j_2}\sqrt{n}\rfloor}\bigr)\right].\label{slim-4}
\end{align}

For \eqref{slim-3}, we can rewrite it into integral form,
\begin{equation}
\label{slim-5}
\begin{aligned}
&n^{-1/2}\sigma_0^2\sum_{-\STail\sqrt{n}\le i \le 0}\left[\sum_{j_1,j_2=1}^N\theta_{j_1}\theta_{j_2}\mP\bigl(i>X_{\lfloor nt_{j_1}\rfloor}^{\lfloor nt_{j_1}b\rfloor +\lfloor r_{j_1}\sqrt{n}\rfloor}\bigr)\mP\bigl(i>X_{\lfloor nt_{j_2}\rfloor}^{\lfloor nt_{j_2}b\rfloor +\lfloor r_{j_2}\sqrt{n}\rfloor}\bigr)\right]\\
&= \sigma_0^2\sum_{j_1,j_2=1}^N\theta_{j_1}\theta_{j_2}\int_{-\STail-1}^0\ind_{\{x\ge -(\lfloor \STail\sqrt{n}\rfloor +1)/\sqrt{n}\}}\Bigl[\mP\bigl(\lceil \sqrt{n}x\rceil /\sqrt{n}>X_{\lfloor nt_{j_1}\rfloor}^{\lfloor nt_{j_1}b\rfloor +\lfloor r_{j_1}\sqrt{n}\rfloor}/\sqrt{n}\bigr)\\
&\quad\quad\quad\quad\quad\quad\quad\quad\quad\quad\quad\cdot \mP\bigl(\lceil \sqrt{n}x\rceil /\sqrt{n}>X_{\lfloor nt_{j_2}\rfloor}^{\lfloor nt_{j_2}b\rfloor +\lfloor r_{j_2}\sqrt{n}\rfloor}/\sqrt{n}\bigr)\Bigr]dx.
\end{aligned}
\end{equation}

Notice that from CLT, we have $\bigl(X_{\lfloor nt\rfloor}^{\lfloor ntb\rfloor +\lfloor r\sqrt{n}\rfloor}-r\sqrt{n}\bigr)/\sqrt{n}\Rightarrow B_{\sigma_1^2t}$. Thus, let $n\to\infty$ in \eqref{slim-5} and use Bounded Convergence Theorem, we have
\begin{equation}
\label{slim-6}
\begin{aligned}
&\lim_{n\to\infty}n^{-1/2}\sigma_0^2\sum_{-\STail\sqrt{n}\le i \le 0}\left[\sum_{j_1,j_2=1}^N\theta_{j_1}\theta_{j_2}\mP\bigl(i>X_{\lfloor nt_{j_1}\rfloor}^{\lfloor nt_{j_1}b\rfloor +\lfloor r_{j_1}\sqrt{n}\rfloor}\bigr)\mP\bigl(i>X_{\lfloor nt_{j_2}\rfloor}^{\lfloor nt_{j_2}b\rfloor +\lfloor r_{j_2}\sqrt{n}\rfloor}\bigr)\right]\\
&= \sigma_0^2\sum_{j_1,j_2=1}^N\theta_{j_1}\theta_{j_2}\int_{-\STail}^0\Bigl[\bP\left(B_{\sigma_1^2t_{j_1}}<x-r_{j_1}\right)\bP\left(B_{\sigma_1^2t_{j_2}}<x-r_{j_2}\right)\Bigr]dx.
\end{aligned}
\end{equation}

For the remaining part in \eqref{slim-4}, we will show that it is negligible as $M$ goes to $\infty$. Recall from \eqref{ldp1}, suppose $M>\max_{j}\{\abs{r_j}\}$. Then,
\begin{align}
&n^{-1/2}\sigma_0^2\sum_{-\ell(n)\sqrt{n}\le i < -\STail\sqrt{n}}\left[\sum_{j_1,j_2=1}^N\abs{\theta_{j_1}}\cdot\abs{\theta_{j_2}}\mP\bigl(i>X_{\lfloor nt_{j_1}\rfloor}^{\lfloor nt_{j_1}b\rfloor +\lfloor r_{j_1}\sqrt{n}\rfloor}\bigr)\mP\bigl(i>X_{\lfloor nt_{j_2}\rfloor}^{\lfloor nt_{j_2}b\rfloor +\lfloor r_{j_2}\sqrt{n}\rfloor}\bigr)\right]\notag\\
&\le Cn^{-1/2}\sum_{j=1}^N\sum_{i < -\STail\sqrt{n}}\mP\bigl(i>X_{\lfloor nt_j\rfloor}^{\lfloor nt_jb\rfloor +\lfloor r_j\sqrt{n}\rfloor}\bigr)\notag\\
&\le Cn^{-1/2}\sum_{j=1}^N\left[\sum_{\lfloor r_j \sqrt{n}\rfloor-nt_j\epsilon\le i < -\STail\sqrt{n}} e^{-K_j(i-\lfloor r_j\sqrt{n}\rfloor)^2/nt_j}+\sum_{i<\lfloor r_j \sqrt{n}\rfloor-nt_j\epsilon}e^{-K_j\abs{i-\lfloor r_j\sqrt{n}\rfloor}}\right]\notag\\
&\le C \sum_{j=1}^N \left[\int_{-\infty}^{-M-r_j+1}e^{-K_jx^2/t_j}dx+\frac{1}{\sqrt{n}}e^{-nK_jt_j\epsilon}\right].\label{slim-7}
\end{align}

Let $n\to\infty$ first and then $m\to\infty$ in \eqref{slim-3} and \eqref{slim-4}, and use \eqref{slim-6}, \eqref{slim-7}. We can see that 
\begin{align*}
&\lim_{n\to\infty}n^{-1/2}\sigma_0^2\sum_{-\ell(n)\sqrt{n}\le i \le 0}\left[\sum_{j_1,j_2=1}^N\theta_{j_1}\theta_{j_2}\mP\bigl(i>X_{\lfloor nt_{j_1}\rfloor}^{\lfloor nt_{j_1}b\rfloor +\lfloor r_{j_1}\sqrt{n}\rfloor}\bigr)\mP\bigl(i>X_{\lfloor nt_{j_2}\rfloor}^{\lfloor nt_{j_2}b\rfloor +\lfloor r_{j_2}\sqrt{n}\rfloor}\bigr)\right]\\
&=\sigma_0^2\sum_{j_1,j_2=1}^N\theta_{j_1}\theta_{j_2}\int_{-\infty}^0\Bigl[\bP\left(B_{\sigma_1^2t_{j_1}}<x-r_{j_1}\right)\bP\left(B_{\sigma_1^2t_{j_2}}<x-r_{j_2}\right)\Bigr]dx.
\end{align*}

By the same token, one can show that
\begin{align}
&\lim_{n\to\infty}n^{-1/2}\sigma_0^2\sum_{0 < i \le \ell(n)\sqrt{n}}\left[\sum_{j_1,j_2=1}^N\theta_{j_1}\theta_{j_2}\mP\bigl(i\le X_{\lfloor nt_{j_1}\rfloor}^{\lfloor nt_{j_1}b\rfloor +\lfloor r_{j_1}\sqrt{n}\rfloor}\bigr)\mP\bigl(i\le X_{\lfloor nt_{j_2}\rfloor}^{\lfloor nt_{j_2}b\rfloor +\lfloor r_{j_2}\sqrt{n}\rfloor}\bigr)\right]\notag\\
&=\sigma_0^2\sum_{j_1,j_2=1}^N\theta_{j_1}\theta_{j_2}\int_0^{\infty}\Bigl[\bP\left(B_{\sigma_1^2t_{j_1}}\ge x-r_{j_1}\right)\bP\left(B_{\sigma_1^2t_{j_2}}\ge x-r_{j_2}\right)\Bigr]dx.\label{slim-8}
\end{align}
Thus, we have shown that
$$\lim_{n\to\infty} \sum_{|i|\le \ell(n)\sqrt{n}}\bE\left[a_{n,i}\bigl(\eta_0(i)-\mu_0\bigr)\right]^2=\sum_{i,j=1}^N\theta_i \theta_j \sigma_0^2\Gamma_2\bigl((t_i,r_i),(t_j,r_j)\bigr).$$
For the second condition in Theorem \ref{LF}, we need to pick $\ell(n)$ in a smart way. Note that from \eqref{mid-2}, 
$$|a_{n,i}|\le n^{-1/4}\sum_{j=1}^N|\theta_j|\overset{def}{=}c_0n^{-1/4}.$$ 
Then,
\begin{align*}
&\sum_{|i|\le \ell(n)\sqrt{n}}\bE\left[a_{n,i}^2(\eta_0(i)-\mu_0\bigr)^2 \ind\{|a_{n,i}\bigl(\eta_0(i)-\mu_0\bigr)|\ge\epsilon\}\right]\\
&\le C\ell(n)\bE\left[(\eta_0(0)-\mu_0\bigr)^2 \ind\{|\bigl(\eta_0(0)-\mu_0\bigr)|\ge n^{1/4}\epsilon/c_0\}\right]. 
\end{align*}
Since $\eta_0(0)$ has finite 2nd moment, the expectation above will vanish as $n\to\infty$. We can pick $\ell(n)$ so that it grows slowly enough. For example 

$$\ell(n)=\left\{\bE\left[(\eta_0(0)-\mu_0\bigr)^2 \ind\{|\bigl(\eta_0(0)-\mu_0\bigr)|\ge n^{1/8}\}\right]\right\}^{-1/2}.$$ 

Therefore, in sum, we have shown that 
$$\sum_{|i|\le \ell(n)\sqrt{n}}a_{n,i}\bigl(\eta_0(i)-\mu_0\bigr)\Rightarrow \sum_{j=1}^N\theta_j S(t_j,r_j),\mbox{ as }n\to\infty.$$
 
Combining this with Lemma \ref{lmm3-9}, the first case has been proved.\par

Case (b):   Under the condition that $\{\eta_0(x)\}_{x\in\bZ}$ is $\pi_0$-distributed, according to \eqref{initial-increments-rep}, $\eta_0(\cdot)$ has the following representation.
\be \eta_0(i)=\sum_{j\in\bZ}\sum_{k=0}^{\infty}\xi_{-k}(j)\left[p^k(i,j)-p^k(i-1,j)\right],\quad i\in\bZ.\ee
Thus, we can rewrite $\sum_{j=1}^N\theta_j \overline{S}_n(t_j,r_j)$ into
$$\sum_{j=1}^N\theta_j \overline{S}_n(t_j,r_j)=\sum_{i\in\bZ}\sum_{j\in\bZ}a_{n,i}\sum_{k=0}^{\infty}\xi_{-k}(j)\left(p^k(i,j)-p^k(i-1,j)\right),$$
where $a_{n,i}$ is defined in \eqref{mid-2}.\par

Now let $\ell(n)$ be any increasing function of $n$ such that $\lim_{n\to\infty}n/\sqrt{\ell(n)}=0$. Similar to Case (a), we would like to show that
\begin{lemma}\label{lmm3-9case2}
\be\lim_{n\to \infty}\bE \Bigl|\sum_{i\in\bZ}\sum_{j\in\bZ}a_{n,i}\sum_{k=\ell(n)}^{\infty}\xi_{-k}(j)\left(p^k(i,j)-p^k(i-1,j)\right)\Bigr|^2=0.\label{case2-eq1}\ee
\end{lemma}
\begin{proof}
\begin{align}
&\bE \Bigl|\sum_{i\in\bZ}\sum_{j\in\bZ}a_{n,i}\sum_{k=\ell(n)}^{\infty}\xi_{-k}(j)\left(p^k(i,j)-p^k(i-1,j)\right)\Bigr|^2\notag\\
&=\sigma_{\xi}^2\sum_{i_1\in\bZ}\sum_{i_2\in\bZ}a_{n,i_1}a_{n,i_2}\sum_{k=\ell(n)}^\infty  \sum_{j\in\bZ}\left(p^k(i_1,j)-p^k(i_1-1,j)\right)\left(p^k(i_2,j)-p^k(i_2-1,j)\right)\notag\\
&=\sigma_{\xi}^2\sum_{i_1\in\bZ}\sum_{i_2\in\bZ}a_{n,i_1}a_{n,i_2}\sum_{k=\ell(n)}^\infty  \left[2q^{k}(i_2-i_1,0)-q^{k}(i_2-i_1+1,0)-q^{k}(i_2-i_1-1,0)\right]\notag\\
&\le \sigma_{\xi}^2\sum_{j\in\bZ}\left|\sum_{k=\ell(n)}^\infty  \left[2q^{k}(j,0)-q^{k}(j+1,0)-q^{k}(j-1,0)\right]\right|\cdot \left|\sum_{i\in\bZ}a_{n,i}a_{n,i+j}\right|\notag\\
&=  \sigma_{\xi}^2\sum_{j\in\bZ}\left|\sum_{k=\ell(n)}^\infty \frac{1}{2\pi}\int_{-\pi}^\pi \phi_Y^k(\vartheta)\left(2e^{-\i j\vartheta}-e^{-\i (j+1)\vartheta}-e^{-\i (j-1)\vartheta}\right)d\vartheta\right|\cdot \left|\sum_{i\in\bZ}a_{n,i}a_{n,i+j}\right|\notag\\
&=  \frac{\sigma_{\xi}^2}{\pi}\sum_{j\in\bZ}\left|\int_{-\pi}^\pi \frac{\phi_Y^{\ell(n)}(\vartheta)\left(1-\cos\vartheta\right)}{1-\phi_Y(\vartheta)}e^{-\i j\vartheta}d\vartheta\right|\cdot \left|\sum_{i\in\bZ}a_{n,i}a_{n,i+j}\right|.\label{case2-eq2}
\end{align}
where $\phi_Y(\vartheta)=\sum_{j\in\bZ}q(0,j)e^{\i j\vartheta}$. Notice that the integrand in \eqref{case2-eq2} is a nonnegative and integrable function due to the fact that $\frac{\phi_Y^{\ell(n)}(\vartheta)\left(1-\cos\vartheta\right)}{1-\phi_Y(\vartheta)}$ is an analytic function (see the proof of Lemma \ref{lmmA-3}) and $\phi_Y(\vartheta)=\abs{\phi_X(\vartheta)}^2$ where $\phi_X(\vartheta)=\sum_{j\in\bZ}w(j)e^{\i j\vartheta}$. Thus, the integral in \eqref{case2-eq2} has the following bound.
\be\left|\int_{-\pi}^\pi \frac{\phi_Y^{\ell(n)}(\vartheta)\left(1-\cos\vartheta\right)}{1-\phi_Y(\vartheta)}e^{-\i j\vartheta}d\vartheta\right|\le C\int_{-\pi}^{\pi}\phi_Y^{\ell(n)}(\vartheta)d\vartheta=Cq^{\ell(n)}(0,0)\le \frac{C}{\sqrt{\ell(n)}}.\label{case2-eq3}\ee
where the last inequality is from \eqref{b2}.\par

For the last summation in \eqref{case2-eq2}, due to assumption \eqref{weight-assump-1}, we have $\#\{j\in\bZ: \sum_{i\in\bZ}a_{n,i}a_{n,i+j}\neq 0\}=O(n)$. Furthermore, we can show that
\begin{lemma}
For all $k\in\bZ$,
\be \lim_{n\to\infty}\sum_{i\in\bZ}a_{n,i}a_{n,i+k} = \sum_{j_1=1}^N\sum_{j_2=1}^N\theta_{j_1}\theta_{j_2}\Gamma_2\left((t_{j_1},r_{j_1}),(t_{j_2},r_{j_2})\right).\label{a-sum-limit}\ee
In addition, we can find a constant $A>0$, such that
\be \left|\sum_{i\in\bZ}a_{n,i}a_{n,i+k}\right|\le A,\quad \forall k\in\bZ, n\in\bN.\label{a-sum-bound}\ee
\end{lemma}

\begin{proof}
\begin{align}
&\sum_{i\in\bZ}a_{n,i}a_{n,i+k}\notag\\
&=n^{-1/2}\sum_{j_1=1}^N\sum_{j_2=1}^N\theta_{j_1}\theta_{j_2}\sum_{i>0, i+k>0}\mP\bigl(i\le X_{\lfloor nt_{j_1}\rfloor}^{\lfloor nt_{j_1}b\rfloor +\lfloor r_{j_1}\sqrt{n}\rfloor}\bigr)\mP\bigl(i+k\le X_{\lfloor nt_{j_2}\rfloor}^{\lfloor nt_{j_2}b\rfloor +\lfloor r_{j_2}\sqrt{n}\rfloor}\bigr)\notag\\
&-n^{-1/2}\sum_{j_1=1}^N\sum_{j_2=1}^N\theta_{j_1}\theta_{j_2}\sum_{i>0, i+k\le 0}\mP\bigl(i\le X_{\lfloor nt_{j_1}\rfloor}^{\lfloor nt_{j_1}b\rfloor +\lfloor r_{j_1}\sqrt{n}\rfloor}\bigr)\mP\bigl(i+k>X_{\lfloor nt_{j_2}\rfloor}^{\lfloor nt_{j_2}b\rfloor +\lfloor r_{j_2}\sqrt{n}\rfloor}\bigr)\notag\\
&-n^{-1/2}\sum_{j_1=1}^N\sum_{j_2=1}^N\theta_{j_1}\theta_{j_2}\sum_{i\le 0, i+k> 0}\mP\bigl(i>X_{\lfloor nt_{j_1}\rfloor}^{\lfloor nt_{j_1}b\rfloor +\lfloor r_{j_1}\sqrt{n}\rfloor}\bigr)\mP\bigl(i+k\le X_{\lfloor nt_{j_2}\rfloor}^{\lfloor nt_{j_2}b\rfloor +\lfloor r_{j_2}\sqrt{n}\rfloor}\bigr)\notag\\
&+n^{-1/2}\sum_{j_1=1}^N\sum_{j_2=1}^N\theta_{j_1}\theta_{j_2}\sum_{i\le 0, i+k\le 0}\mP\bigl(i>X_{\lfloor nt_{j_1}\rfloor}^{\lfloor nt_{j_1}b\rfloor +\lfloor r_{j_1}\sqrt{n}\rfloor}\bigr)\mP\bigl(i+k>X_{\lfloor nt_{j_2}\rfloor}^{\lfloor nt_{j_2}b\rfloor +\lfloor r_{j_2}\sqrt{n}\rfloor}\bigr).\notag
\end{align}
For the first term and the fourth term above, one can use the same technique we have used in proving \eqref{slim-8} to show that 
\begin{align*}
&\lim_{n\to\infty}n^{-1/2}\sum_{i>0, i+k>0}\mP\bigl(i\le X_{\lfloor nt_{j_1}\rfloor}^{\lfloor nt_{j_1}b\rfloor +\lfloor r_{j_1}\sqrt{n}\rfloor}\bigr)\mP\bigl(i+k\le X_{\lfloor nt_{j_2}\rfloor}^{\lfloor nt_{j_2}b\rfloor +\lfloor r_{j_2}\sqrt{n}\rfloor}\bigr)\\
&=\int_0^{+\infty}\bP(B_{\sigma_1^2t_{j_1}}\le r_{j_1}-x)\bP(B_{\sigma_1^2t_{j_2}}\le r_{j_2}-x)dx,
\end{align*}
and
\begin{align*}
&\lim_{n\to\infty}n^{-1/2}\sum_{i\le 0, i+k\le 0}\mP\bigl(i>X_{\lfloor nt_{j_1}\rfloor}^{\lfloor nt_{j_1}b\rfloor +\lfloor r_{j_1}\sqrt{n}\rfloor}\bigr)\mP\bigl(i+k>X_{\lfloor nt_{j_2}\rfloor}^{\lfloor nt_{j_2}b\rfloor +\lfloor r_{j_2}\sqrt{n}\rfloor}\bigr)\\
&=\int_{-\infty}^0\bP(B_{\sigma_1^2t_{j_1}}>r_{j_1}-x)\bP(B_{\sigma_1^2t_{j_2}}>r_{j_1}-x)dx.
\end{align*}

For the second term, 
\begin{align*}
&n^{-1/2}\sum_{i>0, i+k\le 0}\mP\bigl(i\le X_{\lfloor nt_{j_1}\rfloor}^{\lfloor nt_{j_1}b\rfloor +\lfloor r_{j_1}\sqrt{n}\rfloor}\bigr)\mP\bigl(i+k>X_{\lfloor nt_{j_2}\rfloor}^{\lfloor nt_{j_2}b\rfloor +\lfloor r_{j_2}\sqrt{n}\rfloor}\bigr)\\
&=n^{-1/2}\sum_{i=1}^{-k}\mP\bigl(i\le X_{\lfloor nt_{j_1}\rfloor}^{\lfloor nt_{j_1}b\rfloor +\lfloor r_{j_1}\sqrt{n}\rfloor}\bigr)\mP\bigl(i+k>X_{\lfloor nt_{j_2}\rfloor}^{\lfloor nt_{j_2}b\rfloor +\lfloor r_{j_2}\sqrt{n}\rfloor}\bigr)\\
&\le |k|n^{-1/2}\to 0, \quad\mbox{as }n\to\infty.
\end{align*}

By the same token, one can show that 
$$\lim_{n\to\infty}n^{-1/2}\sum_{i\le 0, i+k> 0}\mP\bigl(i>X_{\lfloor nt_{j_1}\rfloor}^{\lfloor nt_{j_1}b\rfloor +\lfloor r_{j_1}\sqrt{n}\rfloor}\bigr)\mP\bigl(i+k\le X_{\lfloor nt_{j_2}\rfloor}^{\lfloor nt_{j_2}b\rfloor +\lfloor r_{j_2}\sqrt{n}\rfloor}\bigr)=0.$$

In sum, we have shown that
\begin{align*}
&\lim_{n\to\infty}\sum_{i\in\bZ}a_{n,i}a_{n,i+k}\\
&=\sum_{j_1=1}^N\sum_{j_2=1}^N\theta_{j_1}\theta_{j_2}\biggl[\int_0^{+\infty}\bP(B_{\sigma_1^2t_{j_1}}\le r_{j_1}-x)\bP(B_{\sigma_1^2t_{j_2}}\le r_{j_2}-x)dx\\
&+\int_{-\infty}^0\bP(B_{\sigma_1^2t_{j_1}}>r_{j_1}-x)\bP(B_{\sigma_1^2t_{j_2}}>r_{j_1}-x)dx\biggr]=\sum_{j_1=1}^N\sum_{j_2=1}^N\theta_{j_1}\theta_{j_2}\Gamma_2\left((t_{j_1},r_{j_1}),(t_{j_2},r_{j_2})\right).
\end{align*}

For the inequality \eqref{a-sum-bound}, it is simply concluded from the limit \eqref{a-sum-limit} and the fact that 
$$\sum_{i\in\bZ}\abs{a_{n,i}a_{n,i+k}}\le \sum_{i\in\bZ}a_{n,i}^2,  \quad k\in\bZ.\qedhere$$
\end{proof}

Combine \eqref{case2-eq3} and \eqref{a-sum-bound} together, we can find a constant $C>0$ to further bound \eqref{case2-eq2}.
$$\bE \Bigl|\sum_{i\in\bZ}\sum_{j\in\bZ}a_{n,i}\sum_{k=\ell(n)}^{\infty}\xi_{-k}(j)\left(p^k(i,j)-p^k(i-1,j)\right)\Bigr|^2\le C\frac{n}{\sqrt{\ell(n)}}\to 0,\mbox{ as }n\to\infty.$$

Thus, the proof of Lemma \ref{lmm3-9case2} is complete.
\end{proof}

For the main part $\sum_{i\in\bZ}\sum_{j\in\bZ}a_{n,i}\sum_{k=0}^{\ell(n)-1}\xi_{-k}(j)\left(p^k(i,j)-p^k(i-1,j)\right)$, again we will use the Lindeberg Feller CLT to show the convergence. First, let us check the variance. Notice that
\begin{align*}
&\bE\left[\sum_{i\in\bZ}\sum_{j\in\bZ}a_{n,i}\sum_{k=0}^{\ell(n)-1}\xi_{-k}(j)\left(p^k(i,j)-p^k(i-1,j)\right)\right]^2\\
&=\sigma_{\xi}^2\sum_{j\in\bZ}\sum_{k=0}^{\ell(n)-1}  \left[2q^{k}(j,0)-q^{k}(j+1,0)-q^{k}(j-1,0)\right] \sum_{i\in\bZ}a_{n,i}a_{n,i+j}\\
&=\sigma_{\xi}^2\sum_{j\in\bZ}\left[a(j-1)+a(j+1)-2a(j)\right]\sum_{i\in\bZ}a_{n,i}a_{n,i+j}\\
&\quad\quad-\sigma_{\xi}^2\sum_{j\in\bZ}\sum_{k=\ell(n)}^{\infty}  \left[2q^{k}(j,0)-q^{k}(j+1,0)-q^{k}(j-1,0)\right] \sum_{i\in\bZ}a_{n,i}a_{n,i+j}.
\end{align*}
where $a(x)$ is defined in \eqref{t37-2}. By the absolute convergence of $\sum_{j\in\bZ}\left[a(j-1)+a(j+1)-2a(j)\right]$ (Lemma \ref{lmmA-4}) and the uniform boundedness of $\sum_{i\in\bZ}a_{n,i}a_{n,i+j}$ from \eqref{a-sum-bound}, we can use Absolute Convergence Theorem to show that
\begin{align*}
&\lim_{n\to\infty}\sigma_{\xi}^2\sum_{j\in\bZ}\left[a(j-1)+a(j+1)-2a(j)\right]\sum_{i\in\bZ}a_{n,i}a_{n,i+j}\\
&=\sigma_{\xi}^2\sum_{j\in\bZ}\left[a(j-1)+a(j+1)-2a(j)\right]\lim_{n\to\infty}\sum_{i\in\bZ}a_{n,i}a_{n,i+j}\\
&= \frac{\sigma_{\xi}^2}{\sigma_1^2}\sum_{j_1=1}^N\sum_{j_2=1}^N\theta_{j_1}\theta_{j_2}\Gamma_2\left((t_{j_1},r_{j_1}),(t_{j_2},r_{j_2})\right).
\end{align*}
where the last equality is from \eqref{a-sum-limit} and \eqref{l39a}.\par

Combine this with \eqref{case2-eq1}, we have shown that
\begin{align}
&\lim_{n\to\infty}\bE\left[\sum_{i\in\bZ}\sum_{j\in\bZ}a_{n,i}\sum_{k=0}^{\ell(n)-1}\xi_{-k}(j)\left(p^k(i,j)-p^k(i-1,j)\right)\right]^2\notag\\
&=\frac{\sigma_{\xi}^2}{\sigma_1^2}\sum_{j_1=1}^N\sum_{j_2=1}^N\theta_{j_1}\theta_{j_2}\Gamma_2\left((t_{j_1},r_{j_1}),(t_{j_2},r_{j_2})\right).\label{case2-eq4}
\end{align}

Lastly, we would like to check the Lindeberg condition. Let us denote 
$$U_{n,k}(j)=\xi_{-k}(j)\sum_{i\in\bZ}a_{n,i}\left[p^k(i,j)-p^k(i-1,j)\right].$$ 
For all $\epsilon>0$, 
\begin{align*}
&\sum_{j\in\bZ}\sum_{k=0}^{\ell(n)-1}\bE\left[U^2_{n,k}(j)\ind\left\{\left|U_{n,k}(j)\right|\ge\epsilon\right\}\right]\le \sum_{j\in\bZ}\sum_{k=0}^{\ell(n)-1}\left\{\bE\left[U^4_{n,k}(j)\right]\right\}^{1/2}\left\{\bP(\left|U_{n,k}(j)\right|\ge\epsilon)\right\}^{1/2}\\
&\le \frac{1}{\epsilon^2}\sum_{j\in\bZ}\sum_{k=0}^{\ell(n)-1}\bE\left[U^4_{n,k}(j)\right]=\frac{\bE\left[\xi^4_{0}(0)\right]}{\epsilon^2}\sum_{j\in\bZ}\sum_{k=0}^{\ell(n)-1}\left[\sum_{i\in\bZ}a_{n,i}\left(p^k(i,j)-p^k(i-1,j)\right)\right]^4.
\end{align*}
where the first inequality is from Cauchy–-Schwarz inequality, and the second is from Chebyshev's inequality.\par

Notice that from the definition of $a_{n,i}$ in \eqref{mid-2}, we can find a constant $C>0$ such that $\abs{a_{n,i}}\le Cn^{-1/4}$. Thus,
$$\sum_{i\in\bZ}\abs{a_{n,i}}\cdot\left|p^k(i,j)-p^k(i-1,j)\right|\le Cn^{-1/4}\sum_{i\in\bZ}\left|p^k(i,j)-p^k(i-1,j)\right|\le 2Cn^{-1/4}.$$ 
Therefore,
\begin{align}
&\sum_{j\in\bZ}\sum_{k=0}^{\ell(n)-1}\bE\left[U^2_{n,k}(j)\ind\left\{\left|U_{n,k}(j)\right|\ge\epsilon\right\}\right]\notag\\
&\le  \frac{4C^2\bE\left[\xi^4_{0}(0)\right]}{\epsilon^2n^{1/2}}\sum_{j\in\bZ}\sum_{k=0}^{\ell(n)-1}\left[\sum_{i\in\bZ}a_{n,i}\left(p^k(i,j)-p^k(i-1,j)\right)\right]^2\to 0,\quad\mbox{as }n\to\infty.\label{case2-eq5}
\end{align}
where the convergence of $\sum_{j\in\bZ}\sum_{k=0}^{\ell(n)-1}\left[\sum_{i\in\bZ}a_{n,i}\left(p^k(i,j)-p^k(i-1,j)\right)\right]^2$ is from \eqref{case2-eq4}.\par

Combine \eqref{case2-eq4} and \eqref{case2-eq5} together, we have shown that
\begin{align*}
&\sum_{i\in\bZ}\sum_{j\in\bZ}a_{n,i}\sum_{k=0}^{\ell(n)-1}\xi_{-k}(j)\left(p^k(i,j)-p^k(i-1,j)\right)\\
&\quad\quad\quad\quad\quad\quad\Rightarrow \cN\left(0, \frac{\sigma_{\xi}^2}{\sigma_1^2}\sum_{j_1=1}^N\sum_{j_2=1}^N\theta_{j_1}\theta_{j_2}\Gamma_2\left((t_{j_1},r_{j_1}),(t_{j_2},r_{j_2})\right)\right),\quad\mbox{as }n\to\infty.
\end{align*}

Combine this with Lemma \ref{lmm3-9case2}, the second case has been proved. \par

Case (c):   For the last case, we assume that the initial increments $\{\eta_0(x)\}_{x\in\bZ}$ are a strongly mixing stationary sequence such that $\exists$ $\delta>0$ s.t.\ $\bE|\eta_0(0)|^{2+\delta}<\infty$, and the strong mixing coefficients of $\eta_0(\cdot)$ satisfy $\sum_{j=0}^\infty (j+1)^{2/\delta}\alpha(j)<\infty$.\par

We first investigate the variance $\bar{\sigma}_n^2 = \Vvv\left[\sum_{i\in \bZ}a_{n,i}\bigl(\eta_0(i)-\mu_0\bigr)\right]$. Notice that
\be\bar{\sigma}_n^2=\sum_{j,k\in\bZ}a_{n,j}a_{n,k}\Cvv \left[\eta_0(j),\eta_0(k)\right]=\sum_{\ell\in \bZ}\Cvv \left[\eta_0(0),\eta_0(\ell)\right]\sum_{k\in\bZ}a_{n,k}a_{n,\ell+k}.\label{slim-9}\ee

To show the limit of $\bar{\sigma}_n^2$,  we need to show that the series of covariances $\sum_{k\in\bZ}\Cvv\left(\eta_0(0),\eta_0(k)\right)$ is absolutely convergent. In order to achieve that, we use the following lemma which is part of Theorem 1.1 in \cite{27}, 
\begin{lemma}\label{covariance-bound}
Suppose $X$ and $Y$ are two integrable real-valued r.v.'s. Let us assume that $XY$ is also integrable and denote $\alpha=\alpha\left(\sigma(X),\sigma(Y)\right)$ in \eqref{strong-mixing-def}. Then
\be |\Cvv(X,Y)|\le 4\int_0^\alpha Q_X(u)Q_Y(u)du,\ee
where $Q_X(u)$, $Q_Y(u)$ are the quantile functions of $|X|$ and $|Y|$ respectively $(i.e.\ Q_X(u)=\inf\{x\in\bR^+: \bP(|X|>x)\le u\}$, $0\le u\le 1)$.
\end{lemma}

Let us denote the quantile function of $\abs{\eta_0(0)-\mu_0}$ by $Q_{\eta}(u)$. Then
\begin{align}
&\sum_{\ell\in\bZ}\left|\Cvv \left[\eta_0(0), \eta_0(\ell)\right]\right|\le 4\sum_{\ell\in\bZ}\int_0^{\alpha(|\ell|)} \left[Q_{\eta}(u)\right]^2du\le 4\int_0^1 \sum_{\ell\in\bZ}\ind_{\{u\le\alpha(|\ell|)\}} \left[Q_{\eta}(u)\right]^2du\notag\\
&\le 4 \left[\int_0^1 \Bigl(\sum_{\ell\in\bZ}\ind_{\{u\le\alpha(|\ell|)\}}\Bigr)^{(2+\delta)/\delta}du\right]^{\delta/(2+\delta)}\cdot \left[\int_0^1\left(Q_{\eta}(u)\right)^{2+\delta}du\right]^{2/(2+\delta)}.\notag
\end{align}
Note that $\alpha(n)\searrow 0$ as $n\to\infty$. Thus,
\begin{align*}
\int_0^1 \Bigl(\sum_{\ell\in\bZ}\ind_{\{u\le\alpha(|\ell|)\}}\Bigr)^{(2+\delta)/\delta}du=&\sum_{j=0}^\infty \int_{\alpha(j+1)}^{\alpha(j)} (2j+1)^{(2+\delta)/\delta}du\\
=&\sum_{j=0}^\infty (2j+1)^{(2+\delta)/\delta}[\alpha(j)-\alpha(j+1)].
\end{align*}
And we have,
\begin{align*}
&\sum_{j=0}^n (2j+1)^{(2+\delta)/\delta}[\alpha(j)-\alpha(j+1)]\\
&= \alpha(0)-(2n+1)^{(2+\delta)/\delta}\alpha(n+1)+\sum_{j=1}^{n}\left[(2j+1)^{(2+\delta)/\delta}-(2j-1)^{(2+\delta)/\delta}\right]\alpha(j)\\
&\le \alpha(0)-(2n+1)^{(2+\delta)/\delta}\alpha(n+1)+C\sum_{j=1}^{n} (j+1)^{2/\delta}\alpha(j).
\end{align*}
Let $n$ go to $\infty$ above, 
$$\sum_{j=0}^\infty (2j+1)^{(2+\delta)/\delta}[\alpha(j)-\alpha(j+1)] \le \alpha(0)+C\sum_{j=1}^{\infty} (j+1)^{2/\delta}\alpha(j) <\infty.$$
Also, 
$$\int_0^1\left(Q_{\eta}(u)\right)^{2+\delta}du=\bE\left[|\eta_0(0)-\mu_0|^{2+\delta}\right]<\infty,$$
where the equality is because if $U$ is uniformly distributed on $(0,1)$, then $Q_{\eta}(U)$ has the same distribution as $\abs{\eta_0(0)-\mu_0}$. In fact, for $x\in\bR^+$, 
\begin{align}
\bP(Q_{\eta}(U)>x)=&\bP\left(\inf\{y\in\bR^+: \bP(\abs{\eta_0(0)-\mu_0}>y)\le U\}>x\right)\notag\\
=&\bP\left\{\bP(\abs{\eta_0(0)-\mu_0}>x)> U\right\}=\bP(\abs{\eta_0(0)-\mu_0}>x).\notag
\end{align}

Therefore, we have shown that 
\be \sum_{\ell\in\bZ}\left|\Cvv \left[\eta_0(0), \eta_0(\ell)\right]\right| \le C\biggl(\sum_{j=0}^{\infty} (j+1)^{2/\delta}\alpha(j)\biggr)^{\delta/(2+\delta)}\left(\bE\left[|\eta_0(0)-\mu_0|^{2+\delta}\right]\right)^{2/(2+\delta)}<\infty.\label{slim-10}\ee

Thus, from \eqref{a-sum-limit} and \eqref{a-sum-bound}, we can let $n$ go to $\infty$ in \eqref{slim-9} and apply Dominated Convergence Theorem to conclude that
\be \lim_{n\to\infty} \bar{\sigma}_n^2=\varsigma^2\sum_{j_1=1}^N\sum_{j_2=1}^N\theta_{j_1}\theta_{j_2}\Gamma_2\left((t_{j_1},r_{j_1}),(t_{j_2},r_{j_2})\right).\label{var-limit}\ee

Since $\lim_{n\to\infty} \bar{\sigma}_n^2=0$ directly implies that $\sum_{j=1}^N\theta_i \overline{S}_n(t_j,r_j)$ converges weakly to zero. For the rest, we assume that $\sum_{j_1=1}^N\sum_{j_2=1}^N\theta_{j_1}\theta_{j_2}\Gamma_2\left((t_{j_1},r_{j_1}),(t_{j_2},r_{j_2})\right)>0$, and use Theorem 2.2(c) (restated below) in \cite{13} to complete the proof. 
\begin{theorem}\label{LPCLT}
Let $\{b_{n,i}: -m_n\le i\le m_n, n\in \bZ^+\}$ be a triangular array of real numbers such that  
\be \limsup_{n\to\infty}\sum_{i\in\bZ}b_{n,i}^2<\infty,\label{LPCLT-condition1}\ee
\be \lim_{n\to\infty}\max_{i\in\bZ}\abs{b_{n,i}}=0.\label{LPCLT-condition2}\ee
 where $b_{n,i}=0$, if $|i|>m_n$.\par
Also, we assume that $\{\bar{\eta}(i): i\in\bZ\}$ is a centered, strongly mixing and non-degenerate (i.e.\ $\Vvv (\bar{\eta}(0))>0$) stationary sequence such that
\be \Vvv \left(\sum_{i=-m_n}^{m_n}b_{n,i}\bar{\eta}(i)\right)=1,\label{LPCLT-condition3}\ee
and there exists $\delta>0$ so that $\bE\left\{|\bar{\eta}(0)|^{2+\delta}\right\}<\infty$ and $\sum_{j=0}^\infty (j+1)^{2/\delta}\alpha(j)<\infty$.\par
Then, 
\be \sum_{i=-m_n}^{m_n}b_{n,i}\bar{\eta}(i)\Rightarrow \cN(0, 1),\quad\mbox{as }n\to\infty.\ee
\end{theorem}

In our case, we can let $b_{n,i}=a_{n,i}/\bar{\sigma}_n$, $\bar{\eta}(i)=\eta_0(i)-\mu_0$, $i\in\bZ$. To use Theorem \ref{LPCLT}, it is enough to show that $\{b_{n,i}\}$ satisfies conditions \eqref{LPCLT-condition1} and\eqref{LPCLT-condition2}.\par 

For condition \eqref{LPCLT-condition1}, from \eqref{a-sum-limit} and \eqref{var-limit}, we see that
$$\lim_{n\to\infty} \sum_{i\in\bZ}b_{n,i}^2=\frac{1}{\varsigma^2}<\infty.$$

For condition \eqref{LPCLT-condition2}, from \eqref{mid-2}, 
$$\abs{a_{n,i}}\le n^{-1/4}\sum_{j=1}^N\abs{\theta_j}, \quad i\in\bZ.$$
Therefore,
$$\max_{i\in\bZ}\abs{b_{n,i}}\le \frac{1}{n^{1/4}\bar{\sigma}_n}\sum_{j=1}^N\abs{\theta_j}\to 0,\quad\mbox{as }n\to\infty.$$

Apply Theorem \ref{LPCLT}, we have
$$\frac{1}{\bar{\sigma}_n}\sum_{i\in \bZ}a_{n,i}\bigl(\eta_0(i)-\mu_0\bigr)\Rightarrow \cN(0,1).$$

Combine this with \eqref{var-limit}, we may conclude that
$$\sum_{j=1}^N\theta_i \overline{S}_n(t_j,r_j)\Rightarrow \cN\left(0,\varsigma^2\sum_{j_1=1}^N\sum_{j_2=1}^N\theta_{j_1}\theta_{j_2}\Gamma_2\left((t_{j_1},r_{j_1}),(t_{j_2},r_{j_2})\right)\right).$$

Thus, the third case has been proved and the proof of Lemma \ref{lmm3-8} is complete.
\end{proof}

Now we turn to the remaining term $\overline{F}_n(t,r)$ in \eqref{2-7}, let us define another mean-zero Gaussian process $\{F(t,r): t\in\bR^+,r\in\bR\}$ which is independent of process $\{S(t,r)\}_{t\in\bR^+,r\in\bR}$ and has covariance 
\be
\mathbb{E}\bigl[F(t,r)F(s,q)\bigr]=\frac{\sigma_\xi^2}{\sigma_1^2}\Gamma_1\bigl((t,r),(s,q)\bigr),\quad t,s\in\bR^+,r,q\in\bR.\label{2-t9a}\ee
We can show that 
\begin{lemma}\label{lmm3-10}
Under the assumptions in Theorem \ref{thm2-4}, $\{\overline{F}_n(t,r)\}_{t\in\bR^+,r\in\bR}$ will converge weakly (in the sense of finite dimensional distributions) to the Gaussian process $ \{F(t,r)\}_{t\in\bR^+,r\in\bR}$ as $n$ goes to $\infty$.
\end{lemma}
\begin {remark}\label{rmk3-2}
Here we gave two equivalent expressions for the $\Gamma_1$ function defined in \eqref{gamma3} which may be used in the following context.
\be
\Gamma_1\bigl((s,q),(t,r)\bigr)=\int_{-\infty}^\infty \Bigl[\bP(B_{\sigma_1^2s}\le q-x)\bP(B_{\sigma_1^2t}>r-x)-\bP(B_{\sigma_1^2s}\le q-x,B_{\sigma_1^2t}>r-x)\Bigr]\,\mathrm{d}x,\label{2-10a}\ee
and
\be\Gamma_1\bigl((s,q),(t,r)\bigr)=\frac{1}{2}\int_{\sigma_1^2|t-s|}^{\sigma_1^2(t+s)}\frac{1}{\sqrt{2\pi v}}\exp\bigl\{-\frac{1}{2v}(r-q)^2\bigr\}dv,\label{2-10b}\ee
for $s,t\in\bR^+$ and $q,r\in\bR$.
\end{remark}

\begin{proof}
Note that $\overline{F}_n(t,r)$ can be rewritten as
\be\overline{F}_n(t,r)=n^{-1/4}\sum_{k=1}^{\lfloor nt\rfloor}\mE\left[\xi_k(X_{\lfloor nt\rfloor -k}^{y(n)})\right]=n^{-1/4}\sum_{k=1}^{\lfloor nt\rfloor}\sum_{x\in\mathbb{Z}}\xi_k(x)\mP\bigl(X_{\lfloor nt\rfloor -k}^{y(n)}=x\bigr),\label{add-1}\ee 
Recall that $X_\centerdot^i$ is defined to be a random walk starting from site $i$ with transition probability $p$ defined in \eqref{21-d1}.\par
 
Thus, we immediately have $\mathbb{E}\overline{F}_n(t,r)=0$.\par

Now let us take a look at the covariance. Suppose $X_\centerdot^{\lfloor ntb\rfloor +\lfloor r\sqrt{n}\rfloor}$ and $X_\centerdot^{\lfloor nsb\rfloor +\lfloor q\sqrt{n}\rfloor}$ are two independent random walks with transition probability $p$.\par

For the case $s=t$, $r\not= q$, let us denote $x_n=\lfloor r\sqrt{n}\rfloor -\lfloor q\sqrt{n}\rfloor$. Then,
\begin{align}
\mathbb{E}\bigl[\overline{F}_n(t,r)\overline{F}_n(t,q)\bigr]=&n^{-1/2}\sum_{k=1}^{\lfloor nt\rfloor}\sum_{x\in\mathbb{Z}}\sigma_\xi^2\mP\bigl(X_{\lfloor nt\rfloor -k}^{\lfloor ntb\rfloor +\lfloor r\sqrt{n}\rfloor}=x\bigr)\mP\bigl(X_{\lfloor nt\rfloor -k}^{\lfloor ntb\rfloor +\lfloor q\sqrt{n}\rfloor}=x\bigr)\notag\\
=&n^{-1/2}\sigma_\xi^2\sum_{k=1}^{\lfloor nt\rfloor}q^{\lfloor nt\rfloor -k}(x_n,0)=n^{-1/2}\sigma_\xi^2\sum_{k=0}^{\lfloor nt\rfloor-1}q^k(0,x_n).\label{F-cov-exp}
\end{align}
The second equality above is because $q$ can be viewed as the transition probability of $X_\centerdot^{\lfloor ntb\rfloor +\lfloor r\sqrt{n}\rfloor}-X_\centerdot^{\lfloor ntb\rfloor +\lfloor q\sqrt{n}\rfloor}$ (see the proof of Lemma \ref{lmm2-2}).\par

Combine \eqref{F-cov-exp} with \eqref{2-l11} in the Appendix,
\begin{align*}
\lim_{n\to\infty}\mathbb{E}\bigl[\overline{F}_n(t,r)\overline{F}_n(t,q)\bigr]=&\frac{\sigma_\xi^2}{2\sigma_1^2}\int_0^{2\sigma_1^2t}\frac{1}{\sqrt{2\pi v}}\exp\bigl\{-\frac{(r-q)^2}{2v}\bigr\}dv\\
=&\frac{\sigma_\xi^2}{\sigma_1^2}\Gamma_1\bigl((t,q),(t,r)\bigr).
\end{align*}

For the case $s\not= t$, we suppose $s<t$ and let $x_n=X_{\lfloor nt\rfloor-\lfloor ns\rfloor}^{\lfloor ntb\rfloor +\lfloor r\sqrt{n}\rfloor}-\lfloor nsb\rfloor -\lfloor q\sqrt{n}\rfloor$. Then, the covariance can be written as
\begin{align*}
\mathbb{E}\bigl[\overline{F}_n(t,r)\overline{F}_n(s,q)\bigr]=&n^{-1/2}\sigma_\xi^2\sum_{k=1}^{\lfloor ns\rfloor}\mP\bigl(X_{\lfloor nt\rfloor-k}^{\lfloor ntb\rfloor +\lfloor r\sqrt{n}\rfloor}-X_{\lfloor ns\rfloor-k}^{\lfloor nsb\rfloor +\lfloor q\sqrt{n}\rfloor}=0\bigr)\\
=&n^{-1/2}\sigma_\xi^2\bE\biggl[\sum_{k=1}^{\lfloor ns\rfloor}\mP\bigl(X_{\lfloor nt\rfloor-k}^{\lfloor ntb\rfloor +\lfloor r\sqrt{n}\rfloor}-X_{\lfloor ns\rfloor-k}^{\lfloor nsb\rfloor +\lfloor q\sqrt{n}\rfloor}=0 \big| X_{\lfloor nt\rfloor-\lfloor ns\rfloor}^{\lfloor ntb\rfloor +\lfloor r\sqrt{n}\rfloor}\bigr)\biggr]\\
=&n^{-1/2}\sigma_\xi^2\bE\biggl[\sum_{k=1}^{\lfloor ns\rfloor}q^{\lfloor ns\rfloor-k}\left(X_{\lfloor nt\rfloor-\lfloor ns\rfloor}^{\lfloor ntb\rfloor +\lfloor r\sqrt{n}\rfloor}-\lfloor nsb\rfloor -\lfloor q\sqrt{n}\rfloor,0\right)\biggr]\\
=&n^{-1/2}\sigma_\xi^2\bE\biggl[\sum_{k=0}^{\lfloor ns\rfloor-1}q^{k}\left(0,x_n\right)\biggr].
\end{align*}
Note that the 3rd equality is because of the Markov property of random walks. \par

Similar to the case $s=t$, we will use Corollary \ref{crllyB-1} in the Appendix to derive the limit. By CLT, we have 
$$n^{-1/2}x_n\Rightarrow B_{\sigma_1^2|t-s|}+(r-q),\quad\mbox{as } n\to\infty.$$ 

We can pick random variables $\{\hat{X}_{\lfloor nt\rfloor-\lfloor ns\rfloor}^{\lfloor ntb\rfloor +\lfloor r\sqrt{n}\rfloor}\}_{n\in\bN}$ such that $\hat{X}_{\lfloor nt\rfloor-\lfloor ns\rfloor}^{\lfloor ntb\rfloor +\lfloor r\sqrt{n}\rfloor}\overset{d}{=}X_{\lfloor nt\rfloor-\lfloor ns\rfloor}^{\lfloor ntb\rfloor +\lfloor r\sqrt{n}\rfloor}$, $n\in\bN$ and $n^{-1/2}\hat{x}_n=n^{-1/2}\left(\hat{X}_{\lfloor nt\rfloor-\lfloor ns\rfloor}^{\lfloor ntb\rfloor +\lfloor r\sqrt{n}\rfloor}-\lfloor nsb\rfloor -\lfloor q\sqrt{n}\rfloor\right)\overset{a.s.}{\to} B_{\sigma_1^2|t-s|}+(r-q)$ as $n\to\infty$ (see Theorem 3.2.2 in \cite{2}).\par

Then, by \eqref{2-l11},
$$\lim_{n\to\infty}n^{-1/2}\sum_{k=0}^{\lfloor ns\rfloor-1}q^{k}\left(0,\hat{x}_n\right)=\frac{1}{2\sigma_1^2}\int_0^{2\sigma_1^2s}\frac{1}{\sqrt{2\pi v}}\exp\Bigl\{-\frac{(B_{\sigma_1^2|t-s|}+r-q)^2}{2v}\Bigr\}dv,\quad a.s.$$

Also, by \eqref{b2}, 
$$n^{-1/2}\sum_{k=0}^{\lfloor ns\rfloor-1}q^{k}\left(0,\hat{x}_n\right)\le n^{-1/2}\left(1+\sum_{k=1}^{\lfloor ns\rfloor-1}Ck^{-1/2}\right)=O(1).$$

Hence, we can apply Dominated Convergence Theorem and get
\begin{align*}
&\lim_{n\to\infty}\mathbb{E}\bigl[\overline{F}_n(t,r)\overline{F}_n(s,q)\bigr]=\lim_{n\to\infty}n^{-1/2}\sigma_\xi^2\bE\biggl\{\sum_{k=0}^{\lfloor ns\rfloor-1}q^{k}\left(0,\hat{x}_n\right)\biggr\}\\
&=\sigma_\xi^2\mathbb{E}\biggl[\frac{1}{2\sigma_1^2}\int_0^{2\sigma_1^2s}\frac{1}{\sqrt{2\pi v}}\exp\Bigl\{-\frac{(B_{\sigma_1^2|t-s|}+r-q)^2}{2v}\Bigr\}dv\biggr]\\
&=\frac{\sigma_\xi^2}{2\sigma_1^2}\int_{-\infty}^{+\infty}\int_0^{2\sigma_1^2s}\frac{1}{\sqrt{2\pi v}}\exp\bigl\{-\frac{(x+r-q)^2}{2v}\bigr\}\frac{1}{\sqrt{2\pi \sigma_1^2|t-s|}}\exp\bigl\{-\frac{x^2}{2\sigma_1^2|t-s|}\bigr\}dvdx\\
&=\frac{\sigma_\xi^2}{2\sigma_1^2}\int_0^{2\sigma_1^2s}\frac{1}{\sqrt{2\pi (v+\sigma_1^2|t-s|)}}\exp\bigl\{-\frac{(r-q)^2}{2(v+\sigma_1^2|t-s|)}\bigr\}dv\\
&=\frac{\sigma_\xi^2}{2\sigma_1^2}\int_{\sigma_1^2|t-s|}^{\sigma_1^2(t+s)}\frac{1}{\sqrt{2\pi v}}\exp\bigl\{-\frac{(r-q)^2}{2v}\bigr\}dv= \frac{\sigma_\xi^2}{\sigma_1^2}\Gamma_1\bigl((t,r),(s,q)\bigr).
\end{align*}

So far we have shown that,
\be\lim_{n\to\infty}\mathbb{E}\bigl[\overline{F}_n(t,r)\overline{F}_n(s,q)\bigr]=\frac{\sigma_\xi^2}{\sigma_1^2}\Gamma_1\bigl((t,r),(s,q)\bigr),\quad (t,r),(s,q)\in\mathbb{R}_+\times\mathbb{R}. \label{F-1}\ee

Again, the next step will be applying Lindeberg Feller Central Limit Theorem to complete the proof.\par

For any fixed $N\in\bN$, $\{(t_j,r_j)\in \bR^+\times \bR:j=1,\ldots,N\}$, and $\{\theta_j\in \bR:j=1,\ldots,N\}$, without losing generality, let us suppose that $t_1\le t_2\le\ldots\le t_N$. We will rewrite $\sum_{j=1}^N\theta_j \overline{F}_n(t_j,r_j)$ into a sum of independent random variables. Notice
\begin{align*}
\sum_{j=1}^N\theta_j \overline{F}_n(t_j,r_j)=&n^{-1/4}\sum_{j=1}^N\theta_j\sum_{k=1}^{\lfloor nt_j\rfloor}\mE\left[\xi_k(X_{\lfloor nt_j\rfloor -k}^{\lfloor nt_jb\rfloor +\lfloor r_j\sqrt{n}\rfloor})\right]\\
=&n^{-1/4}\sum_{j=1}^N\theta_j\sum_{k=1}^{\lfloor nt_j\rfloor}\sum_{x\in\bZ}\xi_k(x)p^{\lfloor nt_j\rfloor -k}\left(\lfloor nt_jb\rfloor +\lfloor r_j\sqrt{n}\rfloor,x\right)\\
=&\sum_{\ell=0}^{N-1}\sum_{k=\lfloor nt_{\ell}\rfloor+1}^{\lfloor nt_{\ell+1}\rfloor}n^{-1/4}\sum_{j=\ell+1}^N\theta_j\sum_{x\in\bZ}\xi_k(x)p^{\lfloor nt_j\rfloor -k}\left(\lfloor nt_jb\rfloor +\lfloor r_j\sqrt{n}\rfloor,x\right)\\
=&\sum_{k=1}^{\lfloor nt_{N}\rfloor}n^{-1/4}\sum_{j=\ell(k)+1}^N\theta_j\sum_{x\in\bZ}\xi_k(x)p^{\lfloor nt_j\rfloor -k}\left(\lfloor nt_jb\rfloor +\lfloor r_j\sqrt{n}\rfloor,x\right).
\end{align*}
where we denote $t_0=0$ and $\ell(k)=i$ iff $\lfloor nt_i\rfloor+1\le k \le \lfloor nt_{i+1}\rfloor$.\par 

Let us also denote 
$$V_{n,k}=n^{-1/4}\sum_{j=\ell(k)+1}^N\theta_j\sum_{x\in\bZ}\xi_k(x)p^{\lfloor nt_j\rfloor -k}\left(\lfloor nt_jb\rfloor +\lfloor r_j\sqrt{n}\rfloor,x\right),\quad k\in\{1,2,\ldots,\lfloor nt_{N}\rfloor\}.$$
$$\sum_{j=1}^N\theta_j \overline{F}_n(t_j,r_j)=\sum_{k=1}^{\lfloor nt_{N}\rfloor}V_{n,k}.$$
The random variables $\{V_{n,k}\}_{1\le k \le \lfloor nt_{N}\rfloor}$ are independent. Due to \eqref{F-1}, we can directly check that the first condition in Lindeberg-Feller CLT (Theorem \ref{LF}) holds
\be\lim_{n\to\infty}\sum_{k=1}^{\lfloor nt_{N}\rfloor}\bE V_{n,k}^2=\sum_{i,j=1}^N\theta_i \theta_j\frac{\sigma_\xi^2}{\sigma_1^2}\Gamma_1\bigl((t_i,r_i),(t_j,r_j)\bigr).\ee

Now let us check the second condition.  For every fixed $\epsilon>0$,
\begin{align}
&\sum_{k=1}^{\lfloor nt_{N}\rfloor}\bE\left(V_{n,k}^2 \ind\{|V_{n,k}|\ge\epsilon\}\right)\label{F-3}\\
&\le n^{-1/2}\sum_{k=1}^{\lfloor nt_{N}\rfloor}C\sum_{j=\ell(k)+1}^N\theta_j^2\bE\left\{\left[\sum_{x\in\bZ}\xi_k(x)p^{\lfloor nt_j\rfloor -k}\left(\lfloor nt_jb\rfloor +\lfloor r_j\sqrt{n}\rfloor,x\right)\right]^2\ind\{|V_{n,k}|\ge\epsilon\}\right\}\notag\\
&\le   Cn^{-1/2}\sum_{k=1}^{\lfloor nt_{N}\rfloor}\sum_{j=\ell(k)+1}^N\theta_j^2\left\{\bE\left[\sum_{x\in\bZ}\xi_k(x)p^{\lfloor nt_j\rfloor -k}\left(\lfloor nt_jb\rfloor +\lfloor r_j\sqrt{n}\rfloor,x\right)\right]^4\right\}^{1/2}\left\{\bP(|V_{n,k}|\ge\epsilon)\right\}^{1/2}.\label{F-2}
\end{align}

For the moment in the last inequality above, note that by assumption, $\xi$ has finite 4th moment. We have
\begin{align*}
&\bE\left[\sum_{x\in\bZ}\xi_k(x)p^{\lfloor nt_j\rfloor -k}\left(\lfloor nt_jb\rfloor +\lfloor r_j\sqrt{n}\rfloor,x\right)\right]^4\\
&\le C\sum_{x,y\in\bZ}\left[p^{\lfloor nt_j\rfloor -k}\left(\lfloor nt_jb\rfloor +\lfloor r_j\sqrt{n}\rfloor,x\right)\right]^2\left[p^{\lfloor nt_j\rfloor -k}\left(\lfloor nt_jb\rfloor+\lfloor r_j\sqrt{n}\rfloor,y\right)\right]^2\\
&\quad\quad\quad\quad+ C\sum_{z\in\bZ}\left[p^{\lfloor nt_j\rfloor -k}\left(\lfloor nt_jb\rfloor +\lfloor r_j\sqrt{n}\rfloor,z\right)\right]^4\\
&\le C\left[q^{\lfloor nt_j\rfloor -k}\left(0,0\right)\right]^2+Cq^{\lfloor nt_j\rfloor -k}\left(0,0\right)\sum_{z\in\bZ}\left[p^{\lfloor nt_j\rfloor -k}\left(\lfloor nt_jb\rfloor +\lfloor r_j\sqrt{n}\rfloor,z\right)\right]^2\\
&\le C\left[q^{\lfloor nt_j\rfloor -k}\left(0,0\right)\right]^2\le  \begin{cases}\frac{C}{\lfloor nt_j\rfloor -k}, & \mbox{if } k< \lfloor nt_j\rfloor,\\ C, & \mbox{if } k= \lfloor nt_j\rfloor.\end{cases}
\end{align*}
where the last inequality is from \eqref{b2}.\par

For the last term in \eqref{F-2}, we can use Markov inequality,
\begin{align*}
\bP(|V_{n,k}|\ge\epsilon)\le &\frac{1}{n\epsilon^4}\bE\left\{\sum_{j=\ell(k)+1}^N\theta_j\sum_{x\in\bZ}\xi_k(x)p^{\lfloor nt_j\rfloor -k}\left(\lfloor nt_jb\rfloor +\lfloor r_j\sqrt{n}\rfloor,x\right)\right\}^4\\
\le & \frac{C}{n\epsilon^4}\sum_{j=\ell(k)+1}^N\theta_j^2\bE\left[\sum_{x\in\bZ}\xi_k(x)p^{\lfloor nt_j\rfloor -k}\left(\lfloor nt_jb\rfloor +\lfloor r_j\sqrt{n}\rfloor,x\right)\right]^4\\
\le &\frac{C}{n\epsilon^4}\sum_{j=\ell(k)+1}^N\theta_j^2 \frac{1}{(\lfloor nt_j\rfloor -k)\vee 1}.
\end{align*}

Thus, \eqref{F-3} can be further bounded by 
\begin{align*}
&\sum_{k=1}^{\lfloor nt_{N}\rfloor}\bE\left(V_{n,k}^2 {\ind}\{|V_{n,k}|\ge\epsilon\}\right)\\
&\le  \frac{C}{\epsilon^2n}\sum_{k=1}^{\lfloor nt_{N}\rfloor}\sum_{j=\ell(k)+1}^N\frac{1}{(\lfloor nt_j\rfloor -k)^{1/2}\vee 1}\left[\sum_{i=\ell(k)+1}^N\frac{1}{(\lfloor nt_i\rfloor -k)\vee 1}\right]^{1/2}\\
&\le \frac{C}{\epsilon^2n}\sum_{k=1}^{\lfloor nt_{N}\rfloor}\left[\sum_{j=\ell(k)+1}^N\frac{1}{(\lfloor nt_j\rfloor -k)^{1/2}\vee 1}\right]^2\le  \frac{C}{\epsilon^2n}\sum_{k=1}^{\lfloor nt_{N}\rfloor}\sum_{j=\ell(k)+1}^N\frac{1}{(\lfloor nt_j\rfloor -k)\vee 1}\\
&=  \frac{C}{\epsilon^2n}\sum_{j=1}^{N}\sum_{k=1}^{\lfloor nt_j\rfloor}\frac{1}{(\lfloor nt_j\rfloor -k)\vee 1}\le \frac{C}{\epsilon^2n}\sum_{j=1}^{N}\log \{(nt_j)\vee 1\}\to 0, \quad\mbox{as } n\to \infty.
\end{align*}

We have checked the second condition in Theorem \ref{LF}, and hence the proof for Lemma \ref{lmm3-10} is complete.
\end{proof}

By Lemma \ref{lmm3-8}, Lemma \ref{lmm3-10} and the independence of $\{\overline{S}_n(t,r)\}_{t\in\bR^+,r\in\bR}$ and $\{\overline{F}_n(t,r)\}_{t\in\bR^+,r\in\bR}$, the proof of Theorem \ref{thm2-4} is complete. 
\end{proof}

\subsection{Process-level tightness}
\begin{proof}[Proof of Theorem \ref{thm2-6}]
For simplicity, let us replace $Q$ with $[0,1]^2$. Theorem 2 in \cite{12} gives a necessary and sufficient condition for the weak convergence of a $D_2$-valued process $X_n$:
\begin{enumerate}
\item Convergence of finite-dimensional distributions: For every finite set $\{(t_i,r_i)\}_{i=1}^N\subset [0,1]^2$, we have
$$\bigl(X_n(t_1,r_1),\ldots,X_n(t_N,r_N)\bigr)\Rightarrow \bigl(X(t_1,r_1),\ldots,X(t_N,r_N)\bigr),\quad \mbox{as }n\to\infty;$$
\item Tightness: $\forall \epsilon>0$, 
$$\lim_{\delta\to 0}\limsup_{n\to\infty}\bP\{w'_\delta(X_n)\ge \epsilon\}=0,$$
where the modulus $w'_\delta$ is defined as
$$w'_\delta(x)=\inf_\Delta\max_{G\in\Delta}\sup_{(t,r),(s,q)\in G}|x(t,r)-x(s,q)|,$$
in which $\Delta$ is any partition of $[0,1]^2$ formed by finitely many lines parallel to the coordinate axes such that any element $G$ of $\Delta$ is a left-closed, right-open rectangle with diameter at least $\delta$. 
\end{enumerate} 

We have already proved the marginal convergence in Theorem \ref{thm2-4}. For the tightness part, instead of using $w'_\delta$, we use the following modulus:
\be w_\delta(x)=\sup_{\begin{subarray}{l} (t,r),(s,q)\in [0,1]^2\\ \|(t,r)-(s,q)\|<\delta\end{subarray}}|x(t,r)-x(s,q)|.\label{tight-1}\ee
One can easily show that for every fixed $0<\delta<1$, $w'_{\delta/2} (x)\le w_\delta(x)$, $\forall x\in D_2$. Thus, it is sufficient to show that $$\lim_{\delta\to 0}\limsup_{n\to\infty}\bP\{w_\delta(X_n)\ge \epsilon\}=0,\quad \forall \epsilon>0,$$
which can be proved by checking the following sufficient conditions given in \cite{11}(Proposition 2):
\begin{lemma}\label{lmm3-11}
Suppose $\{X_n\}$ is a sequence of $D_2$-valued processes such that for all $n>0$, there exists a decreasing sequence $\delta_n\searrow 0$ s.t.
\begin{enumerate}
\item there exist $\beta>0$, $\kappa>2$, and $C>0$ such that for all large enough $n$,
\be \bE(|X_n(t,r)-X_n(s,q)|^\beta)\le C|(t,r)-(s,q)|^\kappa\label{tight-2}\ee
holds for all $(t,r),(s,q)\in [0,1]^2$ with Euclidean distance $|(t,r)-(s,q)|> \delta_n$;
\item $\forall\epsilon,\gamma>0$, there exists $n_0>0$ such that  for all $n\ge n_0$,
\be \bP\{w_{\delta_n}(X_n)>\epsilon\}< \gamma.\label{tight-3}\ee 
\end{enumerate}
Then, for every fixed $\epsilon,\gamma>0$, there exist $0<\delta<1$ and integer $n_0>0$ such that  
$$\bP\{w_{\delta}(X_n)\ge\epsilon\}\le\gamma, \quad\forall n\ge n_0.$$
\end{lemma}

Now let us check the first tightness condition. We set $\beta=12$ and $\delta_n=n^{-\gamma}$.
\begin{lemma}\label{lmm3-12}
Assume the assumptions in Theorem \ref{thm2-6}.
There exists constant $C>0$ s.t.\ for all sufficiently large $n$, 
\be \bE(|\fluc_n(t,r)-\fluc_n(s,q)|^{12})\le C|(t,r)-(s,q)|^\kappa\ee
holds for all $t,s,r,q\in [0,1]$ with $|(t,r)-(s,q)|>n^{-\gamma}$, where $\kappa$ and $\gamma$ can be any fixed numbers satisfying $2<\kappa<3$ and $0<\gamma\le \frac{3}{\kappa}$.
\end{lemma}
\begin{proof}
Recall from \eqref{2-7}, $$\fluc_n(t,r)=\mu_0\overline{H}_n(t,r)+\overline{S}_n(t,r)+\overline{F}_n(t,r),\quad (t,r)\in\bR^+\times\bR.$$ 
By Minkowski Inequality,
\begin{align}
&\left[\bE\left(|\fluc_n(t,r)-\fluc_n(s,q)|^{12}\right)\right]^{1/12}\le |\mu_0|\left[\bE\left(|\overline{H}_n(t,r)-\overline{H}_n(s,q)|^{12}\right)\right]^{1/12}\notag\\
&\hspace{20mm}+\left[\bE\left(|\overline{S}_n(t,r)-\overline{S}_n(s,q)|^{12}\right)\right]^{1/12}+\left[\bE\left(|\overline{F}_n(t,r)-\overline{F}_n(s,q)|^{12}\right)\right]^{1/12}.\label{tight-4}
\end{align}

For the first term on the right of \eqref{tight-4}, recall from \eqref{H-uniform-bound}, $\overline{H}_n(t,r)$ has the uniform bound 
$$\left|\overline{H}_n(t,r)\right|\le Cn^{-1/4},\quad\forall (t,r)\in\bR^+\times \bR.$$ Therefore,
\be|\mu_0|\left[\bE\left(|\overline{H}_n(t,r)-\overline{H}_n(s,q)|^{12}\right)\right]^{1/12}\le Cn^{-1/4}.\label{tight-n5}\ee

For the second term on the right of \eqref{tight-4}, from \eqref{2-5},  
\begin{align}
&\overline{S}_n(t,r)-\overline{S}_n(s,q)\notag\\
&=n^{-1/4}\sum_{i>0}\bigl(\eta_0(i)-\mu_0\bigr)\left[\mP\bigl(i\le X_{\lfloor nt\rfloor}^{\lfloor ntb\rfloor +\lfloor r\sqrt{n}\rfloor}\bigr)-\mP\bigl(i\le X_{\lfloor ns\rfloor}^{\lfloor nsb\rfloor +\lfloor q\sqrt{n}\rfloor}\bigr)\right]\notag\\
&\quad-n^{-1/4}\sum_{i\le 0}\bigl(\eta_0(i)-\mu_0\bigr)\left[\mP\bigl(i>X_{\lfloor nt\rfloor}^{\lfloor ntb\rfloor +\lfloor r\sqrt{n}\rfloor}\bigr)-\mP\bigl(i>X_{\lfloor ns\rfloor}^{\lfloor nsb\rfloor +\lfloor q\sqrt{n}\rfloor}\bigr)\right]\notag\\
&=n^{-1/4}\sum_{i<0}\bigl(\eta_0(-i)-\mu_0\bigr)\left[\mP\bigl(X_{\lfloor nt\rfloor}^i\ge -\lfloor ntb\rfloor -\lfloor r\sqrt{n}\rfloor\bigr)-\mP\bigl(X_{\lfloor ns\rfloor}^i\ge -\lfloor nsb\rfloor -\lfloor q\sqrt{n}\rfloor\bigr)\right]\notag\\
&\quad -n^{-1/4}\sum_{i\ge 0}\bigl(\eta_0(-i)-\mu_0\bigr)\left[\mP\bigl(X_{\lfloor nt\rfloor}^i<-\lfloor ntb\rfloor -\lfloor r\sqrt{n}\rfloor\bigr)-\mP\bigl(X_{\lfloor ns\rfloor}^i<-\lfloor nsb\rfloor -\lfloor q\sqrt{n}\rfloor\bigr)\right].\label{tight-5}
\end{align}
Denote the events
\begin{align*}
A_{1,i}=&\left\{X_{\lfloor nt\rfloor}^i\ge -\lfloor ntb\rfloor -\lfloor r\sqrt{n}\rfloor, X_{\lfloor ns\rfloor}^i< -\lfloor nsb\rfloor -\lfloor q\sqrt{n}\rfloor\right\},\\
A_{2,i}=&\left\{X_{\lfloor nt\rfloor}^i< -\lfloor ntb\rfloor -\lfloor r\sqrt{n}\rfloor, X_{\lfloor ns\rfloor}^i\ge -\lfloor nsb\rfloor -\lfloor q\sqrt{n}\rfloor\right\}.
\end{align*}
Then, we can rewrite \eqref{tight-5} as
\be\overline{S}_n(t,r)-\overline{S}_n(s,q)=n^{-1/4}\sum_{i\in \bZ}\bigl(\eta_0(-i)-\mu_0\bigr)\left[\mP\bigl(A_{1,i}\bigr)-\mP\bigl(A_{2,i}\bigr)\right].\label{tight-6}\ee

We give an intermediate bound for the 12th moment of $\overline{S}_n(t,r)-\overline{S}_n(s,q)$ that work for all three cases (a), (b) and (c) in Theorem \ref{thm2-6}.
\begin{lemma}\label{initial-increments-bounds}
Assume that the initial increments $\{\eta_0(x)\}_{x\in\bZ}$ satisfy either (a), (b) or (c) in Theorem \ref{thm2-6}. Then $\exists$ $C>0$ s.t.
\be \bE \left\{\left[\overline{S}_n(t,r)-\overline{S}_n(s,q)\right]^{12}\right\}\le Cn^{-3}\left\{1+\sum_{m\in\bZ}\left[\mP\bigl(A_{1,m}\bigr)+\mP\bigl(A_{2,m}\bigr)\right]\right\}^6,\quad \forall n\in\bN.\label{tight-6b}\ee
\end{lemma}
\begin{proof}
We first state a lemma which will be used several times in the following context.\par
\begin{lemma}
Let $I$ be an index set. Suppose $\{X_i\}_{i\in I}$ is an i.i.d.\ sequence with finite 12th moment, and $\{a_i\}_{i\in I}$ is an bounded fixed sequence, i.e.\ $\exists$ constant $M>1$ s.t.\ $\abs{a_i}\le M$, $\forall i\in I$. Then, there exists constant $C<\infty$, such that
\be \bE\left[\left(\sum_{i\in I}a_iX_i\right)^{12}\right]\le C\left(1+\sum_{i\in I}a_i^2\right)^6.\label{12th-moment-bound}\ee 
\end{lemma}
\begin{proof}
\begin{align*}
&\bE\left[\left(\sum_{i\in I}a_iX_i\right)^{12}\right]=\sum_{i_1,i_2,\ldots,i_{12}\in I}a_{i_1}a_{i_2}\cdots a_{i_{12}}\bE\left[X_{i_1}X_{i_2}\cdots X_{i_{12}}\right]\\
&\le C \sum_{k=1}^6\sum_{\begin{subarray}{l}\ell_1+\ell_2+\ldots+\ell_k=12\\ \ell_j\ge 2,j=1,2,\ldots,k\end{subarray}} \prod_{j=1}^k\left(\sum_{i\in I}a_i^{\ell_j}\right)\le CM^{10} \sum_{k=1}^6\sum_{\begin{subarray}{l}\ell_1+\ell_2+\ldots+\ell_k=12\\ \ell_j\ge 2,j=1,2,\ldots,k\end{subarray}} \left(\sum_{i\in I}a_i^2\right)^k\\
&\le C\sum_{k=1}^6\left(\sum_{i\in I}a_i^2\right)^k\le C\left(1+\sum_{i\in I}a_i^2\right)^6.
\end{align*}
\end{proof}

Case (a): Assume that $\{\eta_0(x)\}_{x\in\bZ}$ are i.i.d.\ with 12th finite moment. Notice that $$\left|\mP\bigl(A_{1,i}\bigr)-\mP\bigl(A_{2,i}\bigr)\right|\le 1, \quad\forall i\in\bZ.$$ 
By \eqref{12th-moment-bound},
\begin{align}
&\bE \left\{\left[\overline{S}_n(t,r)-\overline{S}_n(s,q)\right]^{12}\right\}=n^{-3}\bE \left\{\sum_{i\in \bZ}\bigl(\eta_0(-i)-\mu_0\bigr)\left[\mP\bigl(A_{1,i}\bigr)-\mP\bigl(A_{2,i}\bigr)\right]\right\}^{12}\notag\\
&\le Cn^{-3}\left\{1+\sum_{m\in \bZ}\left[\mP\bigl(A_{1,m}\bigr)-\mP\bigl(A_{2,m}\bigr)\right]^2\right\}^6\notag\le Cn^{-3}\left\{1+\sum_{m\in\bZ}\left[\mP\bigl(A_{1,m}\bigr)+\mP\bigl(A_{2,m}\bigr)\right]\right\}^6.\notag
\end{align}

Case (b): Assume $\{\eta_0(x): x\in\bZ\}$ is $\pi_0$-distributed. From \eqref{tight-6} and \eqref{initial-increments-rep}, we have
\begin{align*}
&\overline{S}_n(t,r)-\overline{S}_n(s,q)\\
&=n^{-1/4}\sum_{j\in\bZ}\sum_{k=0}^{\infty}\xi_{-k}(j)\sum_{i\in \bZ}\left[p^k(0,j+i)-p^k(0,j+i+1)\right]\cdot\left[\mP\bigl(A_{1,i}\bigr)-\mP\bigl(A_{2,i}\bigr)\right].
\end{align*} 
Note
\begin{align*}
&\sum_{i\in \bZ}\left|p^k(0,j+i)-p^k(0,j+i+1)\right|\cdot\left|\mP\bigl(A_{1,i}\bigr)-\mP\bigl(A_{2,i}\bigr)\right|\\
&\le \sum_{i\in \bZ}\left[p^k(0,j+i)+p^k(0,j+i+1)\right]=2,\quad \forall j\in\bZ,k\in\bZ^+.
\end{align*}
Again, from \eqref{12th-moment-bound},
\begin{align*}
&\bE \left\{\left[\overline{S}_n(t,r)-\overline{S}_n(s,q)\right]^{12}\right\}\\
&\le C n^{-3} \left\{1+\sum_{j\in\bZ}\sum_{k=0}^{\infty}\left(\sum_{i\in \bZ}\left[p^k(0,j+i)-p^k(0,j+i+1)\right]\left[\mP\bigl(A_{1,i}\bigr)-\mP\bigl(A_{2,i}\bigr)\right]\right)^{2}\right\}^6.
\end{align*}

Furthermore,
\begin{align}
&\sum_{j\in\bZ}\sum_{k=0}^{\infty}\left(\sum_{i\in \bZ}\left[p^k(0,j+i)-p^k(0,j+i+1)\right]\left[\mP\bigl(A_{1,i}\bigr)-\mP\bigl(A_{2,i}\bigr)\right]\right)^{2}\notag\\
&=\sum_{j\in\bZ}\sum_{k=0}^{\infty}\sum_{i_1,i_2\in \bZ}\left[p^k(0,j+i_1)-p^k(0,j+i_1+1)\right]\cdot\left[p^k(0,j+i_2)-p^k(0,j+i_2+1)\right]\notag\\
&\quad\quad\quad\quad\quad\quad\quad\cdot\left[\mP\bigl(A_{1,i_1}\bigr)-\mP\bigl(A_{2,i_1}\bigr)\right]\left[\mP\bigl(A_{1,i_2}\bigr)-\mP\bigl(A_{2,i_2}\bigr)\right]\notag\\
&=\sum_{i\in \bZ}\sum_{\ell\in\bZ}\sum_{j\in\bZ}\sum_{k=0}^{\infty}\left[p^k(0,j+i)-p^k(0,j+i+1)\right]\cdot\left[p^k(0,j+i+\ell)-p^k(0,j+i+\ell+1)\right]\notag\\
&\quad\quad\quad\quad\quad\quad\quad\cdot\left[\mP\bigl(A_{1,i}\bigr)-\mP\bigl(A_{2,i}\bigr)\right]\cdot\left[\mP\bigl(A_{1,i+\ell}\bigr)-\mP\bigl(A_{2,i+\ell}\bigr)\right]\notag\\
&=\sum_{i\in \bZ}\sum_{\ell\in\bZ}\left[a(\ell-1)+a(\ell+1)-2a(\ell)\right] \cdot \left[\mP\bigl(A_{1,i}\bigr)-\mP\bigl(A_{2,i}\bigr)\right]\cdot\left[\mP\bigl(A_{1,i+\ell}\bigr)-\mP\bigl(A_{2,i+\ell}\bigr)\right]\notag\\
&\le \sum_{\ell\in\bZ}\left[a(\ell-1)+a(\ell+1)-2a(\ell)\right] \sum_{i\in\bZ}\frac{1}{2}\left\{\left[\mP\bigl(A_{1,i}\bigr)-\mP\bigl(A_{2,i}\bigr)\right]^2+\left[\mP\bigl(A_{1,i+\ell}\bigr)-\mP\bigl(A_{2,i+\ell}\bigr)\right]^2\right\}\notag\\
&= \frac{1}{\sigma_1^2}\sum_{i\in\bZ}\left[\mP\bigl(A_{1,i}\bigr)-\mP\bigl(A_{2,i}\bigr)\right]^2\le  \frac{1}{\sigma_1^2}\sum_{i\in\bZ}\left|\mP\bigl(A_{1,i}\bigr)-\mP\bigl(A_{2,i}\bigr)\right|\le \frac{1}{\sigma_1^2}\sum_{i\in\bZ}\left[\mP\bigl(A_{1,i}\bigr)+\mP\bigl(A_{2,i}\bigr)\right].\notag
\end{align}
where the potential kernel $a(x)$ is defined in \eqref{t37-2} and the last equality is from Lemma \ref{lmmA-4} in the Appendix.\par

Therefore, we can find constant $C<\infty$, such that

$$\bE \left\{\left[\overline{S}_n(t,r)-\overline{S}_n(s,q)\right]^{12}\right\}\le Cn^{-3}\left\{1+\sum_{i\in\bZ}\left[\mP\bigl(A_{1,i}\bigr)+\mP\bigl(A_{2,i}\bigr)\right]\right\}^6.$$

Case (c):
Assume $\{\eta_0(x)\}_{x\in\bZ}$ is a strongly mixing stationary sequence satisfying the following condition. There exists $\varepsilon_0>0$ such that $\bE\left[|\eta_0(0)|^{12+\varepsilon_0}\right]<\infty$ and the strong mixing coefficients $\{\alpha(x)\}_{x\in\bZ^+}$ should have $\sum_{i=0}^{\infty}(i+1)^{10+132/\varepsilon_0}\alpha(i)<\infty$.\par

We will use the following bound borrowed from \cite{27} (see Theorem 2.2 and the derivation of equation C.6).
\begin{lemma}\label{lmm3-13}
Let $m>0$ be an integer and $\{X_i\}_{i\in\bN}$ be a sequence of centered real valued random variables with finite momemts of order $2m$. Let $S_n=\sum_{k=1}^nX_k$. Then there exists two positive constants $a_m,b_m<\infty$ such that 
\begin{align}
\bE \left(S_n^{2m}\right)\le& a_m\left(\int_0^1\sum_{k=1}^n[\alpha^{-1}(u)\wedge n]Q_k^2(u)du\right)^m\notag\\
&\quad\quad+ b_m\sum_{k=1}^n\int_0^1[\alpha^{-1}(u)\wedge n]^{2m-1}Q_k^{2m}(u)du,\notag
\end{align}
where $Q_k(u)$ is the quantile function of $|X_k|$, $\alpha^{-1}(u)=\sum_{i\ge 0}\ind{\{u< \alpha(i)\}}$ and $\{\alpha(k)\}_{k\ge 0}$ are the strong mixing coeffients of $\{X_i\}_{i\in\bN}$.\par
Furthermore, in general, for $r>p\ge 1$, suppose $\{X_i\}_{i\in\bN}$ have finite $r$th moment. Then, there exists a constant $c_p<\infty$ such that
\begin{align}
&\sum_{k=1}^n\int_0^1[\alpha^{-1}(u)\wedge n]^{p-1}Q_k^{p}(u)du \notag\\
&\le c_p \left(\sum_{i=0}^n(i+1)^{(pr-2r+p)/(r-p)}\alpha(i)\right)^{1-p/r}\sum_{k=1}^n\left(\bE|X_k|^r\right)^{p/r}.\label{tight-6c}
\end{align}
\end{lemma}

Let us denote $k_n=\#\{i\in\bZ: \mP\bigl(A_{1,i}\bigr)-\mP\bigl(A_{2,i}\bigr)\neq 0\}$. Then $k_n=O(n)$ due to assumption \eqref{weight-assump-1}. Based on Lemma \ref{lmm3-13} above, let $Q_i(u)$ be the quantile function of $\bigl|\eta_0(-i)-\mu_0\bigr|\cdot\left|\mP\bigl(A_{1,i}\bigr)-\mP\bigl(A_{2,i}\bigr)\right|$, $i\in\bZ$ and $m=6$, we have
\begin{align*} 
&\bE \left\{\left[\overline{S}_n(t,r)-\overline{S}_n(s,q)\right]^{12}\right\}\\
&\le Cn^{-3}\left[\left(\sum_{i\in\bZ}\int_0^1[\alpha^{-1}(u)\wedge k_n]Q_i^2(u)du\right)^6+\sum_{j\in\bZ}\int_0^1[\alpha^{-1}(u)\wedge k_n]^{11}Q_j^{12}(u)du\right].
\end{align*}

Let $p=2, r=12$ in \eqref{tight-6c}, we can get an upper bound for $\sum_{i\in\bZ}\int_0^1[\alpha^{-1}(u)\wedge k_n]Q_i^2(u)du$,
\begin{align*}
&\sum_{i\in\bZ}\int_0^1[\alpha^{-1}(u)\wedge k_n]Q_i^2(u)du\\
&\le c_{2} \left(\sum_{i=0}^{k_n}(i+1)^{1/5}\alpha(i)\right)^{5/6}\sum_{j\in\bZ}\left(\bE|\eta_0(-j)-\mu_0\bigr|^{12}\right)^{1/6}\left[\mP\bigl(A_{1,j}\bigr)-\mP\bigl(A_{2,j}\bigr)\right]^2\\
&\le C \left(\sum_{i=0}^{k_n}(i+1)^{1/5}\alpha(i)\right)^{5/6}\sum_{j\in\bZ}\left[\mP\bigl(A_{1,j}\bigr)+\mP\bigl(A_{2,j}\bigr)\right]\le C \sum_{j\in\bZ}\left[\mP\bigl(A_{1,j}\bigr)+\mP\bigl(A_{2,j}\bigr)\right].
\end{align*}
By the same token, let $p=12$, $r=12+\varepsilon_0$ in \eqref{tight-6c}, we can show that
\begin{align*}
&\sum_{j\in\bZ}\int_0^1[\alpha^{-1}(u)\wedge k_n]^{11}Q_j^{12}(u)du\\
&\le C \left(\sum_{i=0}^{k_n}(i+1)^{10+132/\varepsilon_0}\alpha(i)\right)^{\varepsilon_0/(12+\varepsilon_0)}\sum_{j\in\bZ}\left[\mP\bigl(A_{1,j}\bigr)+\mP\bigl(A_{2,j}\bigr)\right]\\
&\le C \sum_{j\in\bZ}\left[\mP\bigl(A_{1,j}\bigr)+\mP\bigl(A_{2,j}\bigr)\right].
\end{align*}  
Hence, 
\begin{align*}
&\bE \left\{\left[\overline{S}_n(t,r)-\overline{S}_n(s,q)\right]^{12}\right\}\\
&\le Cn^{-3}\left\{\left(\sum_{j\in\bZ}\left[\mP\bigl(A_{1,j}\bigr)+\mP\bigl(A_{2,j}\bigr)\right]\right)^6+\sum_{j\in\bZ}\left[\mP\bigl(A_{1,j}\bigr)+\mP\bigl(A_{2,j}\bigr)\right]\right\}\\
&\le Cn^{-3}\left\{1+\sum_{m\in\bZ}\left[\mP\bigl(A_{1,m}\bigr)+\mP\bigl(A_{2,m}\bigr)\right]\right\}^6. 
\end{align*}

The proof of Lemma \ref{initial-increments-bounds} is complete.
%
\end{proof}
 
Now let us bound the summation $\sum_{m\in\bZ}\left[\mP\bigl(A_{1,m}\bigr)+\mP\bigl(A_{2,m}\bigr)\right]$. 
\begin{lemma}
$\exists$ $C<\infty$, s.t.
\be \sum_{m\in\bZ}\left[\mP\bigl(A_{1,m}\bigr)+\mP\bigl(A_{2,m}\bigr)\right]\le C \left[\sqrt{(t-s)n}+|r-q|\sqrt{n}+1\right].\label{tight-10}\ee
\end{lemma}
\begin{proof}
Suppose $t\ge s$. Note that
\begin{align}
\sum_{m\in\bZ}\mP\bigl(A_{1,m}\bigr)=&\sum_{m\in\bZ}\mP\bigl(X_{\lfloor nt\rfloor}^0\ge -\lfloor ntb\rfloor -\lfloor r\sqrt{n}\rfloor-m, X_{\lfloor ns\rfloor}^0< -\lfloor nsb\rfloor -\lfloor q\sqrt{n}\rfloor-m\bigr)\notag\\
=&\sum_{m\in\bZ}\sum_{\ell>m}\mP\bigl(X_{\lfloor ns\rfloor}^0=-\lfloor nsb\rfloor -\lfloor q\sqrt{n}\rfloor-\ell\bigr)\notag\\
&\quad\quad\quad\cdot\mP\bigl(X_{\lfloor nt\rfloor-\lfloor ns\rfloor}^0\ge \lfloor nsb\rfloor-\lfloor ntb\rfloor +\lfloor q\sqrt{n}\rfloor-\lfloor r\sqrt{n}\rfloor-m+\ell\bigr)\notag\\
\overset{k=\ell-m}{=}&\sum_{m\in\bZ}\sum_{k>0}\mP\bigl(X_{\lfloor ns\rfloor}^0=-\lfloor nsb\rfloor -\lfloor q\sqrt{n}\rfloor-k-m\bigr)\notag\\
&\quad\quad\quad\cdot\mP\bigl(X_{\lfloor nt\rfloor-\lfloor ns\rfloor}^0\ge \lfloor nsb\rfloor-\lfloor ntb\rfloor +\lfloor q\sqrt{n}\rfloor-\lfloor r\sqrt{n}\rfloor+k\bigr)\notag\\
=&\sum_{k>0}\mP\bigl(X_{\lfloor nt\rfloor-\lfloor ns\rfloor}^0\ge \lfloor nsb\rfloor-\lfloor ntb\rfloor +\lfloor q\sqrt{n}\rfloor-\lfloor r\sqrt{n}\rfloor+k\bigr).\label{tight-7}
\end{align} 
Similarly, one can show that
\be\sum_{m\in\bZ}\mP\bigl(A_{2,m}\bigr)=\sum_{k\le 0}\mP\bigl(X_{\lfloor nt\rfloor-\lfloor ns\rfloor}^0< \lfloor nsb\rfloor-\lfloor ntb\rfloor +\lfloor q\sqrt{n}\rfloor-\lfloor r\sqrt{n}\rfloor+k\bigr).\label{tight-8}\ee
Combine \eqref{tight-7} and \eqref{tight-8} together, 
\begin{align}
&\sum_{m\in\bZ}\left[\mP\bigl(A_{1,m}\bigr)+\mP\bigl(A_{2,m}\bigr)\right]\notag\\
&=\sum_{k>0}\mP\bigl(X_{\lfloor nt\rfloor-\lfloor ns\rfloor}^0-\lfloor nsb\rfloor+\lfloor ntb\rfloor -\lfloor q\sqrt{n}\rfloor+\lfloor r\sqrt{n}\rfloor\ge k\bigr)\notag\\
&\quad\quad+\sum_{k< 0}\mP\bigl(X_{\lfloor nt\rfloor-\lfloor ns\rfloor}^0-\lfloor nsb\rfloor+\lfloor ntb\rfloor -\lfloor q\sqrt{n}\rfloor+\lfloor r\sqrt{n}\rfloor\le k\bigr)\notag\\
&=\sum_{k>0}\mP\bigl(\bigl|X_{\lfloor nt\rfloor-\lfloor ns\rfloor}^0-\lfloor nsb\rfloor+\lfloor ntb\rfloor -\lfloor q\sqrt{n}\rfloor+\lfloor r\sqrt{n}\rfloor\bigr|\ge k\bigr)\notag\\
&=\mE\left| X_{\lfloor nt\rfloor-\lfloor ns\rfloor}^0-\lfloor nsb\rfloor+\lfloor ntb\rfloor -\lfloor q\sqrt{n}\rfloor+\lfloor r\sqrt{n}\rfloor\right|\notag\\
&\le \mE\left| X_{\lfloor nt\rfloor-\lfloor ns\rfloor}^0-\lfloor nsb\rfloor+\lfloor ntb\rfloor\right|+|r-q|\sqrt{n}+1\notag\\
&\le  C \left[\sqrt{(t-s)n}+|r-q|\sqrt{n}+1\right].\notag
\end{align}
\end{proof}

Combine \eqref{tight-6b} and \eqref{tight-10} together, we can get the following bound for the second term in \eqref{tight-4}:
\be \left[\bE\left(|\overline{S}_n(t,r)-\overline{S}_n(s,q)|^{12}\right)\right]^{1/12}\le Cn^{-1/4}\left[\left(\sqrt{|t-s|}+|r-q|\right)\sqrt{n}+1\right]^{1/2}.\label{tight-11}\ee

For the third term on the right of \eqref{tight-4}, from \eqref{add-1}, suppose $t\ge s$, 
\begin{align*}
&\overline{F}_n(t,r)-\overline{F}_n(s,q)\\
&=n^{-1/4}\sum_{k=1}^{\lfloor ns\rfloor}\sum_{x\in\mathbb{Z}}\xi_k(x)\left[\mP\bigl(X_{\lfloor nt\rfloor -k}^{\lfloor ntb\rfloor +\lfloor r\sqrt{n}\rfloor}=x\bigr)-\mP\bigl(X_{\lfloor ns\rfloor -k}^{\lfloor nsb\rfloor +\lfloor q\sqrt{n}\rfloor}=x\bigr)\right]\\
&\quad\quad+n^{-1/4}\sum_{k=\lfloor ns\rfloor+1}^{\lfloor nt\rfloor}\sum_{x\in\mathbb{Z}}\xi_k(x)\mP\bigl(X_{\lfloor nt\rfloor -k}^{\lfloor ntb\rfloor +\lfloor r\sqrt{n}\rfloor}=x\bigr).
\end{align*}
From \eqref{12th-moment-bound},
\begin{align}
&\bE\left\{\left[\overline{F}_n(t,r)-\overline{F}_n(s,q)\right]^{12}\right\}\notag\\
&\le Cn^{-3}\Biggl\{1+\sum_{k=1}^{\lfloor ns\rfloor}\sum_{x\in\mathbb{Z}}\left[\mP\bigl(X_{\lfloor nt\rfloor -k}^{\lfloor ntb\rfloor +\lfloor r\sqrt{n}\rfloor}=x\bigr)-\mP\bigl(X_{\lfloor ns\rfloor -k}^{\lfloor nsb\rfloor +\lfloor q\sqrt{n}\rfloor}=x\bigr)\right]^2\notag\\
&\quad\quad\quad\quad+\sum_{k=\lfloor ns\rfloor+1}^{\lfloor nt\rfloor}\sum_{x\in\mathbb{Z}}\mP\bigl(X_{\lfloor nt\rfloor -k}^{\lfloor ntb\rfloor +\lfloor r\sqrt{n}\rfloor}=x\bigr)^2\Biggr\}^6.\label{tight-11b}
\end{align}
For the two summations on the right of \eqref{tight-11b},
\begin{align}
&\sum_{k=1}^{\lfloor ns\rfloor}\sum_{x\in\mathbb{Z}}\left[\mP\bigl(X_{\lfloor nt\rfloor -k}^{\lfloor ntb\rfloor +\lfloor r\sqrt{n}\rfloor}=x\bigr)-\mP\bigl(X_{\lfloor ns\rfloor -k}^{\lfloor nsb\rfloor +\lfloor q\sqrt{n}\rfloor}=x\bigr)\right]^2+\sum_{k=\lfloor ns\rfloor+1}^{\lfloor nt\rfloor}\sum_{x\in\mathbb{Z}}\mP\bigl(X_{\lfloor nt\rfloor -k}^{\lfloor ntb\rfloor +\lfloor r\sqrt{n}\rfloor}=x\bigr)^2\notag\\
&=\sum_{k=1}^{\lfloor nt\rfloor}\mP\bigl(X_{\lfloor nt\rfloor -k}^{\lfloor ntb\rfloor +\lfloor r\sqrt{n}\rfloor}-\tilde{X}_{\lfloor nt\rfloor -k}^{\lfloor ntb\rfloor +\lfloor r\sqrt{n}\rfloor}=0\bigr)+\sum_{k=1}^{\lfloor ns\rfloor}\mP\bigl(X_{\lfloor ns\rfloor -k}^{\lfloor nsb\rfloor +\lfloor q\sqrt{n}\rfloor}-\tilde{X}_{\lfloor ns\rfloor -k}^{\lfloor nsb\rfloor +\lfloor q\sqrt{n}\rfloor}=0\bigr)\notag\\
&\quad\quad-2\sum_{k=1}^{\lfloor ns\rfloor}\mP\bigl(X_{\lfloor nt\rfloor -k}^{\lfloor ntb\rfloor +\lfloor r\sqrt{n}\rfloor}-\tilde{X}_{\lfloor ns\rfloor -k}^{\lfloor nsb\rfloor +\lfloor q\sqrt{n}\rfloor}=0\bigr)\notag\\
&=\sum_{k=0}^{\lfloor nt\rfloor-1}q^k(0,0)+\sum_{k=0}^{\lfloor ns\rfloor -1}q^k(0,0)-2\mE\left\{\sum_{k=0}^{\lfloor ns\rfloor-1} q^k\bigl(X_{\lfloor nt\rfloor -\lfloor ns\rfloor}^{\lfloor ntb\rfloor +\lfloor r\sqrt{n}\rfloor}-\lfloor nsb\rfloor -\lfloor q\sqrt{n}\rfloor,0\bigr)\right\}\notag\\
&=\sum_{k=\lfloor ns\rfloor}^{\lfloor nt\rfloor-1}q^k(0,0)+2\mE\left\{\sum_{k=0}^{\lfloor ns\rfloor-1}\left[q^k(0,0)-q^k\bigl(X_{\lfloor nt\rfloor -\lfloor ns\rfloor}^{\lfloor ntb\rfloor +\lfloor r\sqrt{n}\rfloor}-\lfloor nsb\rfloor -\lfloor q\sqrt{n}\rfloor,0\bigr)\right]\right\}.\label{tight-13}
\end{align}
From \eqref{b2}, the first term in \eqref{tight-13} can be bounded by
\be\sum_{k=\lfloor ns\rfloor}^{\lfloor nt\rfloor-1}q^k(0,0)\le \sum_{k=\lfloor ns\rfloor}^{\lfloor nt\rfloor-1}\frac{C}{\sqrt{k}}\le\sum_{k=1}^{\lfloor nt\rfloor-\lfloor ns\rfloor}\frac{C}{\sqrt{k}}\le C\left[1+\sqrt{(t-s)n}\right].\label{tight-14}\ee
By inequality \eqref{qprop1}, the second term in \eqref{tight-13} is bounded by  
\begin{align}
&\mE\left\{\sum_{k=0}^{\lfloor ns\rfloor-1}\left[q^k(0,0)-q^k\bigl(X_{\lfloor nt\rfloor -\lfloor ns\rfloor}^{\lfloor ntb\rfloor +\lfloor r\sqrt{n}\rfloor}-\lfloor nsb\rfloor -\lfloor q\sqrt{n}\rfloor,0\bigr)\right]\right\}\notag\\
&\le\mE \left[a\left(X_{\lfloor nt\rfloor -\lfloor ns\rfloor}^{\lfloor ntb\rfloor +\lfloor r\sqrt{n}\rfloor}-\lfloor nsb\rfloor -\lfloor q\sqrt{n}\rfloor\right)\right]\le C\mE\left|X_{\lfloor nt\rfloor -\lfloor ns\rfloor}^{\lfloor ntb\rfloor +\lfloor r\sqrt{n}\rfloor}-\lfloor nsb\rfloor -\lfloor q\sqrt{n}\rfloor\right|\notag\\
&\le C\left[(\sqrt{|t-s|}+|r-q|)\sqrt{n}+1\right],\label{tight-15}
\end{align}
where $a(x)$ is defined in \eqref{t37-2} and the second inequality is due to \eqref{nl38}.\par

Combine \eqref{tight-14} and \eqref{tight-15}, we can get a bound for \eqref{tight-13}. And therefore the third term on the right of \eqref{tight-4} can be bounded by
\be \left[\bE\left(|\overline{F}_n(t,r)-\overline{F}_n(s,q)|^{12}\right)\right]^{1/12}\le Cn^{-1/4}\left[\left(\sqrt{|t-s|}+|r-q|\right)\sqrt{n}+1\right]^{1/2}.\label{tight-16}\ee

As a conclusion from \eqref{tight-n5}, \eqref{tight-11} and \eqref{tight-16}, there exists constant $C<\infty$ such that
\begin{align}
\bE\left(|\fluc_n(t,r)-\fluc_n(s,q)|^{12}\right)\le& Cn^{-3}\left[\left(\sqrt{|t-s|}+|r-q|\right)\sqrt{n}+1\right]^{6}\notag\\
\le& C\left(|t-s|^3+|r-q|^6+n^{-3}\right).\label{tight-17}
\end{align}
This bound has exactly the same form as the one found in \cite{11}.\par

Since $t,s,r,q\in [0,1]$ are bounded and $2<\kappa<3$, we can further bound \eqref{tight-17} by
\be\bE\left(|\fluc_n(t,r)-\fluc_n(s,q)|^{12}\right)\le C\left(|t-s|^\kappa+|r-q|^\kappa+n^{-3}\right).\label{tight-18}\ee

Note that $\delta_n^\kappa=n^{-\gamma\kappa}\ge n^{-3}$ and $|(t,r)-(s,q)|> \delta_n$, we can get
 \be\bE\left(|\fluc_n(t,r)-\fluc_n(s,q)|^{12}\right)\le C\left(|t-s|^\kappa+|r-q|^\kappa\right),\label{tight-19}\ee
which proves Lemma \ref{lmm3-12}.
\end{proof}

The second tightness condition can be verified as following.\par

 \begin{lemma}\label{lmm3-14}
Under the settings in Lemma \ref{lmm3-12}, for any fixed $1<\gamma<3/2$ and $\epsilon>0$,  
$$\lim_{n\to\infty}\bP\left(\sup_{\begin{subarray}{l} (t,r),(s,q)\in [0,1]^2\\ \|(t,r)-(s,q)\|<n^{-\gamma}\end{subarray}}|\fluc_n(t,r)-\fluc_n(s,q)|>\epsilon\right)=0.$$
\end{lemma}

\begin{proof}
Let us define the interval $I(k):=[(k-1)n^{-\gamma},(k+1)n^{-\gamma}]\cap [0,1]$. For all fixed $\epsilon>0$, 
\begin{align}
&\bP\left(\sup_{\begin{subarray}{l} (t,r),(s,q)\in [0,1]^2\\ \|(t,r)-(s,q)\|<n^{-\gamma}\end{subarray}}|\fluc_n(t,r)-\fluc_n(s,q)|>\epsilon\right)\notag\\
&\le\bP\left(\bigcup_{k_1=1}^{\lfloor n^\gamma\rfloor}\bigcup_{k_2=1}^{\lfloor n^\gamma\rfloor}\left\{\sup_{ t\in I(k_1), r\in I(k_2)}|\fluc_n(t,r)-\fluc_n(k_1n^{-\gamma},k_2n^{-\gamma})|\ge\frac{\epsilon}{2}\right\}\right)\notag\\
&\le\sum_{k_1=1}^{\lfloor n^\gamma\rfloor}\sum_{k_2=1}^{\lfloor n^\gamma\rfloor}\bP\left(\sup_{t\in I(k_1), r\in I(k_2)}|\fluc_n(t,r)-\fluc_n(k_1n^{-\gamma},k_2n^{-\gamma})|\ge\frac{\epsilon}{2}\right).\label{tight-20}
\end{align}
Similar as before, we can break $\fluc_n(\cdot,\cdot)$ into three parts. Since 
\begin{align*}
&|\fluc_n(t,r)-\fluc_n(k_1n^{-\gamma},k_2n^{-\gamma})|\\
&\le|\mu_0\overline{H}_n(t,r)-\mu_0\overline{H}_n(k_1n^{-\gamma},k_2n^{-\gamma})|+|\overline{S}_n(t,r)-\overline{S}_n(k_1n^{-\gamma},k_2n^{-\gamma})|\\
&\quad\quad+|\overline{F}_n(t,r)-\overline{F}_n(k_1n^{-\gamma},k_2n^{-\gamma})|,
\end{align*}
one can further bound \eqref{tight-20} by three terms.
\begin{align}
&\bP\left(\sup_{\begin{subarray}{l} (t,r),(s,q)\in [0,1]^2\\ \|(t,r)-(s,q)\|<n^{-\gamma}\end{subarray}}|\fluc_n(t,r)-\fluc_n(s,q)|>\epsilon\right)\notag\\
&\le\sum_{k_1=1}^{\lfloor n^\gamma\rfloor}\sum_{k_2=1}^{\lfloor n^\gamma\rfloor}\bP\left(\sup_{t\in I(k_1), r\in I(k_2)}|\mu_0\overline{H}_n(t,r)-\mu_0\overline{H}_n(k_1n^{-\gamma},k_2n^{-\gamma})|\ge\frac{\epsilon}{6}\right)\label{tight-21}\\
&\quad\quad+\sum_{k_1=1}^{\lfloor n^\gamma\rfloor}\sum_{k_2=1}^{\lfloor n^\gamma\rfloor}\bP\left(\sup_{t\in I(k_1), r\in I(k_2)}|\overline{S}_n(t,r)-\overline{S}_n(k_1n^{-\gamma},k_2n^{-\gamma})|\ge\frac{\epsilon}{6}\right)\label{tight-22}\\
&\quad\quad+\sum_{k_1=1}^{\lfloor n^\gamma\rfloor}\sum_{k_2=1}^{\lfloor n^\gamma\rfloor}\bP\left(\sup_{t\in I(k_1), r\in I(k_2)}|\overline{F}_n(t,r)-\overline{F}_n(k_1n^{-\gamma},k_2n^{-\gamma})|\ge\frac{\epsilon}{6}\right).\label{tight-23}
\end{align}

For the first term \eqref{tight-21}, we have observed that $\overline{H}_n(t,r)$ can be uniformly bounded by $Cn^{-1/4}$. Thus, for large enough $n$, $$\bP\left(\sup_{t\in I(k_1), r\in I(k_2)}|\mu_0\overline{H}_n(t,r)-\mu_0\overline{H}_n(k_1n^{-\gamma},k_2n^{-\gamma})|\ge\frac{\epsilon}{6}\right)=0.$$
As a result, the summation in \eqref{tight-21} vanishes as $n\to\infty$.\par

For the second term \eqref{tight-22} and third term \eqref{tight-23}, recall from \eqref{tight-6},
\begin{align}
&\overline{S}_n(t,r)-\overline{S}_n(k_1n^{-\gamma},k_2n^{-\gamma})\notag\\
&=n^{-1/4}\sum_{i\in \bZ}\bigl(\eta_0(-i)-\mu_0\bigr)\left[\mP\left(A_{1,i}(t,r,k_1n^{-\gamma},k_2n^{-\gamma})\right)-\mP\left(A_{2,i}(t,r,k_1n^{-\gamma},k_2n^{-\gamma})\right)\right],\label{tight-28}
\end{align}
where
\begin{align*}
A_{1,i}(t,r,s,q)=&\left\{X_{\lfloor nt\rfloor}^i\ge -\lfloor ntb\rfloor -\lfloor r\sqrt{n}\rfloor, X_{\lfloor ns \rfloor}^i< -\lfloor nsb\rfloor -\lfloor q\sqrt{n}\rfloor\right\},\\
A_{2,i}(t,r,s,q)=&\left\{X_{\lfloor nt\rfloor}^i< -\lfloor ntb\rfloor -\lfloor r\sqrt{n}\rfloor, X_{\lfloor  ns\rfloor}^i\ge -\lfloor  nsb\rfloor -\lfloor q\sqrt{n}\rfloor\right\}.
\end{align*}
And 
\begin{align*}
\overline{F}_n(t,r)-\overline{F}_n(k_1n^{-\gamma},k_2n^{-\gamma})=&n^{-1/4}\sum_{k=1}^{\lfloor nt\rfloor}\sum_{x\in\mathbb{Z}}\xi_k(x)\mP\bigl(X_{\lfloor nt\rfloor -k}^{\lfloor ntb\rfloor +\lfloor r\sqrt{n}\rfloor}=x\bigr)\\
&\quad\quad-n^{-1/4}\sum_{k=1}^{\lfloor k_1n^{1-\gamma}\rfloor}\sum_{x\in\mathbb{Z}}\xi_k(x)\mP\bigl(X_{\lfloor k_1n^{1-\gamma}\rfloor -k}^{\lfloor k_1n^{1-\gamma}b\rfloor +\lfloor k_2n^{1/2-\gamma}\rfloor}=x\bigr).
\end{align*}

 Note that for large enough $n$, and any $t\in I(k_1)$, $r\in I(k_2)$, 
\begin{align*}
|nt-k_1n^{1-\gamma}|=&n|t-k_1n^{-\gamma}|\le n^{1-\gamma}<1/2,\\
|ntb-k_1n^{1-\gamma}b|=&n|b|\cdot|t-k_1n^{-\gamma}|\le |b|n^{1-\gamma}<1/2,\\
|r\sqrt{n}-k_2n^{1/2-\gamma}|=&n^{1/2}|r-k_2n^{-\gamma}|\le n^{1/2-\gamma}<1/2.
\end{align*}
This means that for $t\in I(k_1)$ and $r\in I(k_2)$, each one of $\lfloor nt\rfloor$, $\lfloor ntb\rfloor$ and $\lfloor r\sqrt{n}\rfloor$ can only have at most one jump. For example, $\lfloor nt\rfloor$ can only jump from $\lfloor k_1n^{1-\gamma}\rfloor-1$ to $\lfloor k_1n^{1-\gamma}\rfloor$ or from $\lfloor k_1n^{1-\gamma}\rfloor$ to $\lfloor k_1n^{1-\gamma}\rfloor+1$. As a result, $\overline{S}_n(t,r)$ and $\overline{F}_n(t,r)$ can only have at most 8 values on $I(k_1)\times I(k_2)$. \par

Suppose that $\overline{S}_n(t,r)$ has exactly $m_0\le 8$ values on the interval $I(k_1)\times I(k_2)$. Let us pick one point from each value of $\overline{S}_n(t,r)$ on $I(k_1)\times I(k_2)$ and denote them by $(t_i,r_i)$, $i=1,\ldots,m_0$. Then,
\begin{align*}
&\bP\left(\sup_{t\in I(k_1), r\in I(k_2)}|\overline{S}_n(t,r)-\overline{S}_n(k_1n^{-\gamma},k_2n^{-\gamma})|\ge\frac{\epsilon}{6}\right)\\
&\le \sum_{i=1}^{m_0}\bP\left(|\overline{S}_n(t_i,r_i)-\overline{S}_n(k_1n^{-\gamma},k_2n^{-\gamma})|\ge\frac{\epsilon}{6}\right)\le C\sum_{i=1}^{m_0}\frac{\bE|S_n(t_i,r_i)-S_n(k_1n^{-\gamma},k_2n^{-\gamma})|^{12}}{n^3}\\
&\le C\sum_{i=1}^{m_0}\frac{\left[\left(\sqrt{|t_i-k_1n^{-\gamma}|}+|r_i-k_2n^{-\gamma}|\right)\sqrt{n}+1\right]^{6}}{n^{3}}\le C\frac{\left[\left(n^{-\gamma/2}+n^{-\gamma}\right)\sqrt{n}+1\right]^{6}}{n^{3}}\le Cn^{-3}.
\end{align*}
where the second inequality is from Markov inequality, and the third inequality comes from \eqref{tight-11}.\par
 
Therefore, for any $1<\gamma<3/2$, we have
\begin{align}
&\sum_{k_1=1}^{\lfloor n^\gamma\rfloor}\sum_{k_2=1}^{\lfloor n^\gamma\rfloor}\bP\left(\sup_{\begin{subarray}{l} t\in I(k_1)\\ r\in I(k_2)\end{subarray}}|\overline{S}_n(t,r)-\overline{S}_n(k_1n^{-\gamma},k_2n^{-\gamma})|\ge\frac{\epsilon}{6}\right)\notag\\
&\le C\sum_{k_1=1}^{\lfloor n^\gamma\rfloor}\sum_{k_2=1}^{\lfloor n^\gamma\rfloor}n^{-3}\le Cn^{2\gamma-3}\to 0,\quad\mbox{as }n\to\infty.\notag
\end{align}

From inequality \eqref{tight-16}, one can use the same method to show that 
$$\sum_{k_1=1}^{\lfloor n^\gamma\rfloor}\sum_{k_2=1}^{\lfloor n^\gamma\rfloor}\bP\left(\sup_{\begin{subarray}{l} t\in I(k_1)\\ r\in I(k_2)\end{subarray}}|\overline{F}_n(t,r)-\overline{F}_n(k_1n^{-\gamma},k_2n^{-\gamma})|\ge\frac{\epsilon}{6}\right)\le Cn^{2\gamma-3}\to 0,\quad\mbox{as }n\to\infty.$$ 

In sum, we have proved that \eqref{tight-21}, \eqref{tight-22} and \eqref{tight-23} will vanish when $n$ goes to $\infty$, and hence Lemma \ref{lmm3-14} has been proved.  
\end{proof}
The proof for Theorem \ref{thm2-6} is complete.
\end{proof}

%

\appendix       
\chapter{Potential Kernel}\label{potential-kernel}
Let $q(x,y)$ be the transition kernel defined in \eqref{21-d2}. For the following, we assume \eqref{weight-assump-1} and \eqref{weight-assump-2}. P28.8 in \cite{4} shows that the potential kernel 
$$a(x)=\sum_{k=0}^\infty [q^k(0,0)-q^k(x,0)]$$
is well-defined for every $x\in\bZ$. And it has the following properties.
\begin{lemma}\label{lmmA-1}
The potential kernel $a(x)$ is an even function with order $|x|$ as $x\to \infty$. To be specific,
\be \lim_{x\to +\infty}\frac{a(x)}{x}=\frac{1}{2\sigma_1^2}.\label{nl38}\ee 
\end{lemma}
\begin{proof}
See P28.4 in \cite{4} (page 345, Chapter \uppercase\expandafter{\romannumeral7}).
\end{proof}

\begin{lemma}\label{lmmA-2}
For $k\in\bZ$, we have
\be
\lim_{x\to +\infty}\bigl[a(x+k)-a(x)\bigr]=\frac{k}{2\sigma_1^2};\label{l38a}\ee
\be
\lim_{x\to -\infty}\bigl[a(x+k)-a(x)\bigr]=-\frac{k}{2\sigma_1^2}.\label{l38b}\ee
\end{lemma}

\begin{proof}
See P29.2 in \cite{4} (page 354, Chapter \uppercase\expandafter{\romannumeral7}).
\end{proof}

\begin{lemma}\label{lmmA-3}
The potential kernel $a(x)$ satisfies the following equations:
\begin{enumerate}[(a)]
\item
\be
\sum_{j\in\bZ}q(i,j)a(k-j)=a(k-i)+\ind\{i=k\},\quad i,k\in\bZ,\label{l39b}\ee
\item
\be
a(x-1)+a(x+1)-2a(x)=\frac{1}{\pi}\int_{-\pi}^{\pi}\frac{1-\cos(\theta)}{1-\phi_Y(\theta)}e^{\i x\theta} d\theta,\quad x\in\bZ,\label{l39c}\ee
where $\phi_Y(\theta)=\sum_{x\in\bZ}q(0,x)e^{\i\theta x}$.
\item There exist positive constant $A, c<\infty$ such that for all $x\in \bZ$,
\be 
\abs{a(x-1)+a(x+1)-2a(x)}\le Ae^{-c\abs{x}}. \label{l39e}\ee
\end{enumerate}
\end{lemma}
\begin{remark}\label{rmkA-1}
Note that we can also write \eqref{l39c} in terms of the characteristic function of transition $p$. Let us denote $\phi_X(\theta)=\sum_{x\in\bZ}p(0,x)e^{\i\theta x}$. Then 
\be
a(x-1)+a(x+1)-2a(x)=\frac{1}{\pi}\int_{-\pi}^{\pi}\frac{1-\cos(\theta)}{1-|\phi_X(\theta)|^2}e^{\i x\theta} d\theta,\quad x\in\bZ.\label{l39d}\ee
\end{remark}
\begin{proof}
For the first equation, see T28.1 in \cite{4} (page 352, Chapter \uppercase\expandafter{\romannumeral7}).\par

The second equation is a simple application of the inversion formula. The detail can be found in the proof of P29.5 in \cite{4} (page 355, Chapter \uppercase\expandafter{\romannumeral7}).\par

For the third part, let $h(\theta)=\frac{1-\cos(\theta)}{1-\phi_Y(\theta)}$, $\theta\in [-\pi,\pi]$. From \eqref{l39c}, we can see that $h$ is a real function. Also, $h$ is a periodic funcion with period $2\pi$. We will show that $h$ is indeed analytic.\par


Let us naturally extend the function $h$ to the complex plane, i.e.
$$h(z)=\frac{1-\cos(z)}{1-\phi_Y(z)},
\quad z\in\bC.$$

Since $1-\cos(z)$ and $1-\phi_Y(z)$ are entire functions (analytic over the whole complex plane), $h(z)$ is meromorphic on the whole complex plane and point $z=0$ is its pole.\par

Also note that $q$ has span 1, thus we can show that for $\theta\in[-\pi,\pi]$, $\phi_Y(\theta)=1$ if and only if $\theta=0$ (see Lemma \ref{lmmB-1} in the Appendix).\par 

Therefore, $h(z)$ is analytic on $[-\pi,\pi]\setminus\{0\}$, and hence we only need to show that $z=0$ is a removable singularity.\par  

Note that $1-\cos(z)$ and $1-\phi_Y(z)$ have the following Taylor expansion:
\begin{align*}
1-\cos(z)=&\frac{1}{2}z^2+\sum_{k=2}^{\infty}\frac{(-1)^{k+1}}{(2k)!}z^{2k},\quad z\in\bC.\\
1-\phi_Y(z)=&\sigma_1^2z^2+\sum_{k=2}^\infty\frac{(-1)^{k+1}m_{2k}}{(2k)!}z^{2k},
\quad z\in\bC.
\end{align*}
where $m_{k}=\sum_{x\in\bZ}x^{k}q(0,x)$ (note that $m_{2k-1}=0$ since $q$ is symmetric).\par
 
Thus, we have the following limit,
\begin{align*}
\lim_{z\to 0}h(z)=&\lim_{z\to 0}\frac{1-\cos(z)}{1-\phi_Y(z)}=\lim_{z\to 0}\frac{\frac{1}{2}z^2+\sum_{k=2}^{\infty}\frac{(-1)^{k+1}}{(2k)!}z^{2k}}{\sigma_1^2z^2+\sum_{k=2}^\infty\frac{(-1)^{k+1}m_{2k}}{(2k)!}z^{2k}}\\
=&\lim_{z\to 0}\frac{\frac{1}{2}+\sum_{k=2}^{\infty}\frac{(-1)^{k+1}}{(2k)!}z^{2k-2}}{\sigma_1^2+\sum_{k=2}^\infty\frac{(-1)^{k+1}m_{2k}}{(2k)!}z^{2k-2}}=\frac{1}{2\sigma_1^2}.
\end{align*}

Therefore, $h(z)$ is analytic. And by \eqref{l39c}, we can think of $a(x-1)+a(x+1)-2a(x)$ as the $x$th Fourier coefficient of $h$. From results in Fourier Analysis (see Proposition 1.2.20 in \cite{9}), we conclude that the decay of $a(x-1)+a(x+1)-2a(x)$ is exponentially fast.\par

Thus, the proof for Lemma \ref{lmmA-3} is complete.
\end{proof}

\begin{lemma}\label{lmmA-4}
The series $\sum_{j\in\bZ}\bigl[a(j-1)+a(j+1)-2a(j)\bigr]$ is absolutely convergent and
\be
\sum_{j\in\bZ}\bigl[a(j-1)+a(j+1)-2a(j)\bigr]=\frac{1}{\sigma_1^2}.\label{l39a}\ee
\end{lemma}
\begin{proof}

The first part is a direct result from \eqref{l39e}.\par

For the second part, for any fixed $M,N>0$,
$$\sum_{j=-M}^{N}\bigl[a(j-1)+a(j+1)-2a(j)\bigr]=a(N+1)-a(N)+a(-M-1)-a(-M).$$
Let $M,N\to\infty$, and by \eqref{l38a} and \eqref{l38b}, we have
\begin{align*}
\sum_{j\in\bZ}\bigl[a(j-1)+a(j+1)-2a(j)\bigr]=&\lim_{N\to\infty}\bigl[a(N+1)-a(N)\bigr]+\lim_{M\to\infty}\bigl[a(-M-1)-a(-M)\bigr]\\
=&\frac{1}{2\sigma_1^2}+\frac{1}{2\sigma_1^2}=\frac{1}{\sigma_1^2}. \qedhere
\end{align*}
\end{proof}

\chapter{Local Central Limit Theorem}\label{lclt}
For the following, let us consider a transition probability $p(x,y)$ for a discrete-time random walk $\{X_t\}_{t\in\bZ^+}$ on $\bZ$. Throughout this section, we assume $p$ to be translate invariant, has finite range and span 1, i.e.
\be p(0,x)=p(y,x+y), \quad\forall x,y\in\bZ,\label{p-assump-1}\ee
\be \#\supp(p)<\infty,\label{p-assump-2}\ee
\be \max\{k\in\bZ^+:\exists\ell\in\bZ,\quad s.t.\quad\supp(p)\subset \ell+k\bZ\}=1,\label{p-assump-3}\ee
where $supp(p)=\{x\in\bZ: p(0,x)>0\}$.\par 

For convenience, we denote $p(x)=p(0,x)$. We also denote the mean and variance by $$\wm=\sum_{x\in\bZ}xp(x),\quad \sigma_1^2=\sum_{x\in\bZ}(x-\wm)^2p(x).$$

The following Local Central Limit theorem generalizes Theorem 2.3.5 in \cite{6} to the case $\wm\neq 0$.\par 
\begin{theorem}\label{thmB-1}
(Local Central Limit Theorem) Assume \eqref{p-assump-1}, \eqref{p-assump-2} and \eqref{p-assump-3}. There exists a constant $C<\infty$, such that
\be\left|p^t(x)-\nd^t(x)\right|\le \frac{C}{t},\quad \forall x\in\bZ, t\in\bZ^+.\label{nt1-1}\ee  
where 
\be \nd^t(x)=\frac{1}{\sqrt{2\pi\sigma_1^2t}}\exp\left\{-\frac{(x-t\wm)^2}{2t\sigma_1^2}\right\}.\ee
\end{theorem}

As applications to Theorem \ref{thmB-1}, we list two corollaries. The first corollary is an generalization of Theorem 2.3.6 in \cite{6}. The second corollary is stated in Lemma 4.2 in \cite{1}.\par
\begin{corollary}\label{crllyB-0}
Assume \eqref{p-assump-1}, \eqref{p-assump-2} and \eqref{p-assump-3}. Let $\nabla$ denote the differences in the $x$ variable,
$$\nabla p^{t}(x)=p^t(x+1)-p^t(x), \quad \nabla \nd^{t}(x)=\nd^t(x+1)-\nd^t(x).$$
There exists a constant $C<\infty$, such that
\be \left| \nabla p^{t}(x)-\nabla \nd^{t}(x)\right|\le \frac{C}{t^{3/2}}, \quad \forall x\in\bZ, t\in\bZ^+.\ee
\end{corollary}

\begin{corollary}\label{crllyB-1}
For a mean 0, span 1 random walk $S_n$ on $\bZ$ with finite variance $\sigma^2$, $a\in \bR$ and points $a_n\in\bZ$ s.t. $\lim_{n\to\infty}a_n/\sqrt{n}=a$, then
\be
\lim_{n\to\infty}\frac{1}{\sqrt{n}}\sum_{k=0}^{\lfloor nt\rfloor-1}\bP(S_k=a_n)=\frac{1}{\sigma^2}\int_0^{\sigma^2t}\frac{1}{\sqrt{2\pi v}}\exp\bigl\{-\frac{a^2}{2v}\bigr\}dv.\label{2-l11}\ee
\end{corollary}

\begin{proof}[Proof of Theorem \ref{thmB-1}]

The characteristic function of the transition $p$ is $\phi(\theta)=\sum_{k\in\bZ}p(k)e^{\i k\theta}$. Since $p$ has finite range, $\phi(\theta)$ is the sum of finitely many exponential functions and hence, analytic. We can also define a ``normalized'' characteristic function $\tilde{\phi}(\theta)=\phi(\theta)e^{-\i\theta \wm}=\sum_{k\in\bZ}p(k)e^{\i (k-\wm)\theta}$, which is also analytic. \par

Let us first provide a lemma which gives bounds for the characteristic function $\tilde{\phi}(\theta)$.
\begin{lemma}\label{lmmB-1}
Under the assumptions in Theorem \ref{thmB-1},
\begin{enumerate}[(a)]
\item For every fixed $\epsilon>0$, 
\be\sup\{|\tilde{\phi}(\theta)|:\theta\in [-\pi,\pi], |\theta|\ge \epsilon\}<1.\label{nl1-1}\ee
\item There is a constant $b>0$ such that
\be |\tilde{\phi}(\theta)|\le 1-b\theta^2,\quad \forall \theta\in[-\pi,\pi].\label{nl1-2}\ee
In particular, for $r>0$,
\be |\tilde{\phi}(\theta)|^r\le [1-b\theta^2]^r\le \exp\{-br\theta^2\},\quad \forall \theta\in[-\pi,\pi].\label{nl1-3}\ee 
\end{enumerate}
\end{lemma}
\begin{proof}

For part (a), since $|\tilde{\phi}(\theta)|=|\phi(\theta)|$, and by continuity and compactness, we only need to show that $|\phi(\theta)|<1$ for any $\theta\in [-\pi,\pi]\setminus \{0\}$, which is proved by Theorem 3.5.1 in \cite{2}.\par 

For part (b),  note that $\tilde{\phi}(\theta)$ has the following Taylor expansion:
\be\tilde{\phi}(\theta)=1-\frac{\sigma_1^2\theta^2}{2}+\tilde{h}(\theta),\label{nl1-4}\ee
where $\tilde{h}(\theta)=O(|\theta|^3)$ as $\theta\to 0$. \par

We can pick an $\epsilon_1>0$ such that for all $|\theta|<\epsilon_1$, $|\tilde{h}(\theta)|\le \frac{\sigma_1^2\theta^2}{4}$. Let $\epsilon_0=\epsilon_1\wedge \frac{\sqrt{2}}{\sigma_1}$. Then, for $\forall |\theta|<\epsilon_0$, 
$$|\tilde{\phi}(\theta)|\le |1-\frac{\sigma_1^2\theta^2}{2}|+|\tilde{h}(\theta)|\le 1-\frac{\sigma_1^2\theta^2}{2}+\frac{\sigma_1^2\theta^2}{4}=1-\frac{\sigma_1^2\theta^2}{4}.$$

For every $\pi\ge |\theta|\ge \epsilon_0$,  by the result in part (a), let $1-a=\sup\{|\tilde{\phi}(\theta)|:\theta\in [-\pi,\pi], |\theta|\ge \epsilon_0\}<1$. Then $0<a\le 1$ and
$$|\tilde{\phi}(\theta)|\le 1-a \le 1-\frac{a}{\pi^2}\theta^2.$$

Let $b=\frac{\sigma_1^2}{4}\wedge \frac{a}{\pi^2}$ and we are done.
\end{proof}

Next, let us give an approximation to $\left[\tilde{\phi}\left(\theta/\sqrt{t}\right)\right]^t$.

\begin{lemma}\label{lmmB-3}
Assume the assumptions in Theorem \ref{thmB-1}, there exist $\epsilon>0$ and $c<\infty$ such that for all positive integers $t$ and all $|\theta|<\epsilon\sqrt{t}$, 
\begin{enumerate} [(a)]
\item We define $\tilde{g}(\theta,t)$ and $\tilde{F}_t(\theta)$ as following:
\be\left[\tilde{\phi}\left(\frac{\theta}{\sqrt{t}}\right)\right]^t=\exp\left\{-\frac{\sigma_1^2\theta^2}{2}+\tilde{g}(\theta,t)\right\}=[1+\tilde{F}_t(\theta)]\exp\left\{-\frac{\sigma_1^2\theta^2}{2}\right\}.\label{nl2-1}\ee
\item 
\be|\tilde{g}(\theta,t)|\le \left(\frac{\sigma_1^2\theta^2}{4}\right)\wedge\left(\frac{c|\theta|^3}{t^{1/2}}\right).\label{nl2-3}\ee
\item 
\be|\tilde{F}_t(\theta)|\le \exp\left\{\frac{\sigma_1^2\theta^2}{4}\right\}+1.\label{nl2-4}\ee
\end{enumerate}
\end{lemma}
\begin{proof}

For part (a),  by the continuity of $\tilde{\phi}$, there exists $\delta>0$ such that $|\tilde{\phi}(\theta)-1|\le \frac{1}{2}$, for all $|\theta|\le \delta$. And thus, $\tilde{g}(\theta,t)$ and $\tilde{F}_t(\theta)$ are well-defined if the inequality $|\theta|<\delta\sqrt{t}$ holds.\par

For part (b),  from \eqref{nl1-4} and Taylor expansion, we get
\begin{align*}
\log \tilde{\phi}(\theta)=&\log \left(1-[\frac{1}{2}\sigma_1^2\theta^2-\tilde{h}(\theta)]\right)\\
=&-\frac{1}{2}\sigma_1^2\theta^2+\tilde{h}(\theta)-\frac{1}{2}[\frac{1}{2}\sigma_1^2\theta^2-\tilde{h}(\theta)]^2+O(|\theta|^6)\\
=&-\frac{1}{2}\sigma_1^2\theta^2+\tilde{h}(\theta)-\frac{1}{8}\sigma_1^4\theta^4+O(|\theta|^5).
\end{align*}

Hence, 
$$t\log \tilde{\phi}\left(\frac{\theta}{\sqrt{t}}\right)=-\frac{1}{2}\sigma_1^2\theta^2+t\cdot\tilde{h}\left(\frac{\theta}{\sqrt{t}}\right)-\frac{\sigma_1^4\theta^4}{8t}+t\cdot O\left(\frac{|\theta|^5}{t^{5/2}}\right).$$

Compare it with $t\log \tilde{\phi}\left(\frac{\theta}{\sqrt{t}}\right)=-\frac{\sigma_1^2\theta^2}{2}+\tilde{g}(\theta,t)$, we can get an estimate for $\tilde{g}(\theta,t)$,
$$\tilde{g}(\theta,t)=t\cdot\tilde{h}\left(\frac{\theta}{\sqrt{t}}\right)-\frac{\sigma_1^4\theta^4}{8t}+t\cdot O\left(\frac{|\theta|^5}{t^{5/2}}\right)=t\cdot\tilde{h}\left(\frac{\theta}{\sqrt{t}}\right)+t\cdot O\left(\frac{|\theta|^4}{t^2}\right)=t\cdot o\left(\frac{|\theta|^2}{t}\right).$$

Thus, there exists $0<\epsilon_1<\delta$, such that for all $t>0$ and all $|\theta|<\epsilon_1\sqrt{t}$,
\be |\tilde{g}(\theta,t)|\le \frac{1}{4}\sigma_1^2 t\cdot \frac{\theta^2}{t}=\frac{1}{4}\sigma_1^2\theta^2.\label{nl2-5}\ee

Moreover, since $\tilde{h}\left(\frac{\theta}{\sqrt{t}}\right)=O\left(\frac{|\theta|^3}{t^{3/2}}\right)$, $\tilde{g}(\theta,t)=t\cdot O\left(\frac{|\theta|^3}{t^{3/2}}\right)$. We can find $0<\epsilon_2<\delta$ and $0<c<\infty$, such that for all positive integer $t$ and all $|\theta|<\epsilon_2\sqrt{t}$,
\be |\tilde{g}(\theta,t)|\le ct\cdot \frac{|\theta|^3}{t^{3/2}}=\frac{c|\theta|^3}{t^{1/2}}.\label{nl2-6}\ee

Take $\epsilon=\epsilon_1\wedge\epsilon_2$, \eqref{nl2-3} is achieved.\par

Part (c) is a straightforward result from part (b) by using the fact that $|e^z|\le e^{|z|}$.
\end{proof}

The next lemma studies the error term of the normal approximation for the multi-step transition probability $p^t(x)$.\par

\begin{lemma}\label{lmmB-4}
Assume the assumptions in Theorem \ref{thmB-1}. Let us define $b_t(x,r)$ by the equation
\be p^t(x)=\nd^t(x)+b_t(x,r)+\frac{1}{2\pi \sqrt{t}}\int_{|s|\le r}e^{-\frac{\i xs}{\sqrt{t}}}e^{\i\sqrt{t}\wm s}e^{-\frac{\sigma_1^2s^2}{2}}\tilde{F}_t(s)ds,\label{nl3-1}\ee
where $\nd^t(x)=\frac{1}{\sqrt{2\pi\sigma_1^2t}}\exp\left\{-\frac{(x-t\wm)^2}{2t\sigma_1^2}\right\}$. Then, there exist $\epsilon>0$, $0<c<\infty$ and $\zeta>0$ such that 
\be |b_t(x,r)|\le c t^{-1/2}e^{-\zeta r^2},\quad \forall 0\le r \le \epsilon\sqrt{t}.\label{nl3-2}\ee
\end{lemma}

\begin{proof}
We set $\epsilon$ to be the same as in Lemma \ref{lmmB-3}.\par

By the inversion formula, 
\begin{align}
p^t(x)=&\frac{1}{2\pi}\int_{-\pi}^\pi [\phi(\theta)]^te^{-\i x\theta}d\theta \overset{s=\sqrt{t}\theta}{=}\frac{1}{2\pi\sqrt{t}}\int_{-\sqrt{t}\pi}^{\sqrt{t}\pi} \left[\phi\left(\frac{s}{\sqrt{t}}\right)\right]^te^{-\frac{\i x}{\sqrt{t}}s}ds\notag\\
=&\frac{1}{2\pi\sqrt{t}}\int_{-\sqrt{t}\pi}^{\sqrt{t}\pi} \left[\tilde{\phi}\left(\frac{s}{\sqrt{t}}\right)\right]^te^{\i\sqrt{t}\wm s}e^{-\frac{\i x}{\sqrt{t}}s}ds.\label{nl3-3}
\end{align}

From \eqref{nl1-1}, there exists $\beta_1>0$ such that  
$$|\tilde{\phi}(\theta)|\le e^{-\beta_1},\quad \pi\ge |\theta|\ge \epsilon.$$

Let us split the integral \eqref{nl3-3} into two parts:
\begin{align*}
p^t(x)=&\frac{1}{2\pi\sqrt{t}}\int_{\epsilon\sqrt{t}\le |s|\le \pi\sqrt{t}} \left[\tilde{\phi}\left(\frac{s}{\sqrt{t}}\right)\right]^te^{\i\sqrt{t}\wm s}e^{-\frac{\i x}{\sqrt{t}}s}ds\\
&\quad\quad+\frac{1}{2\pi\sqrt{t}}\int_{|s|<\epsilon\sqrt{t} } \left[\tilde{\phi}\left(\frac{s}{\sqrt{t}}\right)\right]^te^{\i\sqrt{t}\wm s}e^{-\frac{\i x}{\sqrt{t}}s}ds.
\end{align*}

Note that the first integral in the above equation has the following bound.
\begin{align*}
&\left|\frac{1}{2\pi\sqrt{t}}\int_{\epsilon\sqrt{t}\le |s|\le \pi\sqrt{t}} \left[\tilde{\phi}\left(\frac{s}{\sqrt{t}}\right)\right]^te^{\i\sqrt{t}\wm s}e^{-\frac{\i x}{\sqrt{t}}s}ds\right|\le\frac{1}{2\pi\sqrt{t}}\int_{\epsilon\sqrt{t}\le |s|\le \pi\sqrt{t}}\left|\tilde{\phi}\left(\frac{s}{\sqrt{t}}\right)\right|^tds\\
&\le \frac{1}{2\pi\sqrt{t}}\int_{\epsilon\sqrt{t}\le |s|\le \pi\sqrt{t}}\left|\tilde{\phi}\left(\frac{s}{\sqrt{t}}\right)\right|^tds \le \frac{1}{2\pi\sqrt{t}}\int_{\epsilon\sqrt{t}\le |s|\le \pi\sqrt{t}}e^{-\beta_1 t}ds= \frac{\pi-\epsilon}{\pi}e^{-\beta_1 t}.
\end{align*}

By using the inversion formula, we can similarly split $\nd^t(x)$ into two parts:
\begin{align*}
\nd^t(x)=&\frac{1}{2\pi}\int_{-\infty}^{+\infty}e^{-\i x\theta}e^{\i t\wm \theta-\frac{1}{2}t\sigma_1^2\theta^2}d\theta\overset{s=\sqrt{t}\theta}{=}\frac{1}{2\pi\sqrt{t}}\int_{-\infty}^{+\infty}e^{-\frac{\i xs}{\sqrt{t}}}e^{\i\sqrt{t}\wm s-\frac{1}{2}\sigma_1^2s^2}ds\\
=&\frac{1}{2\pi\sqrt{t}}\int_{|s|\ge \epsilon\sqrt{t}}e^{-\frac{\i xs}{\sqrt{t}}}e^{\i\sqrt{t}\wm s-\frac{1}{2}\sigma_1^2s^2}ds+\frac{1}{2\pi\sqrt{t}}\int_{|s|< \epsilon\sqrt{t}}e^{-\frac{\i xs}{\sqrt{t}}}e^{\i\sqrt{t}\wm s-\frac{1}{2}\sigma_1^2s^2}ds.
\end{align*}

Again, same as before, the first term goes to zero exponentially fast.
\begin{align}
\left|\frac{1}{2\pi\sqrt{t}}\int_{|s|\ge \epsilon\sqrt{t}}e^{-\frac{\i xs}{\sqrt{t}}}e^{\i\sqrt{t}\wm s-\frac{1}{2}\sigma_1^2s^2}ds\right|\le& \frac{1}{2\pi\sqrt{t}}\int_{|s|\ge \epsilon\sqrt{t}}e^{-\frac{1}{2}\sigma_1^2s^2}ds\overset{s=\sqrt{t}\theta}{=}\frac{1}{2\pi}\int_{|\theta|\ge\epsilon}e^{-\frac{1}{2}\sigma_1^2t\theta^2}d\theta\notag\\
=&\frac{1}{\pi}\int_\epsilon^{+\infty}e^{-\frac{1}{2}\sigma_1^2t\theta^2}d\theta\le\frac{1}{\pi}\int_\epsilon^{+\infty}\frac{\theta}{\epsilon}\cdot e^{-\frac{1}{2}\sigma_1^2t\theta^2}d\theta\notag\\
=&\frac{1}{\pi\epsilon\sigma_1^2t}e^{-\frac{1}{2}\sigma_1^2\epsilon^2t}.\label{nl3-4}
\end{align}

Thus, we can pick $\beta_2>0$ such that 
$$\frac{1}{2\pi\sqrt{t}}\int_{|s|\ge \epsilon\sqrt{t}}e^{-\frac{\i xs}{\sqrt{t}}}e^{\i\sqrt{t}\wm s-\frac{1}{2}\sigma_1^2s^2}ds=O\left(e^{-\beta_2t}\right).$$

Let $\beta=\beta_1\wedge\beta_2$. Then 
\begin{align*}
p^t(x)-\nd^t(x)=&O\left(e^{-\beta t}\right)+\frac{1}{2\pi\sqrt{t}}\int_{|s|<\epsilon\sqrt{t} } \left\{\left[\tilde{\phi}\left(\frac{s}{\sqrt{t}}\right)\right]^t-e^{-\frac{1}{2}\sigma_1^2s^2}\right\}e^{\i\sqrt{t}\wm s}e^{-\frac{\i x}{\sqrt{t}}s}ds\\
=&O\left(e^{-\beta t}\right)+\frac{1}{2\pi\sqrt{t}}\int_{|s|<\epsilon\sqrt{t}} \tilde{F}_t(s)e^{-\frac{1}{2}\sigma_1^2s^2}e^{\i\sqrt{t}\wm s}e^{-\frac{\i x}{\sqrt{t}}s}ds.
\end{align*}

By simply choosing $\zeta$ to be strictly less than $\beta/\epsilon^2$, we prove the result for $r=\epsilon\sqrt{t}$. For $0\le r < \epsilon\sqrt{t}$, we use the estimate \eqref{nl2-4}.
\begin{align}
&\left|\int_{r<|s|<\epsilon\sqrt{t}} \tilde{F}_t(s)e^{-\frac{1}{2}\sigma_1^2s^2}e^{\i\sqrt{t}\wm s}e^{-\frac{\i x}{\sqrt{t}}s}ds\right|\le\int_{r<|s|<\epsilon\sqrt{t}} |\tilde{F}_t(s)|e^{-\frac{1}{2}\sigma_1^2s^2}ds\notag\\
&\le\int_{r<|s|<\epsilon\sqrt{t}} \left[e^{\frac{\sigma_1^2s^2}{4}}+1\right] e^{-\frac{1}{2}\sigma_1^2s^2}ds\le 2\int_r^{+\infty} e^{-\frac{\sigma_1^2s^2}{4}}ds.\label{nl3-5}
\end{align}

In order to get a bound better than the one in \eqref{nl3-4}, we use the following inequality (see Formula 7.1.13 in \cite{8}).
\begin{lemma}\label{lmmB-5}
For any $r\ge 0$, we have
\be \frac{1}{r+\sqrt{r^2+2}}<e^{r^2}\int_{r}^\infty e^{-s^2}ds\le\frac{1}{r+\sqrt{r^2+\frac{4}{\pi}}}.\label{need-to-be-decided}\ee
\end{lemma}

We get 
$$\left|\int_{r<|s|<\epsilon\sqrt{t}} \tilde{F}_t(s)e^{-\frac{1}{2}\sigma_1^2s^2}e^{\i\sqrt{t}\wm s}e^{-\frac{\i x}{\sqrt{t}}s}ds\right|\le \frac{4e^{-\frac{\sigma_1^2r^2}{4}}}{\sigma_1^2r/2+\sigma_1\sqrt{(\sigma_1r/2)^2+4/\pi}}\le \frac{2\sqrt{\pi}}{\sigma_1}e^{-\frac{\sigma_1^2r^2}{4}}.$$

Hence, 
\begin{align*}
p^t(x)-\nd^t(x)=&O\left(e^{-\beta t}\right)+\frac{1}{2\pi\sqrt{t}}\int_{r<|s|<\epsilon\sqrt{t}} e^{-\frac{\i xs}{\sqrt{t}}}e^{\i\sqrt{t}\wm s}e^{-\frac{\sigma_1^2s^2}{2}}\tilde{F}_t(s)ds\\
&\quad\quad+\frac{1}{2\pi \sqrt{t}}\int_{|s|\le r}e^{-\frac{\i xs}{\sqrt{t}}}e^{\i\sqrt{t}\wm s}e^{-\frac{\sigma_1^2s^2}{2}}\tilde{F}_t(s)ds\\
=&O\left(e^{-\beta t}\right)+O\left(t^{-1/2}e^{-\frac{\sigma_1^2r^2}{4}}\right)+\frac{1}{2\pi \sqrt{t}}\int_{|s|\le r}e^{-\frac{\i xs}{\sqrt{t}}}e^{\i\sqrt{t}\wm s}e^{-\frac{\sigma_1^2s^2}{2}}\tilde{F}_t(s)ds.
\end{align*}

We can choose $\zeta$ to be strictly less than $\frac{\beta}{\epsilon^2}\wedge \frac{\sigma_1^2}{4}$. And hence, the proof for Lemma \ref{lmmB-4} is complete. 
\end{proof}

Now let us prove the main theorem. Let us take $r=t^{1/12}$ in Lemma \ref{lmmB-4}, we have
\be p^t(x)=\nd^t(x)+O\left(t^{-1/2}e^{-\zeta t^{1/6}}\right)+\frac{1}{2\pi \sqrt{t}}\int_{|s|\le t^{1/12}}e^{-\frac{\i xs}{\sqrt{t}}}e^{\i\sqrt{t}\wm s}e^{-\frac{\sigma_1^2s^2}{2}}\tilde{F}_t(s)ds.\label{lclt-approx}\ee

Notice that from \eqref{nl2-3},
$$\left|\tilde{F}_t(\theta)\right|=\left| e^{\tilde{g}(\theta,t)}-1\right|\le C\left|\tilde{g}(\theta,t)\right|\le \frac{C|\theta|^3}{t^{1/2}},\quad \abs{\theta}\le t^{1/12}.$$
where the first inequality is because $e^{x}-1=O(x)$ for all $x$ in a bounded set.\par

Thus,
\begin{align*}
&\left|\frac{1}{2\pi \sqrt{t}}\int_{|s|\le t^{1/12}}e^{-\frac{\i xs}{\sqrt{t}}}e^{\i\sqrt{t}\wm s}e^{-\frac{\sigma_1^2s^2}{2}}\tilde{F}_t(s)ds\right|\le  \frac{C}{2\pi \sqrt{t}}\int_{|s|\le t^{1/12}}\frac{|s|^3}{t^{1/2}}e^{-\frac{\sigma_1^2s^2}{2}}ds\\
&\le \frac{C}{2\pi t}\int_{s\in\bR}|s|^3e^{-\frac{\sigma_1^2s^2}{2}}ds = O(t^{-1}).
\end{align*}

Therefore, 
$$p^t(x)-\nd^t(x)=O\left(t^{-1}\right).$$
And the proof of Theorem \ref{thmB-1} is complete.
\end{proof}

\begin{proof}[Proof of Corollary \ref{crllyB-0}]
From \eqref{lclt-approx} in the proof of Theorem \ref{thmB-1}, 
\begin{align}
\nabla p^t(x)=&\nabla \nd^t(x)+O\left(t^{-1/2}e^{-\zeta t^{1/6}}\right)\notag\\
                    &\quad\quad+\frac{1}{2\pi \sqrt{t}}\int_{|s|\le t^{1/12}}\left(e^{-\i (x+1)s/\sqrt{t}}-e^{-\i xs/\sqrt{t}}\right)e^{\i\sqrt{t}\wm s}e^{-\frac{\sigma_1^2s^2}{2}}\tilde{F}_t(s)ds.\label{lclt-diff-approx}
\end{align}
Notice that for $|s|\le t^{1/12}$,
$$\left|e^{-\i(x+1)s/\sqrt{t}}-e^{-\i xs/\sqrt{t}}\right|=\left|e^{-\i s/\sqrt{t}}-1\right|\le \frac{\abs{s}}{\sqrt{t}}.$$

Then,
\begin{align*}
&\left|\frac{1}{2\pi \sqrt{t}}\int_{|s|\le t^{1/12}}\left(e^{-\i (x+1)s/\sqrt{t}}-e^{-\i xs/\sqrt{t}}\right)e^{\i\sqrt{t}\wm s}e^{-\frac{\sigma_1^2s^2}{2}}\tilde{F}_t(s)ds\right|\\
&\le \frac{C}{2\pi \sqrt{t}}\int_{|s|\le t^{1/12}}\frac{|s|^4}{t}e^{-\frac{\sigma_1^2s^2}{2}}ds\le \frac{C}{2\pi t^{3/2}}\int_{s\in\bR}|s|^4e^{-\frac{\sigma_1^2s^2}{2}}ds=O(t^{-3/2}).
\end{align*}

Hence,
$$\nabla p^t(x)=\nabla \nd^t(x) + O(t^{-3/2}).$$
This proves Corollary \ref{crllyB-0}.
\end{proof}

\begin {proof}[Proof of Corollary \ref{crllyB-1}]
First, One can use LCLT to show that 
\be \lim_{m\to\infty}\sup_{x\in\bZ}\sqrt{m}\left|\bP(S_m=x)-\frac{1}{\sqrt{2\pi m\sigma^2}}\exp\left\{-\frac{x^2}{2m\sigma^2}\right\}\right|=0.\label{locapp1}\ee

Then, 
\begin{align}
&\frac{1}{\sqrt{n}}\sum_{k=0}^{\lfloor nt\rfloor-1}\bP(S_k=a_n) = \frac{1}{\sqrt{n}}\sum_{k=1}^{\lfloor nt\rfloor-1}\left[\frac{1}{\sqrt{2\pi k\sigma^2}}\exp\left\{-\frac{a_n^2}{2k\sigma^2}\right\}+o(k^{-1/2})\right]\notag\\
&=\frac{1}{n}\sum_{k=1}^{\lfloor nt\rfloor-1}\frac{1}{\sqrt{2\pi (k/n)\sigma^2}}\exp\left\{-\frac{(a_n/\sqrt{n})^2}{2(k/n)\sigma^2}\right\}+o(1)\notag\\
&=\int_0^t\frac{\ind\{u\le (\lfloor nt\rfloor -1)/n\}}{\sqrt{2\pi (\lceil nu\rceil /n)\sigma^2}}\exp\left\{-\frac{(a_n/\sqrt{n})^2}{2(\lceil nu\rceil /n)\sigma^2}\right\}du+o(1).\label{locapp2}
\end{align}
Notice that the integrand in \eqref{locapp2} is bounded by $\frac{1}{\sqrt{2\pi u\sigma^2}}$.  By Dominated Convergence Theorem, we have
$$\lim_{n\to\infty}\frac{1}{\sqrt{n}}\sum_{k=0}^{\lfloor nt\rfloor-1}\bP(S_k=a_n)=\int_0^t\frac{1}{\sqrt{2\pi u\sigma^2}}\exp\left\{-\frac{a^2}{2u\sigma^2}\right\}du.$$
The proof is complete by substitution $v=u\sigma^2$.
\end {proof}

\bibliographystyle{chicago}
\bibliography{thesis_zhai}

\end{document}